\documentclass[11pt]{amsart}
\usepackage{amssymb}
\usepackage{amsmath}
\usepackage{amsfonts}
\usepackage{mathrsfs}
\usepackage{psfrag}
\usepackage{verbatim}
\usepackage{color}
\usepackage{enumitem}
\usepackage{bm}
\usepackage[usenames,dvipsnames]{xcolor}
\usepackage{dsfont}
\usepackage{overpic}
\usepackage{subcaption}
\usepackage{multirow}
\numberwithin{equation}{section} 

\usepackage[colorlinks=true]{hyperref}
\hypersetup{urlcolor=blue, citecolor=ForestGreen}

\DeclareGraphicsExtensions{}
\theoremstyle{plain}
\newtheorem{theorem}{Theorem}[section]

\newtheorem{lemma}[theorem]{Lemma}
\newtheorem{proposition}[theorem]{Proposition}

\theoremstyle{definition}
\newtheorem{assumption}[theorem]{Assumption}

\theoremstyle{remark}
\newtheorem{remark}[theorem]{Remark}

\newcommand{\eps}{\varepsilon}

\newcommand{\BR}{\mathbb{R}}

\newcommand{\BE}{\mathbb{E}}

\newcommand{\BN}{\mathbb{N}}

\newcommand{\ind}{\mathds{1}}

\begin{document}
\title[{\sc uit}-{\sc qclt} for multiscale diffusions]{Uniform-in-time Gaussian fluctuations for multiscale nonlinear stochastic systems via Malliavin Calculus
}

\author[Shivam Singh Dhama]{}
\address{Department of Mathematics and Statistics\\ Boston University, Boston, {\sc ma}, 02215, {\sc usa}}
\email{ssdhama@bu.edu, shivamsd.maths@gmail.com}

 \subjclass[2020]{60F17, 60H07, 60F05.}
 \keywords{Multiscale diffusions, Quantitative {\sc clt},  Malliavin Calculus, Poisson Equations}

\date{\today.  
}

\maketitle

\centerline{\scshape Shivam Singh Dhama}
\medskip
{\footnotesize
\centerline{Department of Mathematics and Statistics, Boston University, Boston, Massachusetts, 02215, {\sc usa}}
}


\begin{center}
\rule{17cm}{.01mm}
\end{center}
\begin{abstract}
We establish a uniform-in-time quantitative central limit theorem ({\sc qclt}) for a nonlinear slow-fast stochastic system. We identify significant weaker sufficient conditions that enable us to obtain time-independent bounds for the Wasserstein distance between the fluctuation process and a centered Gaussian random variable. To prove our main result, we utilize tools from Malliavin calculus, specifically the second-order Poincar\'e inequality. In this context, applying the Poincar\'e inequality requires demonstrating uniform bounds over time for both the first- and second-order Malliavin derivatives.
\end{abstract}
\begin{center}
\rule{17cm}{.01mm}
\end{center}


\section{Introduction}
\subsection{Problem Statement, Background and a Relevant Literature} Slow–fast systems arise naturally in the modeling of many complex phenomena, for example, in physics, biology, finance, climate science, and engineering. These systems contain variables that evolve on different time scales. Some components change rapidly, while others evolve much more slowly. The presence of multiple time scales makes the analysis and simulation of such systems challenging. Therefore, it is often desirable to derive simpler effective models that capture the essential behavior of the original system. Such reduced models are typically more amenable to both theoretical analysis and numerical computation. Averaging and homogenization theories provide powerful tools for obtaining these effective descriptions. They characterize the limiting behavior of the slow variables. Since the seminal work of Khasminskii \cite{Khasminskii68} over the past decades, averaging and homogenization effects for multiscale systems have been extensively studied through various asymptotic results, including Law of Large Numbers ({\sc lln}), Central Limit Theorem ({\sc clt}), and Large Deviation Principle ({\sc ldp}).

In this paper, we consider the slow-fast stochastic system
\begin{equation}\label{E:Multiscale-Diffusion-Main-Eq}
\begin{aligned}
dX_t^{\varepsilon} &= c(X_t^\varepsilon, Y_t^\eta) \thinspace dt+ \sqrt{\varepsilon}\sigma(Y_t^\eta) \thinspace dW_t^1, \quad X_0^\eps=x_0  \\
dY_t^{\eta} &= \frac{1}{\eta} f(Y_t^\eta) \thinspace dt+ \frac{1}{\sqrt{\eta}}\tau( Y_t^\eta) \thinspace dW_t^2, \quad Y_0^\eta = y_0,
\end{aligned}
\end{equation}
where, $0<\eps ,\eta \ll 1,$ and the processes $W^1$ and $W^2$ are two independent one-dimensional Brownian motions defined on a common probability space. The functions $c: \BR \times \BR \to \BR,$ $\sigma : \BR \to \BR$, $f: \BR \to \BR$ and $\tau : \BR \to \BR$ satisfy Assumptions \ref{A:Boundedness of sigma-tau} through \ref{A:eta-eps-interaction}. The study of system \eqref{E:Multiscale-Diffusion-Main-Eq} is important as it appears in certain climate models with scaling $\sqrt{\eta} \thinspace t$, see \cite{kifer2001averaging}. For system \eqref{E:Multiscale-Diffusion-Main-Eq}, we define the fluctuation process 
$\theta_t^\eps \triangleq \eps^{-\frac{1}{2}} \left({X_t^\eps- \bar{X}_t}\right),$
where, the averaged function $\bar{X}_t$ satisfies equation \eqref{E:Semigroup-Equatn}. In the article \cite{spiliopoulos2014fluctuation}, the author investigated a qualitative {\sc clt}, i.e., the asymptotic behavior of the process $\theta^\eps$ for a fully coupled slow--fast version of system \eqref{E:Multiscale-Diffusion-Main-Eq} over a finite time horizon \([0, T]\) for fixed \( T < \infty\). More recently, the authors of \cite{bour_spilio_2025} established a Quantitative Central Limit Theorem ({\sc qclt}) for a fully coupled version of \eqref{E:Multiscale-Diffusion-Main-Eq} on a finite time interval, employing techniques from Malliavin calculus. Building on the results from \cite{bour_spilio_2025}, a natural question arises: \textit{Can we obtain a uniform-in-time {\sc qclt} for system \eqref{E:Multiscale-Diffusion-Main-Eq}?} Answering this question is the primary goal of the present paper.

We aim to establish a {\sc qclt} for the system \eqref{E:Multiscale-Diffusion-Main-Eq} in a uniform-in-time ({\sc uit}) setting. More specifically, we aim to provide an explicit rate of convergence for the fluctuation process $\theta^\eps$ to a Gaussian process, expresses in terms of the small parameters $ \eps$ and $\eta$, with this rate being independent of time. To derive this explicit rate, we employ the tools from Malliavin calculus and partial differential equations ({\sc pde}). In particular,
from Malliavin calculus, we utilize the Poincar{\'e} inequality \cite{vido_2020}, which enables us to quantify the rates based on the first- and second-order Malliavin derivatives of the pre-limit process. In the study of {\sc pde}, we use a powerful tool to control various fluctuation terms by constructing several Poisson equations in non-compact spaces. At certain stages of the proof, we employ techniques similar to those presented in \cite{bour_spilio_2025} to obtain quantitative rates of convergence. The key assumptions in this paper are the dissipativity of the coefficients and the ergodicity of the fast process, namely, the existence of a unique invariant measure.

The present work is novel for several reasons. First, we establish a {\sc uit} explicit convergence rate for the fluctuations of a multiscale system, rendering our result quantitative rather than merely qualitative. Second, we provide significantly weaker sufficient conditions for the nonlinear functions $c(x,y)$ and $f(y)$ that yield time-independent bounds. The most technical aspect of the paper concerns the estimation of convergence rates for Malliavin derivatives and the construction of several Poisson equations on the entire space. Obtaining these estimates requires a careful decomposition and systematic treatment of numerous terms in order to derive sufficiently sharp bounds. Another delicate aspect of the analysis is the derivation of quantitative convergence rates for several fluctuation terms, particularly in the computation of prelimit expectations and variances (Section \ref{S:Prelimit-Exp-Var}). This difficulty is addressed by constructing appropriate Poisson equations and utilizing the fact that the solutions to these equations can grow at most polynomially, followed by the uniform moment bounds for the involved processes.

We provide some heuristic comments to clarify the methodology of our proof. Since the main objective of this paper is to obtain the {\sc uit} estimates, therefore, we expect to see exponential decay in the time argument of the bounds for the Malliavin derivatives, as these derivatives need to be integrated. On another note, since our final estimate relies on separate computations of the first- and second-order Malliavin derivatives, it is crucial for the bounds of the second-order derivatives to converge to zero faster than those of the first-order derivatives in the small parameters $\eps$ and $\eta$. This intuition can be illustrated with a simple example of one-dimensional Brownian motion, where the first-order Malliavin derivative is a non-zero constant and the second derivative equals zero. Indeed, this situation is reflected in our estimates. Further, due to the ergodicity of the fast process, the fluctuation terms are consistently small; however, they are estimated using several Poisson equations.

We now turn to the literature most relevant to the present work. There has been extensive research on the asymptotic analysis of multiscale stochastic systems, 
particularly in the context of averaging and homogenization for stochastic differential equations and stochastic partial differential equations ({\sc spde}s) driven by 
various types of noise. Notable contributions include, without claiming completeness, \cite{Khasminskii68, freidlin1999comparison, kifer2001averaging, pavliotis2008multiscale, spiliopoulos2014fluctuation, gailus2017statistical, rockner2021averaging, rockner2021diffusion} and the references therein for It\^o stochastic systems driven by continuous noise under different scaling regimes, \cite{freidlin2021averaging, goddard2023study} for slow--fast systems with multiple invariant measures, and \cite{cerrai2009averaging, huang2020central, nualart2022central, gerolla2025fluctuations} for the asymptotic behavior of {\sc spde}s. For multiscale systems driven by fractional Brownian motion and jump noise, we refer to, for example, \cite{hairer2020averaging, bourguin2021typical, budhiraja2016moderate} and the references therein. Most of the aforementioned works investigate asymptotic behavior over finite time horizons, typically on compact intervals \([0,T]\) for a fixed \(T>0\). More recently, increasing attention has been devoted to the study of multiscale stochastic systems over infinite time horizons, namely on \([0,\infty)\). For recent developments in this direction, we refer the interested reader to, e.g., \cite{Dobson-Ceisan-21, Dobson-Crisan, budhiraja2024large, Shivam-Kostas, sharrock2026efficient, bourguin2026uniform}, which address a variety of stochastic systems and asymptotic regimes.

The present work is closely related to the recent articles, e.g., \cite{bour_spilio_2025, Dobson-Crisan} and \cite{bourguin2026uniform}. First, we compare our findings with those in \cite{bour_spilio_2025}; the authors study quantitative fluctuations over finite time horizons, whereas we investigate quantitative fluctuations over infinite time horizons. On the other hand, the paper \cite{bour_spilio_2025} considers a fully-coupled system (where, the coefficients depends on both slow $X^\eps$ and fast $Y^\eta$ processes); however, in this work, we focus on a comparatively smaller class of nonlinear systems. Regarding \cite{Dobson-Crisan}, the authors examine the quantitative asymptotic behavior of a fully-coupled system under a uniform-in-time setting but with weaker regularity conditions (local Lipschitz continuity with superlinear growth). The main result in this article is demonstrated by utilizing the property that the space derivatives of the associated Markov semigroup decay exponentially over time, rather than using any tools from Malliavin calculus. As a result, our work differs from theirs because we incorporate techniques from Malliavin calculus. More recently, another paper \cite{bourguin2026uniform} analyzed uniform-in-time fluctuations for a different system: the McKean-Vlasov stochastic system for a finite number of particle systems. The authors of that paper utilized the Malliavin derivative computation approach to establish their main result.

The paper is organized as follows. In Section \ref{S:Main-Result}, we present our main result, viz., Theorem \ref{T:Main-Result}, along with its proof. Section \ref{S:First-MD} focuses on obtaining estimates for the first-order Malliavin derivative, while Section \ref{S:Second-MD} is devoted to the bounds related to the second-order Malliavin derivative. The key findings of these subsections are outlined in Propositions \ref{P:First-Der-X} and \ref{P:Second-der-Prop}. Further, in Section \ref{S:Prelimit-Exp-Var}, we compute the bounds associated with the pre-limit expectation and variance as indicated in Propositions \ref{P:Prelimit-1} and \ref{P:Prelimit-2}. Finally, we conclude the manuscript with some useful lemmas, proofs, and preliminaries for the Poisson equation and Malliavin calculus in Section \ref{S:Appendix}.


\subsection{Notations in this paper} We list some of the special notations used throughout this paper. The symbol $\triangleq$ is read ``is defined to equal." We denote the set of all positive integers and real numbers by $\BN$ and $\BR,$ respectively. For a random variable $Y$, which has a normal distribution with parameters $\mu$ and $\xi^2,$ we write $Y \sim \mathscr{N}(\mu, \xi^2).$ For two random variables $U$, $V$ and a Lipschitz function $\mathsf{f}$, 
$$d_W(U,V) \triangleq \sup_{\mathsf{f}:|\mathsf{f}|_{\text{Lip}}\le 1}\left|\BE \left[\mathsf{f}(U)\right] - \BE \left[\mathsf{f}(V)\right]\right|,$$
represents the Wasserstein distance between the laws of random variables $U$ and $V.$ The partial derivatives of the function \(\mathsf{j}:\mathbb{R}^n \to \mathbb{R}\) with respect to the \(i^{\text{th}}\) variable are denoted using subscript notation: \(\mathsf{j}_{x_i} \triangleq \frac{\partial \thinspace \mathsf{j}}{\partial \thinspace x_i}\). This notation is also applicable to higher-order derivatives. For \(0 \leq r \leq t\), the first-order and second-order Malliavin derivatives of the process \(U_t\) are denoted by \(D_r \thinspace U_t\) and \(D_r^2 \thinspace U_t\), respectively. For set $\mathscr{Y} \subseteq \BR$, and $i \in \{0,1,2\}$,  $\mathscr{C}^{2,\alpha}(\BR,\mathscr{Y})$ denotes the collection of functions $c(x,y)$ such that $c(\cdot, y) \in \mathscr{C}^2(\BR)$ for any $y \in \mathscr{Y}$ and $\partial_x^i c(x,\cdot)$ are H\"older continuous in variable $y$ and uniformly in $x$ with exponent $\alpha$. The set $C_b^{\infty} \left(
  \mathbb{R}^n \right)$ represents the collection of all functions $f$ such that $f$ and all of its partial derivatives of all
orders are bounded.

\subsection{Assumptions}
In this section, we present the assumptions regarding the coefficients for system \eqref{E:Multiscale-Diffusion-Main-Eq} under which we establish our main result, namely, Theorem \ref{T:Main-Result}.

\begin{assumption}\label{A:Boundedness of sigma-tau} 
For the functions $ f : \BR \to \BR, \tau : \BR \to \BR$, we assume:
\begin{enumerate}
\item The diffusion coefficient $\tau^2$ is uniformly nondegenerate, that is, there exist two positive constants $\lambda_1$ and  $\lambda_2$ such that $ \lambda_1 \le \tau^2 \le \lambda_2$. \\
\item There exists a positive number $\lambda^*$ such that
$$y \cdot f(y) \le -\lambda^* |y|^2, \quad y \in \BR.$$
In particular, if the function $f$ is differentiable, then $\sup_{y \in \BR}f_y(y) \le -\lambda^*$.
\end{enumerate}
\end{assumption}

\begin{assumption}\label{A:Derivatives-Assumption-c}
For the functions $c : \BR \times \BR \to \BR$ and $\sigma : \BR \to \BR$, we assume:
\begin{enumerate}
\item For $\mathscr{Y} \subseteq \BR$, and $i \in \{0,1,2\}$, the function $c(x,y) \in \mathscr{C}^{2,\gamma^*}(\BR,\mathscr{Y})$, i.e., $c(\cdot, y) \in \mathscr{C}^2(\BR)$ for all $y \in \mathscr{Y}$ and $\partial_x^i c(x,\cdot)$ are H\"older continuous in variable $y$ and uniformly in $x$ with exponent $\gamma^* \in (0,1)$. The function $c(x,y)$ has two bounded derivatives in variable $x$ and  bounded derivatives up to order two in variable $y$. All the derivatives of $\sigma(y)$ up to second-order are uniformly bounded, that is, for any $y_1, y_2 \in \BR$, $j \in \{1,2\},$
\begin{equation*}
\begin{aligned} 
\sup_{i; \thinspace x \in \BR}\left|\partial_x^i c(x,y_1) - \partial_x^i c(x,y_2) \right| & \le K \thinspace |y_1 - y_2|^{\gamma^*}, \quad \text{and} \\
 \sup_{i,j; \thinspace x, y \in \BR} \left\{|\partial^j_x c(x,y)| + |\partial^i_y c(x,y)| + |\partial_y^i \sigma(y)| \right\} & < K.
\end{aligned}
\end{equation*}
\item 
The function $c(x,y)$ is uniformly dissipative in the first argument: for any $x, \thinspace y \in \BR,$ there exists a constant $\alpha>0$ such that
\begin{equation}\label{E:Ass1}
\sup_{z \in \BR}[(x-y) \cdot (c(x,z)-c(y,z))] \le - \alpha \thinspace |x-y|^2.
\end{equation}
Consequently, there exists a constant $\beta>0$ such that
$\sup_{y \in \BR} [x \cdot c(x,y)] \le -\alpha \thinspace |x|^2 + \beta.$  In, particular, if $c(x,y)$ is continuously differentiable with respect to $x$, then $c_x(x,y)< -\alpha$ for any $x,y \in \BR$.
\item The function $c(x,y)$ has linear growth in variable $x$ uniformly in $y$: for any $x \in \BR,$ there exists a constant $K>0$ such that $\sup_{y \in \BR}|c(x,y)| \le K(1+|x|).$
\item The first equation in system \eqref{E:Multiscale-Diffusion-Main-Eq} has a unique strong solution.
\end{enumerate}
\end{assumption}

\begin{assumption}\label{A:Derivatives-Assumption-f}
For the functions $f: \BR \to \BR$, and $\tau: \BR \to \BR$, we assume:
\begin{enumerate}
\item The function $f(y)$ has two derivatives and $\tau(y)$ has two bounded derivatives. For the functions $f$ and $\tau$, all the derivatives are H\"older continuous, with exponent $\zeta \in (0,1)$. 
\item In addition to Assumption \ref{A:Boundedness of sigma-tau}, we also assume: for some $\lambda>0,$ and any $p \in \BN,$
\begin{equation}\label{E:Ass2}
\sup_{y \in \BR} \left[ 2p \thinspace f_y(y) + p(2p-1) [\tau_y(y)]^2 \right] \le - \lambda.
\end{equation}
\end{enumerate}
\end{assumption}

\begin{assumption}\label{A:Semigroup-boundedness}
For $m \in \{1,2 \},$ and the function $\Psi(t)$ which satisfies equation \eqref{E:Semigroup-Equatn}, there exist time-independent positive constants $K$ and $\sf{K}$ such that
$$\sup_{t \ge 0} \left| \int_0^t \Psi(t)^m \Psi(s)^{-m} \thinspace ds \right| \le K, \qquad \sup_{t \ge 0} |\Psi(t)| \le \sf{K}.$$
\end{assumption}

\begin{assumption}\label{A:eta-eps-interaction} 
For $\eta = \eta(\eps),$ we assume that $\eta \searrow 0$ as $\eps$ tends to zero and 
$$ \lim_{\eps \searrow 0} \frac{\eps}{\eta} = \gamma^2 \in (0, \infty].$$
\end{assumption}
Examples of function $c(x,y)$ that satisfy the above assumptions include $c(x,y) = -\alpha x + \frac{1}{1+e^{-y}}, \thinspace \alpha>0$ and $c(x,y) = -x + \frac{e^y - e^{-y}}{e^y + e^{-y}}.$
There are a few remarks regarding the assumptions mentioned above:
\begin{remark}
As a consequence of Assumption \ref{A:Boundedness of sigma-tau}, we can conclude that $\lim_{|y| \to \infty}y \cdot f(y)= -\infty.$ This implies that Assumption \ref{A:Boundedness of sigma-tau} ensures the existence of a unique invariant measure for the process \( Y^\eta \). Therefore, this assumption is related to the concept of ergodicity.
\end{remark}

\begin{remark}
The first part of Assumption \ref{A:Derivatives-Assumption-c} is crucial to study the regularity properties of the solutions of various Poisson partial differential equations which are used to control the fluctuation terms. For more details, we refer to Section \ref{S:Poisson-Equation} (or \cite{pardoux2003poisson}). The boundedness of the derivatives is employed to control the first- and second-order derivatives of the function \( c(x,y) \), which appears solely as coefficients in the integral equations for the first- and second-order Malliavin derivatives. We refer to the relevant equations \ref{E:First-Order-Malliavin-derivative}, \ref{E:Second-der-X-W1-W1}, \ref{E:Second-der-X-W1-W2} and \ref{E:Second-der-X-W2-W2} for more details. We utilize the dissipativity condition outlined in equation \eqref{E:Ass1}  to establish uniform moment bounds for the process \( X_t^\eps - \bar{X}_t \). Specifically, during the proof of these moment bounds, we need to examine the product \( (X_t^\eps - \bar{X}_t) \{c(X_t^\eps, y)-c(\bar{X}_t, y)\} \), which can be effectively controlled using equation \eqref{E:Ass1}. For further details, we refer to Lemmas \ref{L:2nd-power} and \ref{L:4th-power-LLN}, as well as Propositions \ref{P:Prelimit-1} and \ref{P:Prelimit-2} in the section on prelimit expectation and variance. Additionally, the condition \( c_x(x,y) < -\alpha \) allows us to derive {\sc uit} estimates for the first- and second-order Malliavin derivatives of the slow process \( X^\eps \).
\end{remark}

\begin{remark}
Assumption \ref{A:Derivatives-Assumption-f} plays a crucial role. Part (1) along with Assumption \ref{A:Boundedness of sigma-tau} ensures that the second equation, corresponding to the fast process \( Y^\eta \), in system \eqref{E:Multiscale-Diffusion-Main-Eq} is well-posed. Furthermore, the assumption outlined in equation \eqref{E:Ass2} is utilized to derive the uniform moment bounds for the first- and second-order Malliavin derivatives of both the slow and fast processes (see, for example, Lemmas \ref{L:1st-der-Y-W2}, \ref{L:1st-der-X-W2}, \ref{L:Second-der-Integrating-factor}, and \ref{L:Second-der-Y-W1-W1}).
\end{remark}

\begin{remark}
Assumption \ref{A:Semigroup-boundedness} is a result of Assumption \ref{A:Derivatives-Assumption-c} and is primarily used to estimate the bounds in Section \ref{S:Prelimit-Exp-Var} on pre-limit expectation and variance. The final assumption (Assumption \ref{A:eta-eps-interaction}) demonstrates that the final rate obtained in Theorem \ref{T:Main-Result} depends on the interaction of the parameters \(\eps\) and \(\eta\), and that this final rate converges to zero.
\end{remark}

\section{Our main result (Theorem  \ref{T:Main-Result}) and its proof}\label{S:Main-Result}
\subsection{Main Result} In this section, we present our main result, which is a uniform-in-time quantitative central limit theorem, along with some related remarks. Before we do that, we recall a quantitative central limit theorem (Proposition \ref{P:CQCLT}) for multiscale diffusions on the compact time interval \([0,T]\) for a fixed \(T > 0\), as well as some relevant equations that will be frequently used throughout this manuscript.

\begin{proposition}\cite[Theorem 2.3]{bour_spilio_2025}\label{P:CQCLT}
Let $T>0$ be fixed, $\theta_t^\eps \triangleq \frac{X_t^\eps-\bar{X}_t}{\sqrt{\eps}}$ and $N \sim \mathscr{N}(0, \sigma_t^2).$ For sufficiently small $\eps$ and $\eta,$ and $0<\xi < \frac{1}{2},$ there exists a time-dependent positive constant $C$ such that
\begin{multline*}
\sup_{t \in [0,T]}d_W \left( \theta_t^\eps,N \right) \le C \left[ \eta^{\frac{1}{4}} + \eps^{\frac{1}{4}} + \left( \frac{\eta}{\eps}- \frac{1}{\gamma^2} \right)^{\frac{1}{2}}  + \left( \frac{\eta}{\eps} \right)^{\frac{1}{2}} \eta^{\frac{1}{2}-\xi} + \eps^{\frac{1}{2}-\xi} \right]\\
+ C \left( \left( \frac{\eta}{\eps} \right)^{\frac{1}{2}} \eta^{\frac{1}{4}} + \left( \frac{\eta}{\eps} \right) \eta^{\frac{1}{4}} + \left( 1+ \frac{\eta}{\eps} \right)e^{-\frac{CT}{\eta}} \right).
\end{multline*}
where, $\sigma_t^2$ represents the limiting variance and is defined in equation \eqref{E:Limiting-Variance} below.
\end{proposition}

Before this result, the author in \cite{spiliopoulos2014fluctuation} demonstrated that the rescaled process $\theta_t^\eps$ converges weakly in the space of continuous functions in $\mathscr{C}([0,T];\BR)$ to the solutions of the Ornstein-Uhlenbeck type process given by the equation:
$d \theta_t = \bar{c}'(\bar{X}_t) \theta_t \thinspace dt + \bar{q}^{1/2}(\bar{X}_t) \thinspace d\tilde{W}_t,$ where $\tilde{W}$ represents a standard one-dimensional Brownian motion and $\bar{q}(x)= \int_{\BR} q(x,y) \thinspace \mu_x(dy)$ (the function $q(x,y)$ is defined in \eqref{E:q-function}). The solution $\theta_t$ to this equation is a Gaussian process expressed as
$$\theta_t = \Psi(t) \int_0^t [\Psi(s)]^{-1} \bar{q}^{1/2}(\bar{X}_s) \thinspace d\tilde{W}_s.$$
Here, $\bar{X}$ and $\Psi(t)$ are again the solutions to the ordinary differential equations ({\sc ode}):
\begin{equation}\label{E:Semigroup-Equatn}
d \bar{X}_t = \bar{c}(\bar{X_t}) \thinspace dt, \quad \bar{X}_0= x_0; \quad \text{and} \quad {d}\Psi(t) = \bar{c}'(\bar{X}_t)\Psi(t) \thinspace dt,
\end{equation}
respectively. It is important to note that the mean of the process \(\theta_t\) is zero, while the limiting variance \(\sigma_t^2\) is defined as follows:
\begin{equation}\label{E:Limiting-Variance}
\sigma_t^2 \triangleq \int_0^t e^{\int_s^t 2 \bar{c}'(\bar{X}_u)du} \bar{q}(\bar{X}_s) \thinspace ds, 
\end{equation} 
and for the solution $\Upsilon(x,y)$ to the Poisson equation \eqref{E:Poisson-equation}, the function $q(x,y)$ is defined as:
\begin{equation}\label{E:q-function}
q(x,y) \triangleq  \sigma^2(y) + \frac{\tau^2(y)}{\gamma^2} (\Upsilon_y(x,y))^2.
\end{equation}

We now state our main result.
\begin{theorem}[Uniform-in-time {\sc qclt}]\label{T:Main-Result}
Let $\theta_t^\eps \triangleq \frac{X_t^\eps-\bar{X}_t}{\sqrt{\eps}}$ and $N \sim \mathscr{N}(0, \sigma_t^2).$ Assume that Assumptions \ref{A:Boundedness of sigma-tau} through \ref{A:eta-eps-interaction} hold. Then, for sufficiently small $\eps$ and $\eta$, there exists a (time-independent) positive constant $K$ such that
\begin{multline*}
\begin{aligned}
\sup_{t \ge 0} d_W(\theta_t^\eps, N) & \le  K \sqrt{\sqrt{\eta + \eps} + \left( \frac{\eta}{\eps} - \frac{1}{\gamma^2} \right)}  + K \sqrt{ \eps + \frac{\eta}{\eps} (\sqrt{\eps}+ \sqrt{\eta})}\\
 & \qquad \qquad \quad + K \left[ \left( \frac{\eta}{\eps} \right)^{\frac{1}{2}} \eta^{\frac{1}{4}} + \eps^{\frac{1}{4}}  +  \left( \frac{\eta}{\eps} \right)^{\frac{1}{8}} \eta^{\frac{1}{4}} + \left( \frac{\eta}{\eps} \right) \eta^{\frac{1}{4}} + \left( \frac{\eta}{\eps} \right)^{\frac{1}{4}} \eta^{\frac{1}{4}} +  \left( \frac{\eta}{\eps} \right)^{\frac{5}{8}} \eta^{\frac{1}{4}}  \right].
 \end{aligned}
\end{multline*}
\end{theorem}
\begin{remark}
In this manuscript, the constant $K$ is always \textit{time-independent}; however, it may depend on other parameters. The value of $K$ can change from one line to another.
\end{remark}

\begin{remark}
The convergence rate in our results is inherently suboptimal. This is because our proof relies on the findings in \cite{vido_2020}, which provide suboptimal bounds. For clarity in this manuscript, we focus solely on the one-dimensional case. However, the multi-dimensional case can be extended by adding dimensions without requiring new concepts.
\end{remark}

\subsection{Proof of Theorem \ref{T:Main-Result}}
In this section, we prove our main result, namely, Theorem \ref{T:Main-Result}. The proof is built on several key components, specifically Propositions \ref{P:First-Der-X}, \ref{P:Second-der-Prop}, \ref{P:Prelimit-1}, and \ref{P:Prelimit-2}. Propositions \ref{P:Prelimit-1} and \ref{P:Prelimit-2} examine the asymptotic behavior of the expectation and variance of the pre-limit process. In contrast, Propositions \ref{P:First-Der-X} and \ref{P:Second-der-Prop} focus on the first- and second-order Malliavin derivatives, respectively. Additionally, we utilize the result from \cite[Theorem 2.1]{vido_2020}, which is presented below.
\begin{proposition}\cite[Theorem 2.1]{vido_2020}\label{P:vidotto}
Let ${F} \in \mathbb{D}^{2,4}$ be such that $\mathbb{E}(F) = 0$ and
$\operatorname{Var}(F) = \tilde{\sigma}^2$, and let $N \sim
\mathscr{N}(0,\tilde{\sigma}^2)$. Then,
\begin{align*}
d_W(F,N) &\leq \frac{\sqrt{10}}{2 \tilde{\sigma}^2 }\BE \left[\|D^2 F \otimes D^2 F  \|_{\mathscr{H}^{\otimes 2}}^2 \right]^{\frac{1}{4}}\BE \left[\|DF\|_{\mathscr{H}}^4 \right]^{\frac{1}{4}} ,
\end{align*}
where the Malliavin operators $D$ and $D^2$ are the ones defined in
Subsection \ref{S:MD}, and $\otimes$ is the convolution operator defined as follows: for any two function $\mathsf{f}$, $\mathsf{g} \in \mathscr{H},$ 
\begin{equation}\label{E:Convolution}
[\mathsf{f} \otimes \mathsf{g}] \triangleq \int_0^t \mathsf{f}(x,z) \thinspace \mathsf{g}(y,z)\thinspace dz. 
\end{equation}
\end{proposition}
To utilize Proposition \ref{P:vidotto}, we define $\tilde{\theta}_t^\eps \triangleq \sqrt{\frac{\sigma_t^2}{\operatorname{Var}(\theta_t^\eps)}}\left[\theta^\eps_t - \BE(\theta^\eps_t)\right]$, where the process $\theta_t^\eps = \frac{X_t-\bar{X}_t}{\sqrt{\eps}}$, and $\sigma_t^2$ is the limiting variance as defined in equation \eqref{E:Limiting-Variance}. As $\mathbb{E}(\theta_t^\eps) \neq 0$ and $\mathbb{E}\left[(\theta_t^\eps)^2\right] \neq \sigma_t^2$, we
can write, denoting $N \sim \mathscr{N}(0,\sigma_t^2)$,
\begin{equation}\label{E:Rearranged-Equation}
\begin{aligned}
d_W \left( \theta_t^\eps,N \right) &\leq d_W \left( \theta_t^\eps,\tilde{\theta}_t^\eps \right)  +
d_W \left( \tilde{\theta}_t^\eps,N \right) \leq \mathbb{E} \left( \left|{\theta_t^\eps - \tilde{\theta}_t^\eps}\right| \right) + d_W \left(
  \tilde{\theta}_t^\eps,N \right)\\
  &\leq \mathbb{E} \left( |{\theta_t^\eps}|
    \right)\left|{1-\sqrt{\frac{\sigma_t^2}{\operatorname{Var}(\theta_t^\eps)}}}\right|
    + \sqrt{\frac{\sigma_t^2}{\operatorname{Var}(\theta_t^\eps)}} \thinspace
    |{\mathbb{E}(\theta_t^\eps)}| + d_W \left(
  \tilde{\theta}_t^\eps,N \right),
\end{aligned}
\end{equation}
where, the second inequality comes from the fact that the Wasserstein distance is bounded from above by the $L^1(\Omega)$-norm due to the Lipschitz property of its associated test functions. As $\mathbb{E}(\tilde{\theta}_t^\eps) = 0$ and $\mathbb{E}\left[(\tilde{\theta}_t^\eps)^2\right] = \sigma_t^2$, Proposition \ref{P:vidotto} is applicable to deal with the last term $d_W \left(
  \tilde{\theta}_t^\eps,N \right)$ in equation \eqref{E:Rearranged-Equation} to give 
\begin{equation}\label{E:Tilde-F-N}
\begin{aligned}
d_W \left(
  \tilde{\theta}_t^\eps,N \right) \leq  \frac{\sqrt{10}}{2 {\sigma}^2_t }\BE \left[\|D^2 \tilde{\theta}_t^\eps \otimes D^2 \tilde{\theta}_t^\eps  \|_{\mathscr{H}^{\otimes 2}}^2 \right]^{\frac{1}{4}}\BE \left[\|D \tilde{\theta}_t^\eps\|_{\mathscr{H}}^4 \right]^{\frac{1}{4}}.
\end{aligned}
\end{equation}

We are now in the position to prove our main result.

 \begin{proof}[Proof of Theorem \ref{T:Main-Result}]
We start noting
$D \tilde{\theta}_t^\eps =
\sqrt{\frac{\sigma_t^2}{\eps \operatorname{Var}(\theta_t^\eps)}} \thinspace DX_t^\eps$,  $D^2 \tilde{\theta}_t^\eps =
\sqrt{\frac{\sigma_t^2}{\eps \operatorname{Var}(\theta_t^\eps)}} \thinspace D^2 X_t^\eps$ and using equations \eqref{E:Rearranged-Equation} and \eqref{E:Tilde-F-N} to obtain
\begin{equation}\label{E:Main-result-Proof-Eq-1}
\begin{aligned}
d_W \left(
  {\theta}_t^\eps,N \right) & \leq  \mathbb{E} \left( |{\theta_t^\eps}|
    \right)\left|{1-\sqrt{\frac{\sigma_t^2}{\operatorname{Var}(\theta_t^\eps)}}}\right|
    + \sqrt{\frac{\sigma_t^2}{\operatorname{Var}(\theta_t^\eps)}} \thinspace
    |{\mathbb{E}(\theta_t^\eps)}| \\
  & \qquad \qquad \qquad \qquad \qquad \quad +    \frac{\sqrt{10}}{2 \eps \operatorname{Var}(\theta_t^\eps) } \thinspace \BE \left[\|D^2 X_t^\eps \otimes D^2 X_t^\eps  \|_{\mathscr{H}^{\otimes 2}}^2 \right]^{\frac{1}{4}}\BE \left[\|D X_t^\eps\|_{\mathscr{H}}^4 \right]^{\frac{1}{4}}.
\end{aligned}
\end{equation} 
We now apply Propositions \ref{P:Prelimit-1} and \ref{P:Prelimit-2} to control the terms $\sqrt{\frac{\sigma_t^2}{\operatorname{Var}(\theta_t^\eps)}} \thinspace
    |{\mathbb{E}(\theta_t^\eps)}|$ and $\mathbb{E} \left( |{\theta_t^\eps}|
    \right)\left|{1-\sqrt{\frac{\sigma_t^2}{\operatorname{Var}(\theta_t^\eps)}}}\right|,
    $ respectively, to obtain 
    \begin{equation}\label{E:Main-result-Proof-Eq-2}
    \begin{aligned}
    \mathbb{E} \left( |{\theta_t^\eps}|
    \right)\left|{1-\sqrt{\frac{\sigma_t^2}{\operatorname{Var}(\theta_t^\eps)}}}\right|
  &  + \sqrt{\frac{\sigma_t^2}{\operatorname{Var}(\theta_t^\eps)}} \thinspace
    |{\mathbb{E}(\theta_t^\eps)}| \\
    & \qquad  \le K \frac{\mathbb{E} \left( |{\theta_t^\eps}|
    \right)}{\sqrt{\operatorname{Var}(\theta_t^\eps)}} \sqrt{\sqrt{\eta + \eps} + \left( \frac{\eta}{\eps} - \frac{1}{\gamma^2} \right)  +  \frac{\eta^2}{{\eps}}+  {\eps + \frac{\eta}{\eps} (\sqrt{\eps}+ \sqrt{\eta})}} \\
    & \qquad  \qquad \qquad \qquad \qquad \quad  + K \sqrt{\frac{\sigma_t^2}{\operatorname{Var}(\theta_t^\eps)}}
    \left[  \frac{\eta}{\sqrt{\eps}} +  \sqrt{\eps + \frac{\eta}{\eps} (\sqrt{\eps}+ \sqrt{\eta})} \right],
    \end{aligned}
    \end{equation}
whereas, in equation \eqref{E:Main-result-Proof-Eq-1}, the term $\frac{\sqrt{10}}{2 \eps \operatorname{Var}(\theta_t^\eps) }\BE \left[\|D^2 X_t^\eps \otimes D^2 X_t^\eps  \|_{\mathscr{H}^{\otimes 2}}^2 \right]^{\frac{1}{4}}\BE \left[\|D X_t^\eps\|_{\mathscr{H}}^4 \right]^{\frac{1}{4}}$ can be handled by combining Propositions \ref{P:First-Der-X} and \ref{P:Second-der-Prop} to get
\begin{equation*}
\begin{aligned}
\BE \left[\|D^2X_t^\eps \otimes D^2X_t^\eps  \|_{\mathscr{H}^{\otimes 2}}^2 \right]^{\frac{1}{4}}\BE \left[\|DX_t^\eps\|_{\mathscr{H}}^4 \right]^{\frac{1}{4}} \le  K \left(\eta^{\frac{3}{4}}\eps^{\frac{1}{2}} + \eps^{\frac{5}{4}}+ \eps^{\frac{7}{8}}\eta^{\frac{3}{8}} + \eta^{\frac{5}{4}}+  \eps^{\frac{3}{4}} \eta^{\frac{1}{2}} + \eps^{\frac{3}{8}} \eta^{\frac{7}{8}}\right).
\end{aligned}
\end{equation*}
Further the above equation yields
 \begin{equation}\label{E:Main-result-Proof-Eq-3}
 \begin{aligned}
 \frac{\sqrt{10}}{2 \eps \operatorname{Var}(\theta_t^\eps) } & \BE \left[\|D^2X_t^\eps \otimes D^2X_t^\eps  \|_{\mathscr{H}^{\otimes 2}}^2 \right]^{\frac{1}{4}}\BE \left[\|DX_t^\eps\|_{\mathscr{H}}^4 \right]^{\frac{1}{4}} \\
 & \qquad \quad  \le  \frac{K}{\operatorname{Var}(\theta_t^\eps)} \left[ \left( \frac{\eta}{\eps} \right)^{\frac{1}{2}} \eta^{\frac{1}{4}} + \eps^{\frac{1}{4}}  +  \left( \frac{\eta}{\eps} \right)^{\frac{1}{8}} \eta^{\frac{1}{4}} + \left( \frac{\eta}{\eps} \right) \eta^{\frac{1}{4}} + \left( \frac{\eta}{\eps} \right)^{\frac{1}{4}} \eta^{\frac{1}{4}} +  \left( \frac{\eta}{\eps} \right)^{\frac{5}{8}} \eta^{\frac{1}{4}}  \right].
 \end{aligned}
 \end{equation}
Putting equations \eqref{E:Main-result-Proof-Eq-1}, \eqref{E:Main-result-Proof-Eq-2} and \eqref{E:Main-result-Proof-Eq-3} together, we have
\begin{equation*}
\begin{aligned}
d_W(\theta_t^\eps, N) 
& \le K \frac{\BE(|\theta_t^\eps|)}{ \sqrt{\text{Var}(\theta_t^\eps)}} \sqrt{\sqrt{\eta + \eps} + \left( \frac{\eta}{\eps} - \frac{1}{\gamma^2} \right) + \eps + \frac{\eta}{\eps} (\sqrt{\eps}+ \sqrt{\eta})} \\
& \qquad \qquad \qquad + K\frac{\sigma_t}{\sqrt{\text{Var}(\theta_t^\eps)}} \left[ \frac{\eta}{\sqrt{\eps}} + \left[\eps + \frac{\eta}{\eps} (\sqrt{\eps}+ \sqrt{\eta})\right]^{\frac{1}{2}} \right] \\
& \qquad \qquad \qquad + \frac{K}{\operatorname{Var}(\theta_t^\eps)} \left[ \left( \frac{\eta}{\eps} \right)^{\frac{1}{2}} \eta^{\frac{1}{4}} + \eps^{\frac{1}{4}}  +  \left( \frac{\eta}{\eps} \right)^{\frac{1}{8}} \eta^{\frac{1}{4}} + \left( \frac{\eta}{\eps} \right) \eta^{\frac{1}{4}} + \left( \frac{\eta}{\eps} \right)^{\frac{1}{4}} \eta^{\frac{1}{4}} +  \left( \frac{\eta}{\eps} \right)^{\frac{5}{8}} \eta^{\frac{1}{4}}  \right].
\end{aligned}
\end{equation*}
Finally, for sufficiently small $\eps$ and $\eta$, the variance of the process $\theta_t^\eps$ is a time-independent constant which gives
\begin{multline*}
\begin{aligned}
\sup_{t \ge 0} d_W(\theta_t^\eps, N) & \le  K \sqrt{\sqrt{\eta + \eps} + \left( \frac{\eta}{\eps} - \frac{1}{\gamma^2} \right)}  + K \sqrt{ \eps + \frac{\eta}{\eps} (\sqrt{\eps}+ \sqrt{\eta})}\\
 & \qquad \qquad \qquad + K \left[ \left( \frac{\eta}{\eps} \right)^{\frac{1}{2}} \eta^{\frac{1}{4}} + \eps^{\frac{1}{4}}  +  \left( \frac{\eta}{\eps} \right)^{\frac{1}{8}} \eta^{\frac{1}{4}} + \left( \frac{\eta}{\eps} \right) \eta^{\frac{1}{4}} + \left( \frac{\eta}{\eps} \right)^{\frac{1}{4}} \eta^{\frac{1}{4}} +  \left( \frac{\eta}{\eps} \right)^{\frac{5}{8}} \eta^{\frac{1}{4}}  \right].
 \end{aligned}
\end{multline*}
This concludes the proof of the result.
\end{proof}

\section{First-order Malliavin Derivatives}\label{S:First-MD}
In this section, we establish a uniform-in-time bound for the first-order Malliavin derivative of the process $X^\eps,$ i.e., a {\sc uit} bound for the quantity $\BE \left[\|DX_t^\eps\|_{\mathscr{H}}^4 \right]$. This is obtained in Proposition \ref{P:First-Der-X}. It is important to note that for a random variable $F$, the Malliavin derivative $DF$ is a two-dimensional vector $(D^{W_1}F, \thinspace D^{W_2}F)$, where $D^W F$ represents the Malliavin derivative with respect to the Brownian motion $W$. For $t \ge r,$ the Malliavin derivatives for system \eqref{E:Multiscale-Diffusion-Main-Eq} are given by the integral equations
\begin{equation}\label{E:First-Order-Malliavin-derivative}
\begin{aligned}
D_r^{W_1}X_t^{\varepsilon} &= \sqrt{\varepsilon}\sigma(Y_r^\eta) + \int_r^t \left[c_x(X_s^\varepsilon, Y_s^\eta) D_r^{W_1}X_s^{\varepsilon} + c_y(X_s^\varepsilon, Y_s^\eta)D_r^{W_1}Y_s^{\eta} \right] ds + \sqrt{\eps} \int_r^t \sigma_y (Y_s^\eta) D_r^{W_1}Y_s^{\eta} \thinspace dW_s^1 ,   \\
D_r^{W_1}Y_t^{\eta} &=  \frac{1}{\eta} \int_r^t \left[ f_y(Y_s^\eta)D_r^{W_1}Y_s^{\eta}  \right] ds + \frac{1}{\sqrt{\eta}}  \int_r^t \tau_y(Y_s^\eta)D_r^{W_1}Y_s^{\eta} \thinspace dW_s^2,  \qquad \quad \text{and} \\
D_r^{W_2}X_t^{\varepsilon} &= \int_r^t \left[c_x(X_s^\varepsilon, Y_s^\eta) D_r^{W_2}X_s^{\varepsilon} + c_y(X_s^\varepsilon, Y_s^\eta)D_r^{W_2}Y_s^{\eta} \right] ds +  \sqrt{\eps}\int_r^t \sigma_y(Y_s^\eta) D_r^{W_2}Y_s^{\eta} \thinspace dW_s^1,   \\
D_r^{W_2}Y_t^{\eta} &=  \frac{1}{\sqrt{\eta}} \tau(Y_r^n) + \frac{1}{\eta}  \int_r^t \left[ f_y(Y_s^\eta)D_r^{W_2}Y_s^{\eta}  \right] ds  + \frac{1}{\sqrt{\eta}}  \int_r^t \tau_y(Y_s^\eta)D_r^{W_2}Y_s^{\eta} \thinspace dW_s^2. 
\end{aligned}
\end{equation}

\begin{lemma}\label{L:1st-der-Y-W1}
Let the process $Y_t^\eta$ be the solution of the second equation in system \eqref{E:Multiscale-Diffusion-Main-Eq} and Assumptions \ref{A:Boundedness of sigma-tau} and \ref{A:Derivatives-Assumption-f} hold. Then, for any $t \ge r,$ and $p \in \BN,$ we have 
$$\BE\left[|D_r^{W_1}Y_t^{\eta}|^p \right] =0.$$
\end{lemma}
The proof of Lemma \ref{L:1st-der-Y-W1} is trivial and we skip the details. In the next lemma, we obtain the bound for the term $\BE\left[(D_r^{W_1}X_t^{\varepsilon})^p\right],~p \in \BN.$

\begin{lemma}\label{L:1st-der-Y-W2}
Let the process $Y_t^\eta$ be the solution of the second equation in system \eqref{E:Multiscale-Diffusion-Main-Eq} and Assumptions \ref{A:Boundedness of sigma-tau} and \ref{A:Derivatives-Assumption-f} hold. Then, for any $t \ge r,$ and $p \in \BN,$ there exists a time-independent positive constant $K$ such that 
$$\BE \left[|D_r^{W_2}Y_t^{\eta}|^p \right] \le \frac{K}{ (\sqrt{\eta})^p}  e^{\frac{-p \lambda }{\eta}(t-r)}.$$
\end{lemma}
\begin{proof}
From equation \eqref{E:First-Order-Malliavin-derivative}, we have
\begin{equation*}
D_r^{W_2}Y_t^{\eta} = \frac{1}{\sqrt{\eta}}\tau(Y_r^\eta) \exp\left[ \int_r^t \frac{1}{\eta} f_y(Y_s^\eta) \thinspace ds  + \int_r^t \frac{1}{\sqrt{\eta}}\tau_y( Y_s^\eta) \thinspace dW_s^2 - \frac{1}{2 \eta} \int_r^t [\tau_y( Y_s^\eta)]^2 \thinspace ds \right].
\end{equation*} 
Increasing the power to $p$ and performing a simple algebra yield
\begin{equation*}
\begin{aligned}
\left|D_r^{W_2}Y_t^{\eta}\right|^p  & = \left[ \frac{1}{\sqrt{\eta}}\left|\tau(Y_r^\eta)\right| \right]^p \exp \left[ \int_r^t \frac{p}{\eta} f_y(Y_s^\eta)\thinspace ds  + \int_r^t \frac{p}{\sqrt{\eta}}\tau_y( Y_s^\eta) \thinspace dW_s^2  \right. \\
& \qquad \qquad \qquad  \qquad \qquad \quad \quad \left. -  \frac{p}{2 \eta} \int_r^t [\tau_y( Y_s^\eta)]^2 \thinspace ds + \frac{p^2 }{2 \eta} \int_r^t [\tau_y( Y_s^\eta)]^2 \thinspace ds - \frac{p^2}{2 \eta} \int_r^t [\tau_y( Y_s^\eta)]^2 \thinspace ds  \right] \\
& =  \left[ \frac{1}{\sqrt{\eta}}\left|\tau(Y_r^\eta)\right| \right]^p \exp \left[ \int_r^t \frac{p}{\eta} f_y(Y_s^\eta) \thinspace ds -  \frac{p}{2 \eta} \int_r^t [\tau_y( Y_s^\eta)]^2 \thinspace ds + \frac{p^2 }{2 \eta} \int_r^t [\tau_y( Y_s^\eta)]^2 \thinspace ds  \right] \times \\
& \qquad \qquad \qquad  \qquad \qquad \quad \quad \qquad \qquad \qquad \quad  \exp \left[ \int_r^t \frac{p}{\sqrt{\eta}}\tau_y( Y_s^\eta) \thinspace dW_s^2 - \frac{p^2}{2 \eta} \int_r^t [\tau_y( Y_s^\eta)]^2 \thinspace ds  \right] \\
& =  \left[ \frac{1}{\sqrt{\eta}}\left|\tau(Y_r^\eta)\right| \right]^p \exp \left[ \int_r^t \frac{p}{\eta} \left\{f_y(Y_s^\eta)+  \frac{1}{2}(p-1)[\tau_y( Y_s^\eta)]^2 \right\} ds  \right] \times \\
& \qquad \qquad \qquad  \qquad \qquad \quad \quad \qquad \qquad \qquad \quad  \exp \left[ \int_r^t \frac{p}{\sqrt{\eta}}\tau_y( Y_s^\eta) \thinspace dW_s^2 - \frac{p^2}{2 \eta} \int_r^t [\tau_y( Y_s^\eta)]^2 \thinspace ds  \right]. 
\end{aligned}
\end{equation*}
Next, noting the term $\exp \left[ \int_r^t \frac{p}{\sqrt{\eta}}\tau_y( Y_s^\eta) \thinspace dW_s^2 - \frac{p^2}{2 \eta} \int_r^t [\tau_y( Y_s^\eta)]^2 \thinspace ds  \right]$ is a martingale, using Assumption \ref{A:Derivatives-Assumption-f}: $\sup_{y} \left[ f_y(y) + \frac{1}{2}(p-1) [\tau_y(y)]^2 \right] \le - \lambda$ and the uniform boundedness of $\tau$, we obtain the required bound.
\end{proof}

\begin{lemma}\label{L:1st-der-generic-lemma}
For $i \in \{1,2\}$ and $p \in \BN$, let the process $D_r^{W_i}X_t^{\varepsilon}$ be the solution of equations in system \eqref{E:Multiscale-Diffusion-Main-Eq}. Then, there exist time-independent positive constants $K$ and $K_{c_x}$ such that
\begin{multline*}
\BE \left[(D_r^{W_i}X_t^{\varepsilon})^{2p} \right] \le K{\varepsilon^{p}}e^{-(2p)\alpha (t-r)} \ind_{(i=1)} + K \underbrace{\int_r^t \cdots \int_r^t}_{2p\operatorname{-times}} \prod_{j=1}^{2p}e^{-\alpha (t-s_j)} \prod_{j=1}^{2p} \left\{ \BE (D_r^{W_i}Y_{s_j}^{\eta})^{2p} \right\}^{\frac{1}{2p}}\thinspace ds_1 \cdots ds_{2p} \\
+ K \eps^p e^{-(2p)\alpha (t-r)} \left[ \underbrace{ \int_r^t \cdots \int_r^t}_{2p\operatorname{-times}} \prod_{j=1}^{2p}e^{K_{c_x} (s_j-r)} \prod_{j=1}^{2p} \left\{ \BE (D_r^{W_i}Y_{s_j}^{\eta})^{4p} \right\}^{\frac{1}{2p}}\thinspace ds_1 \cdots ds_{2p} \right]^{\frac{1}{2}}.
\end{multline*}
\end{lemma}
\begin{proof}
For $i \in \{1,2\}$, we have
\begin{multline*}
D_r^{W_i}X_t^{\varepsilon} = \sqrt{\varepsilon}\sigma(Y_r^\eta)\ind_{(i=1)}  + \int_r^t \left[c_x(X_s^\varepsilon, Y_s^\eta) D_r^{W_i}X_s^{\varepsilon} + c_y(X_s^\varepsilon, Y_s^\eta)D_r^{W_i}Y_s^{\eta} \right] ds \\
 + \sqrt{\eps} \int_r^t \sigma_y (Y_s^\eta) D_r^{W_i}Y_s^{\eta} \thinspace dW_s^1. 
\end{multline*}
Increasing the power to $2p$ on both sides of the solution\footnote{For the general {\sc sde} \cite[Problem 6.15]{KS91}: $dX_t = [A_t  X_t+ a_t ] \thinspace dt + [S_t X_t + \sigma_t] \thinspace dW_t$ and the integrating factor process $Z_t  \triangleq \exp \left[ \int_0^t A_u \thinspace du + \int_0^t S_u \thinspace dW_u - \frac{1}{2} \int_0^t S_u^2 \thinspace du \right]$,
the solution to the {\sc sde} is given by
\begin{equation*}
\begin{aligned}
X_t & = Z_t \left[ X_0 + \int_0^t Z_u^{-1} \left\{ a_u - S_u \sigma_u \right\} \thinspace du + \int_0^t Z_u^{-1} \sigma_u \thinspace dW_u \right].
\end{aligned}
\end{equation*}} of above equation followed by taking expectation, one can get
\begin{multline*}
\BE \left[(D_r^{W_i}X_t^{\varepsilon})^{2p} \right]  \le  {\varepsilon^{p}} e^{-(2p)\alpha (t-r)}\sigma(Y_r^\eta)^{2p}\ind_{(i=1)}+ K \BE \left[ \int_r^t e^{\int_s^t c_x(X_u^\varepsilon, Y_u^\eta) \thinspace du}c_y(X_s^\varepsilon, Y_s^\eta)D_r^{W_i}Y_s^{\eta} \thinspace ds \right]^{2p} \\
 + \eps^p K  \BE \left[ \int_r^t e^{\int_s^t c_x(X_u^\varepsilon, Y_u^\eta) \thinspace du}\sigma_y( Y_s^\eta)D_r^{W_i}Y_s^{\eta} \thinspace dW_s^1 \right]^{2p}.
\end{multline*}
Boundedness of $\sigma$ gives
\begin{equation*}
\begin{aligned}
& \BE  \left[(D_r^{W_i}X_t^{\varepsilon})^{2p} \right]  \le  K {\varepsilon^{p}} e^{-(2p)\alpha (t-r)} \ind_{(i=1)}  + K \BE \left[ \int_r^t e^{\int_s^t c_x(X_u^\varepsilon, Y_u^\eta) \thinspace du}c_y(X_s^\varepsilon, Y_s^\eta)D_r^{W_i}Y_s^{\eta} \thinspace ds \right]^{2p} \\
&   + \eps^p K \BE \left[ \int_r^t e^{\int_s^t c_x(X_u^\varepsilon, Y_u^\eta) \thinspace du}\sigma_y( Y_s^\eta)D_r^{W_i}Y_s^{\eta} \thinspace dW_s^1 \right]^{2p}  \triangleq K {\varepsilon^{p}}e^{-(2p)\alpha (t-r)}\ind_{(i=1)}+  I_1(t,r; \eps, \eta) + I_2(t,r; \eps, \eta).
\end{aligned}
\end{equation*}
We now handle the terms $I_1(t,r; \eps, \eta) $ and $I_2(t,r; ; \eps, \eta)$ separately. Applying H\"older's inequality repeatedly and Assumption \ref{A:Derivatives-Assumption-c} to the function $c$, we obtain
\begin{equation*}
\begin{aligned}
I_1(t,r; \eps, \eta) &  \le  K \BE \left[ \int_r^t e^{\int_s^t c_x(X_u^\varepsilon, Y_u^\eta) \thinspace du} |D_r^{W_i}Y_s^{\eta}| \thinspace ds \right]^{2p} \le K \BE \left[ \int_r^t e^{-\alpha (t-s)} |D_r^{W_i}Y_s^{\eta}| \thinspace ds \right]^{2p} \\
& \le  K \underbrace{\int_r^t \cdots \int_r^t}_{2p\operatorname{-times}} \prod_{j=1}^{2p}e^{-\alpha (t-s_j)} \prod_{j=1}^{2p} \left\{ \BE (D_r^{W_i}Y_{s_j}^{\eta})^{2p} \right\}^{\frac{1}{2p}}\thinspace ds_1 \cdots ds_{2p}.
\end{aligned}
\end{equation*}
Next, for $I_2(t,r; \eps, \eta)$, using H\"older's inequality and martingale moment inequality \cite[Proposition 3.26]{KS91} followed by Assumption \ref{A:Derivatives-Assumption-c}, we have
\begin{equation*}
\begin{aligned}
I_2(t,r; \eps, \eta) &  = K \eps^p \BE \left[ \left( e^{\int_r^t c_x(X_u^\varepsilon, Y_u^\eta) \thinspace du} \right)^{2p} \left( \int_r^t e^{-\int_r^s c_x(X_u^\varepsilon, Y_u^\eta)\thinspace du}\sigma_y( Y_s^\eta)D_r^{W_i}Y_s^{\eta} \thinspace dW_s^1 \right)^{2p} \right] \\
& \le K \eps^p \left[ \BE \left( e^{\int_r^t c_x(X_u^\varepsilon, Y_u^\eta) \thinspace du} \right)^{4p} \right]^{\frac{1}{2}}  \left[ \BE \left( \int_r^t e^{-\int_r^s c_x(X_u^\varepsilon, Y_u^\eta) \thinspace du}\sigma_y( Y_s^\eta)D_r^{W_i}Y_s^{\eta} \thinspace dW_s^1 \right)^{4p}\right]^{\frac{1}{2}} \\
& \le K \eps^p e^{-(2p)\alpha (t-r)} \left[ \BE \left( \int_r^t e^{-2\int_r^s c_x(X_u^\varepsilon, Y_u^\eta) \thinspace du}\sigma_y( Y_s^\eta)^2 (D_r^{W_i}Y_s^{\eta})^2 \thinspace ds \right)^{2p}\right]^{\frac{1}{2}}.
\end{aligned}
\end{equation*}
The boundedness of $\sigma_y$ from Assumption \ref{A:Derivatives-Assumption-c}, the inequality $-c_x \le |c_x| \le K_{c_x}$ and H\"older's inequality yield
\begin{equation*}
\begin{aligned}
I_2(t,r; \eps, \eta) & \le K \eps^p e^{-(2p)\alpha (t-r)} \left[ \BE \left( \int_r^t e^{K_{c_x} (s-r)} (D_r^{W_i}Y_s^{\eta})^2 \thinspace ds \right)^{2p}\right]^{\frac{1}{2}}\\
& \le K \eps^p e^{-(2p)\alpha (t-r)} \left[ \int_r^t \cdots \int_r^t \prod_{j=1}^{2p}e^{K_{c_x} (s_j-r)} \prod_{j=1}^{2p} \left\{ \BE (D_r^{W_i}Y_{s_j}^{\eta})^{4p} \right\}^{\frac{1}{2p}}\thinspace ds_1 \cdots ds_{2p} \right]^{\frac{1}{2}}.
\end{aligned}
\end{equation*}
Putting  all these expressions together, we obtain the required result.
\end{proof}
In the next two lemmas, we estimate the first-order Malliavin derivative terms $\BE \left[|D_r^{W_1}X_t^{\varepsilon}|^{2p} \right]$ and $\BE \left[|D_r^{W_2}X_t^{\varepsilon}|^{2p} \right]$. The proofs of these lemmas employ Lemma \ref{L:1st-der-generic-lemma}.
\begin{lemma}\label{L:1st-der-X-W1}
Let the process $X_t^\eps$ be the solution of the first equation in system \eqref{E:Multiscale-Diffusion-Main-Eq}. Then, for any $t \ge r,$ and $p \in \BN$, there exists a time-independent positive constant $K$ such that
$$\BE\left[\left|D_r^{W_1}X_t^{\varepsilon}\right|^{2p} \right]  \le  K {\varepsilon^{p}}e^{-(2p)\alpha (t-r)}.$$
\end{lemma}
\begin{proof}
For $i=1$, we apply Lemmas \ref{L:1st-der-Y-W1} and \ref{L:1st-der-generic-lemma} to obtain the required result.
\end{proof}

\begin{lemma}\label{L:1st-der-X-W2}
Let the process $X_t^\eps$ be the solution of the first equation in system \eqref{E:Multiscale-Diffusion-Main-Eq}. Then, for any $t \ge r,$ and $p \in \BN$, there exists a time-independent positive constant $K$ such that
$$\BE \left[\left|D_r^{W_2}X_t^{\varepsilon}\right|^{2p} \right] \le K \left[ \eta^p + \eps^p \right]e^{-(2p)\alpha (t-r)}.$$
\end{lemma}

\begin{proof}
For $i=2$, applying Lemmas \ref{L:1st-der-Y-W2} and \ref{L:1st-der-generic-lemma}, we obtain
\begin{equation*}
\begin{aligned}
\BE \left[|D_r^{W_2}X_t^{\varepsilon}|^{2p} \right] & \le K \underbrace{\int_r^t \cdots \int_r^t}_{2p\operatorname{-times}} \prod_{j=1}^{2p}e^{-\alpha (t-s_j)} \prod_{j=1}^{2p} \left\{ \BE (D_r^{W_2}Y_{s_j}^{\eta})^{2p} \right\}^{\frac{1}{2p}}\thinspace ds_1 \cdots ds_{2p} \\
& \qquad  \quad + K \eps^p e^{-(2p)\alpha (t-r)} \left[ \underbrace{\int_r^t \cdots \int_r^t}_{2p\operatorname{-times}} \prod_{j=1}^{2p}e^{K_{c_x} (s_j-r)} \prod_{j=1}^{2p} \left\{ \BE (D_r^{W_2}Y_{s_j}^{\eta})^{4p} \right\}^{\frac{1}{2p}}\thinspace ds_1 \cdots ds_{2p} \right]^{\frac{1}{2}} \\
& \qquad  \quad \triangleq J_1(t,r; \eps, \eta) + J_2(t,r; \eps, \eta).
\end{aligned}
\end{equation*}
First, for $J_1(t,r; \eps, \eta),$ we use Lemma \ref{L:1st-der-Y-W2} to get
\begin{equation*}
\begin{aligned}
J_1(t,r; \eps, \eta) \le K \int_r^t \cdots \int_r^t \prod_{j=1}^{2p} \frac{1}{\sqrt{\eta}}e^{-\alpha (t-s_i)} e^{- \frac{ \lambda}{ \eta}(s_i-r)} \thinspace ds_1 \cdots &     ds_{2p}  = \frac{K}{\eta^p} \left( \int_r^t e^{-\alpha (t-s)} e^{- \frac{ \lambda}{ \eta}(s-r)}ds  \right)^{2p}\\
& = \frac{K}{\eta^p} e^{-(2p) \alpha t} e^{(2p)\frac{\lambda}{\eta}r} \left( \int_r^t e^{-\left( \frac{\lambda}{\eta}- \alpha \right)s} ds \right)^{2p}.
\end{aligned}
\end{equation*}
Next, for sufficiently small $\eta$ and fixed $\alpha, \thinspace \lambda$, the term $ \frac{\lambda}{\eta} - \alpha$ is a positive number, hence for some $\eta_0>0$ such that $0 < \eta \le \eta_0$, we have
\begin{equation*}
\begin{aligned}
J_1(t,r; \eps, \eta)  \le \frac{K}{\eta^p} e^{-(2p) \alpha t} e^{(2p)\frac{\lambda}{\eta}r} \left( \int_r^\infty e^{-\left( \frac{\lambda}{\eta}- \alpha \right)s} ds \right)^{2p} & \le \frac{K}{\eta^p} e^{-(2p) \alpha t} e^{(2p)\frac{\lambda}{\eta}r} \frac{1}{\left(-\alpha+ \frac{\lambda}{\eta}\right)^{2p}} e^{- (2p)\left( \frac{\lambda}{\eta}- \alpha \right)r}\\
 & \le  K \eta^p e^{-(2p)\alpha(t-r)}. 
\end{aligned}
\end{equation*}
For $J_2(t,r; \eps, \eta),$ using Lemma \ref{L:1st-der-Y-W2}, we obtain
\begin{equation*}
\begin{aligned}
J_2(t,r; \eps, \eta) &  \le   K \eps^p e^{-(2p)\alpha (t-r)} \left[ \left( \int_r^t e^{2 K_{c_x}(u-r)} \frac{1}{\eta}e^{-\frac{2 \lambda}{\eta}(u-r)}du \right)^{2p} \right]^{\frac{1}{2}} \\
& \le \frac{K \eps^p}{\eta^p} e^{-(2p)\alpha (t-r)} \left[ e^{-4p K_{c_x}r} e^{4p \frac{\lambda}{\eta}r} \left( \int_r^t e^{2 K_{c_x}u} e^{-2 \frac{\lambda}{\eta}u}du \right)^{2p} \right]^{\frac{1}{2}} \\
& = \frac{K \eps^p}{\eta^p} e^{-(2p)\alpha (t-r)} e^{-2p K_{c_x}r} e^{2p \frac{\lambda}{\eta}r} \left( \int_r^t  e^{-2 \left(\frac{\lambda}{\eta}- K_{c_x} \right)u}du \right)^{p}.
\end{aligned}
\end{equation*}
For sufficiently small $\eta$, the term $ \frac{\lambda}{\eta} - K_{c_x}$ is a positive number, which gives $J_2(t,r; \eps, \eta) \le  K \eps^p e^{-(2p)\alpha (t-r)}.$
This completes the proof of lemma.
\end{proof}
We now present the main result of this subsection. This result is utilized in the proof of Theorem \ref{T:Main-Result}.
\begin{proposition}\label{P:First-Der-X}
Let the process $X_t^\eps$ be the solution of the first equation in system \eqref{E:Multiscale-Diffusion-Main-Eq}. Then, there exists a time-independent positive constant $K$ such that
$$\sup_{t \ge 0}\BE \left[\|DX_t^\eps\|_{\mathscr{H}}^4 \right] \le K \left[ \eps^2 + \eta^2 \right].$$
\end{proposition}
\begin{proof}
We start noting 
$\BE \left[\|DX_t^\eps\|_{\mathscr{H}}^4 \right] = \int_{[0,t]^2}\BE \left[m_u(t)m_v(t)\right]du \thinspace dv,$
where, $m_r(t) \triangleq |D_r^{W_1}X_t^\eps|^2 + |D_r^{W_2}X_t^\eps|^2.$ An application of H\"older's inequality yields
\begin{equation*}
\begin{aligned}
\BE \left[\|DX_t^\eps\|_{\mathscr{H}}^4 \right] & = \int_{[0,t]^2} \BE\left[m_u(t)m_v(t)\right]du \thinspace dv  \le \int_{[0,t]^2} \left[ \BE m_u(t)^2 \right]^{\frac{1}{2}} \left[\BE m_v(t)^2 \right]^{\frac{1}{2}} \thinspace du \thinspace dv\\
& \qquad \quad \qquad \quad = \left[\int_{[0,t]} \left[ \BE m_u(t)^2 \right]^{\frac{1}{2}}du \right]^2 \le K \left[\int_{[0,t]} \left[ \BE |D_u^{W_1}X_t^\eps|^4 + \BE |D_u^{W_2}X_t^\eps|^4  \right]^{\frac{1}{2}}du \right]^2. \\
\end{aligned}
\end{equation*}
To control the first-order Malliavin derivatives terms in the above inequality, we apply Lemmas \ref{L:1st-der-X-W1} and \ref{L:1st-der-X-W2} to get
$\BE \left[\|DX_t^\eps\|_{\mathscr{H}}^4 \right] \le  K [\eps^2 + \eta^2].$ Hence, the result is proved.
\end{proof}

\section{Second-order Malliavin Derivatives}\label{S:Second-MD}
In this section, we establish a uniform-in-time bound for the second-order Malliavin derivative of the process $X^\eps,$ i.e., a {\sc uit} bound for the quantity $\BE \left[\|D^2X_t^\eps \otimes D^2X_t^\eps  \|_{\mathscr{H}^{\otimes 2}}^2 \right]$. This is established in Proposition \ref{P:Second-der-Prop}. We recall the framework for second-order Malliavin derivatives from \cite{bour_spilio_2025}. For a random variable $F$ and the convolution operator $\otimes$ defined in equation \eqref{E:Convolution},
\begin{equation*}
D^2 F= 
\begin{pmatrix}
    D^{W^1, W^1} F & D^{W^1, W^2} F \\
    D^{W^2, W^1} F & D^{W^2, W^2} F 
\end{pmatrix}, \quad \text{and}
\end{equation*}

\begin{equation}\label{E:Matrix-Second-der}
D^2X_t^\eps \otimes D^2X_t^\eps= 
\begin{pmatrix}
    \sum\limits_{k=1}^2 \int_0^t  D^{W^1, W^k}_{u,v} X_t^\eps D^{W^k, W^1}_{u,w} X_t^\eps \thinspace du  & \sum\limits_{k=1}^2 \int_0^t  D^{W^1, W^k}_{u,v} X_t^\eps D^{W^k, W^2}_{u,w} X_t^\eps \thinspace du \\
    \sum\limits_{k=1}^2 \int_0^t  D^{W^2, W^k}_{u,v} X_t^\eps D^{W^k, W^1}_{u,w} X_t^\eps \thinspace du & \sum\limits_{k=1}^2 \int_0^t  D^{W^2, W^k}_{u,v} X_t^\eps D^{W^k, W^2}_{u,w} X_t^\eps \thinspace du
\end{pmatrix}.
\end{equation}

We employ the following lemmas very frequently in this section.
 
\begin{lemma}\label{L:Second-der-Integrating-factor}
For any $t \ge r,$ let the process $Z$ be defined as follows:
\begin{equation}\label{E:Z-process}
Z_{t,r} \triangleq \exp \left[ \int_{r}^t \frac{1}{\eta}f_y(Y_s^\eta)\thinspace ds + \int_{r}^t \frac{1}{\sqrt{\eta}}\tau_y(Y_s^\eta)\thinspace dW_s^2 - \frac{1}{2\eta} \int_{r}^t \tau_y(Y_s^\eta)^2 \thinspace ds \right].
\end{equation}
Suppose that Assumption \ref{A:Derivatives-Assumption-f} holds. Then, there exists a time-independent constant $K>0$ such that 
$$\BE \left[ \left| Z_{t,r}\right|^{2p} \right] \le K e^{-\frac{(2p)\lambda}{\eta}(t-r)}.$$
\end{lemma}

\begin{proof}
We use the martingale arguments to obtain
\begin{equation*}\label{E:Sec-der-Y-W1W1-Eq3}
\begin{aligned}
\BE \left[ \left| Z_{t,r}\right|^{2p} \right] & = \BE \exp \left[ \int_{r}^t \frac{2p}{\eta}f_y(Y_u^\eta)\thinspace du + \int_{r}^t \frac{2p}{\sqrt{\eta}}\tau_y(Y_u^\eta)\thinspace dW_u^2 - \frac{2p}{2\eta} \int_{r}^t \tau_y(Y_u^\eta)^2 \thinspace du \right. \\
& \qquad \qquad \qquad \qquad \qquad  \quad  \qquad \qquad \qquad \qquad \quad   \left. +  \frac{2p^2}{\eta} \int_{r}^t \tau_y(Y_u^\eta)^2 \thinspace du - \frac{2p^2}{\eta} \int_{r}^t \tau_y(Y_u^\eta)^2 \thinspace du  \right] \\
& = \BE \left[ \exp \left(  \int_{r}^t \frac{2p}{\eta}f_y(Y_u^\eta)\thinspace du + \frac{2p^2}{\eta} \int_{r}^t \tau_y(Y_u^\eta)^2 \thinspace du - \frac{2p}{2\eta} \int_{r}^t \tau_y(Y_u^\eta)^2 \thinspace du   \right) \times \right. \\
& \qquad \qquad \qquad \qquad \qquad \quad      \left. \exp \left( \int_{r}^t \frac{2p}{\sqrt{\eta}}\tau_y(Y_u^\eta)\thinspace dW_u^2 - \frac{2p^2}{\eta} \int_{r}^t \tau_y(Y_u^\eta)^2 \thinspace du  \right)    \right] \le K e^{-\lambda \frac{2p}{\eta} (t-r)},
\end{aligned}
\end{equation*}
where the last inequality follows from the condition $\sup_{y} \left[f_y(y) + (p-\frac{1}{2} ) [\tau_y(y)]^2 \right] \le - \lambda$ and noting the term $\exp \left( \int_{r}^t \frac{2p}{\sqrt{\eta}}\tau_y(Y_u^\eta)\thinspace dW_u^2 - \frac{2p^2}{\eta} \int_{r}^t \tau_y(Y_u^\eta)^2 \thinspace du  \right)$ is a martingale. Hence, the lemma is proved.
\end{proof}

\begin{lemma}\label{L:51}
Let the process  $Z$ be defined in equation \eqref{E:Z-process} above. Then, for any $t \ge r_1 \vee r_2,$ and the process $Q \in L^{2p}(\Omega)$, there exists a time-independent positive constant $K$ such that 
$$\frac{1}{\eta^p} \BE \left[ \left| \int_{r_1 \vee r_2}^t Z_{t,s} Q_s \thinspace dW_s^2 \right|^{2p} \right] \le \frac{K}{\eta}  \int_{r_1 \vee r_2}^t e^{- \frac{2p \lambda}{\eta}(t-s)} \BE (|Q_s|^{2p}) \thinspace ds. $$
\end{lemma}

\begin{proof}
For any process $Q \in L^{2p}(\Omega)$ and $Z$ defined in equation \eqref{E:Z-process}, we begin noting
\begin{equation}\label{E:E52}
\begin{aligned}
\BE \left( \frac{1}{\sqrt{\eta}} \int_{r_1 \vee r_2}^t Z_{t,s} Q_s \thinspace dW_s^2 \right)^{2p}
&  =  \BE \left( \frac{1}{\sqrt{\eta}} Z_{t, r_1 \vee r_2} \int_{r_1 \vee r_2}^t Z_{s, r_1 \vee r_2}^{-1} Q_s \thinspace dW_s^2 \right)^{2p}  \triangleq \BE \left( \frac{1}{\sqrt{\eta}} Z_{t, r_1 \vee r_2} V_{t, r_1 \vee r_2} \right)^{2p},
\end{aligned}
\end{equation}
where the process $V_{t, r_1 \vee r_2} \triangleq \int_{r_1 \vee r_2}^t Z_{s, r_1 \vee r_2}^{-1} Q_s \thinspace dW_s^2$ and the process $Z_{t, r_1 \vee r_2}$ satisfies
$Z_{t,r}  = 1+ \frac{1}{\eta} \int_r^t Z_{s,r} f_y (Y_s^\eta) ds + \frac{1}{\sqrt{\eta}} \int_r^t Z_{s,r} \tau_y (Y_s^\eta) dW_s^2.$ Applying It\^o's product formula to the function \\ $\frac{1}{\sqrt{\eta}} Z_{t, r_1 \vee r_2} V_{t, r_1 \vee r_2}$, It\^o's formula to the process $\left|\frac{1}{\sqrt{\eta}} Z_{t, r_1 \vee r_2} V_{t, r_1 \vee r_2} \right|^{2p}$, Young's inequality\footnote{For any $\zeta > 0$, we have the following inequalities: 
 $a^{p-1} b  \le \frac{p-1}{p} \zeta^{\frac{p}{p-1}} a^p + \frac{1}{p \thinspace \zeta^p} b^p,$ and $a^{p-2} b^2  \le \frac{p-2}{p} \zeta^{\frac{p}{p-2}} a^p + \frac{2}{p \thinspace \zeta^p} b^p.$ Initially, we often keep \(\zeta\) arbitrary and choose it later to ensure that the quantity of interest is sufficiently small.} to the products
$\left(\frac{1}{\sqrt{\eta}} Z_{t, r_1 \vee r_2} V_{t, r_1 \vee r_2} \right)^{2p-1} Q_t,$ and $\left(\frac{1}{\sqrt{\eta}} Z_{t, r_1 \vee r_2} V_{t, r_1 \vee r_2} \right)^{2p-2} Q_t^2,$
followed by the condition \\ $\sup_{y \in \BR}\left\{ f_y(y) + \eps' K  \|\tau_y\|_\infty  + \eps'' (2p-1)\frac{K}{2}  +  (2p-1)\frac{K}{2} \|\tau_y\|^2_\infty \right\} < - \lambda, $ we obtain
\begin{equation*}
\frac{d}{dt} \BE  \left( \left|\frac{1}{\sqrt{\eta}} Z_{t, r_1 \vee r_2} V_{t, r_1 \vee r_2} \right|^{2p} \right) \le -\frac{2p \lambda}{\eta} \BE \left[ \left(\frac{1}{\sqrt{\eta}} Z_{t, r_1 \vee r_2} V_{t, r_1 \vee r_2} \right)^{2p} \right] + \frac{2p}{\eta} \BE \left[ |Q_t|^p \left\{ K \|\tau_y\|_\infty + K \right\} \right].
\end{equation*}
Solving the above inequality, we get
\begin{equation}\label{E:E53}
\BE  \left( \left|\frac{1}{\sqrt{\eta}} Z_{t, r_1 \vee r_2} V_{t, r_1 \vee r_2} \right|^{2p} \right) \le \frac{K}{\eta}  \int_{r_1 \vee r_2}^t e^{- \frac{2p \lambda}{\eta}(t-s)} \BE (|Q_s|^{2p}) ds,
\end{equation}
which completes the proof of the lemma.
\end{proof}

\begin{lemma}\label{L:L2}
Let the process  $Z$ be defined in equation \eqref{E:Z-process} and the process $Q$ in Lemma \ref{L:51} equals $\tau_{yy}(Y_t^\eta) D_{r_2}^{W_2}Y_t^\eta \cdot D_{r_1}^{W_2}Y_t^\eta$. Then, for any $t \ge r_1 \vee r_2,$ there exists a time-independent positive constant $K$ such that
$$\frac{1}{\eta^p} \BE \left[ \left| \int_{r_1 \vee r_2}^t Z_{t,s} \tau_{yy}(Y_s^\eta) D_{r_2}^{W_2}Y_s^\eta \cdot D_{r_1}^{W_2}Y_s^\eta \thinspace dW_s^2 \right|^{2p} \right] \le \frac{K}{\eta^{2p}} e^{-\frac{(2p)\lambda}{\eta}(t- r_1 \wedge r_2)}.$$
\end{lemma}
\begin{proof}
Using Lemma \ref{L:51} and equations \eqref{E:E52}, \eqref{E:E53} followed by H\"older's inequality and Lemma \ref{L:1st-der-Y-W2}, we have 
\begin{equation*}
\begin{aligned}
\BE  \left( \left|\frac{1}{\sqrt{\eta}} Z_{t, r_1 \vee r_2} V_{t, r_1 \vee r_2} \right|^{2p} \right) & \le \frac{K}{\eta}  \int_{r_1 \vee r_2}^t e^{- \frac{2p \lambda}{\eta}(t-s)} \left( \BE |D_{r_2}^{W_2}Y_s^\eta|^{4p}  \right)^{\frac{1}{2}} \left( \BE |D_{r_1}^{W_2}Y_s^\eta|^{4p}  \right)^{\frac{1}{2}} ds \\
& \le \frac{K}{\eta^{2p+1}}  \int_{r_1 \vee r_2}^t e^{- \frac{2p \lambda}{\eta}(t-s)}  e^{-\frac{2p \lambda}{\eta}(s-r_2)} e^{-\frac{ 2p \lambda}{\eta}(s-r_1)} ds \le \frac{K}{\eta^{2p}} e^{-\frac{(2p)\lambda}{\eta}(t- r_1 \wedge r_2)}.
\end{aligned}
\end{equation*}
This concludes the proof of the lemma.
\end{proof}

In Section \ref{S:D2-11} below, we focus on establishing a bound for the term $\BE \left[(D_{r_1,r_2}^{W_1 W_1} X_t^\eps)^p \right]$, where the process $D_{r_1,r_2}^{W_1 W_1} X_t^\eps$ is the solution of equation \eqref{E:Second-der-X-W1-W1}. This bound is specified in Lemma \ref{L:Second-der-X-W1-W1}.

\subsection{Bound associated with the term $D_{r_1,r_2}^{W_1 W_1} X_t^\eps$}\label{S:D2-11}
For $t \ge r_1 \vee r_2,$ the second-order derivative processes $D_{r_1,r_2}^{W_1 W_1} X_t^\eps$ and $D_{r_1,r_2}^{W_1 W_1} Y_t^\eta$ satisfy the following stochastic integral equations:
\begin{equation}\label{E:Second-der-X-W1-W1}
\begin{aligned}
D_{r_1,r_2}^{W_1 W_1} X_t^\eps & = \sqrt{\eps} \sigma_y(Y_{r_2}^\eta) D_{r_1}^{W_1}Y_{r_2}^\eta + \sqrt{\eps} \sigma_y(Y_{r_1}^\eta) D_{r_2}^{W_1}Y_{r_1}^\eta\\
& \quad \quad + \int_{r_1 \vee r_2}^t \left[c_{xx}(X_s^\eps,Y_s^\eta) D_{r_2}^{W_1}X_s^\eps \cdot D_{r_1}^{W_1}X_s^\eps +  c_{xy}(X_s^\eps,Y_s^\eta) D_{r_2}^{W_1}Y_s^\eta \cdot D_{r_1}^{W_1}X_s^\eps \right]ds \\
& \quad \quad + \int_{r_1 \vee r_2}^t \left[c_{xy}(X_s^\eps,Y_s^\eta) D_{r_2}^{W_1}X_s^\eps \cdot D_{r_1}^{W_1}Y_s^\eta +  c_{yy}(X_s^\eps,Y_s^\eta) D_{r_1}^{W_1}Y_s^\eta \cdot D_{r_2}^{W_1}Y_s^\eta \right]ds \\
& \quad \quad + \int_{r_1 \vee r_2}^t \left[c_{x}(X_s^\eps,Y_s^\eta) D_{r_1, r_2}^{W_1 W_1}X_s^\eps  +  c_{y}(X_s^\eps,Y_s^\eta) D_{r_1, r_2}^{W_1 W_1}Y_s^\eta \right]ds \\
& \quad \quad + \sqrt{\eps} \int_{r_1 \vee r_2}^t \sigma_{yy} (Y_s^\eta) D_{r_1}^{W_1}Y_s^\eta \cdot D_{r_2}^{W_1}Y_s^\eta \thinspace dW_s^1 + \sqrt{\eps} \int_{r_1 \vee r_2}^t \sigma_{y} (Y_s^\eta) D_{r_1, r_2}^{W_1 W_1}Y_s^\eta \thinspace dW_s^1,
\end{aligned}
\end{equation}
and 
\begin{multline}\label{E:Second-der-Y-W1-W1}
D_{r_1,r_2}^{W_1 W_1} Y_t^\eta = \frac{1}{\eta} \int_{r_1 \vee r_2}^t f_{yy}(Y_s^\eta) D_{r_2}^{W_1}Y_s^\eta \cdot D_{r_1}^{W_1}Y_s^\eta \thinspace ds  +  \frac{1}{\eta} \int_{r_1 \vee r_2}^t f_{y}(Y_s^\eta)\cdot D_{r_1, r_2}^{W_1 W_1}Y_s^\eta \thinspace ds \\
+ \frac{1}{\sqrt{\eta}} \int_{r_1 \vee r_2}^t \tau_{yy}(Y_s^\eta) D_{r_2}^{W_1}Y_s^\eta \cdot D_{r_1}^{W_1}Y_s^\eta \thinspace dW_s^2  +  \frac{1}{\sqrt{\eta}} \int_{r_1 \vee r_2}^t \tau_{y}(Y_s^\eta)\cdot D_{r_1, r_2}^{W_1 W_1}Y_s^\eta \thinspace dW_s^2.
\end{multline}
\begin{lemma}\label{L:Second-der-Y-W1-W1}
Let the process $D_{r_1,r_2}^{W_1 W_1} Y_t^\eta$ be the solution of equation \eqref{E:Second-der-Y-W1-W1} above. Then, for any $p \in \BN,$ and $t \ge r_1 \vee r_2,$ we have
$$\BE\left[(D_{r_1,r_2}^{W_1 W_1} Y_t^\eta)^{2p} \right]=0.$$
\end{lemma}
\begin{proof}
From equation \eqref{E:Second-der-Y-W1-W1}, we get 
\begin{multline*}
D_{r_1,r_2}^{W_1 W_1} Y_t^\eta  = \int_{r_1 \vee r_2}^t Z_{t,s} \left\{ \frac{1}{\eta} f_{yy}(Y_s^\eta) D_{r_2}^{W_1}Y_s^\eta \cdot D_{r_1}^{W_1}Y_s^\eta - \frac{1}{\eta} \tau_{y}(Y_s^\eta) \tau_{yy}(Y_s^\eta) D_{r_2}^{W_1}Y_s^\eta \cdot D_{r_1}^{W_1}Y_s^\eta  \right\} ds \\
 + \frac{1}{\sqrt{\eta}} \int_{r_1 \vee r_2}^t Z_{t,s} \tau_{yy}(Y_s^\eta) D_{r_2}^{W_1}Y_s^\eta \cdot D_{r_1}^{W_1}Y_s^\eta \thinspace dW_s^2, 
\end{multline*}
where, the process $Z_{t,r_1 \vee r_2}$ is defined in equation \eqref{E:Z-process}. Now, increasing the power to $2p$ on both sides and then taking expectation, we have
\begin{equation}\label{E:Sec-der-Y-W1W1-Eq1}
\begin{aligned}
\BE & \left[ \left(D_{r_1,r_2}^{W_1 W_1} Y_t^\eta\right)^{2p} \right]  \le \frac{K}{\eta^{2p}} \BE \left( \int_{r_1 \vee r_2}^t Z_{t,s} \left\{ f_{yy}(Y_s^\eta)  -  \tau_{y}(Y_s^\eta) \tau_{yy}(Y_s^\eta)\right\} D_{r_2}^{W_1}Y_s^\eta \cdot D_{r_1}^{W_1}Y_s^\eta \thinspace ds \right)^{2p} \\
&  \quad   + \frac{K}{\eta^p} \BE \left( \int_{r_1 \vee r_2}^t Z_{t,s} \tau_{yy}(Y_s^\eta) D_{r_2}^{W_1}Y_s^\eta \cdot D_{r_1}^{W_1}Y_s^\eta \thinspace dW_s^2 \right)^{2p}  \triangleq \mathscr{I}_1(t, r_1 \vee r_2; \eta) + \mathscr{I}_2(t, r_1 \vee r_2; \eta).
\end{aligned}
\end{equation}
For $\mathscr{I}_2(t, r_1 \vee r_2; \eta)$, Lemma \ref{L:51} with the process $Q_s \triangleq \tau_{yy}(Y_s^\eta) D_{r_2}^{W_1}Y_s^\eta \cdot D_{r_1}^{W_1}Y_s^\eta$ and Lemma \ref{L:1st-der-Y-W1} yield $\mathscr{I}_2(t, r_1 \vee r_2; \eta)=0.$ For $\mathscr{I}_1(t, r_1 \vee r_2; \eta)$, we use H\"older's inequality repeatedly to get 
\begin{multline*}
\mathscr{I}_1(t, r_1 \vee r_2; \eta)  \le  \frac{K}{\eta^{2p}} \underbrace{\int_{r_1 \vee r_2}^t \cdots \int_{ r_1 \vee r_2}^t}_{2p\operatorname{-times}} \left[ \prod_{i=1}^{2p} (\BE Z_{t,s_i}^{4p}) \right]^{\frac{1}{4p}} \times \\
 \left[ \prod_{i=1}^{2p} \left(\BE \left\{[f_{yy}(Y_s^\eta)  -  \tau_{y}(Y_s^\eta) \tau_{yy}(Y_s^\eta)]^{4p} [D_{r_2}^{W_1}Y_{s_i}^\eta \cdot D_{r_1}^{W_1}Y_{s_i}^\eta]^{4p} \right\} \right) \right]^{\frac{1}{4p}} \thinspace ds_1 \cdots ds_{2p}.
\end{multline*}
Finally, using Lemmas \ref{L:1st-der-Y-W1}, \ref{L:Second-der-Integrating-factor}, we get $\mathscr{I}_1(t, r_1 \vee r_2; \eta)=0$. 
\end{proof}

\begin{lemma}\label{L:Second-der-X-W1-W1}
Let the process $D_{r_1,r_2}^{W_1 W_1} X_t^\eps$ be the solution of equation \eqref{E:Second-der-X-W1-W1}. Then, for any $p \in \BN$ and $t \ge r_1 \vee r_2$, there exists a time-independent positive constant $K$ such that
$$\BE \left[\left|D_{r_1,r_2}^{W_1 W_1} X_t^\eps\right|^{2p} \right] \le K \eps^{2p} e^{-{(2p) \alpha}(t- r_1 \wedge r_2 )}.$$
\end{lemma}
\begin{proof} 
Solving equation \eqref{E:Second-der-X-W1-W1} followed by increasing the power to $2p$ and taking expectation on both sides, we obtain
\begin{equation*}
\begin{aligned}
& \BE \left[ (D_{r_1,r_2}^{W_1 W_1} X_t^\eps)^{2p} \right]  \le K \eps^p \BE \left( e^{2p\int_{r_1 \vee r_2}^t c_x(X_u^\varepsilon, Y_u^\eta) \thinspace du} \left[ \sigma_y(Y_{r_2}^\eta) D_{r_1}^{W_1}Y_{r_2}^\eta  \right]^{2p} \right) \\
& \quad  + K \eps^p \BE \left( e^{2p\int_{r_1 \vee r_2}^t c_x(X_u^\varepsilon, Y_u^\eta)\thinspace du} \left[ \sigma_y(Y_{r_1}^\eta) D_{r_2}^{W_1}Y_{r_1}^\eta \right]^{2p}\right) \\
& \quad + K \BE \left( \int_{r_1 \vee r_2}^t e^{\int_{s}^t c_x(X_u^\varepsilon, Y_u^\eta) \thinspace du} \left[c_{xx}(X_s^\eps,Y_s^\eta) D_{r_2}^{W_1}X_s^\eps \cdot D_{r_1}^{W_1}X_s^\eps +  c_{xy}(X_s^\eps,Y_s^\eta) D_{r_2}^{W_1}Y_s^\eta \cdot D_{r_1}^{W_1}X_s^\eps \right]ds \right)^{2p} \\
& \quad + K \BE \left( \int_{r_1 \vee r_2}^t e^{\int_{s}^t c_x(X_u^\varepsilon, Y_u^\eta) \thinspace du} \left[c_{xy}(X_s^\eps,Y_s^\eta) D_{r_2}^{W_1}X_s^\eps \cdot D_{r_1}^{W_1}Y_s^\eta +  c_{yy}(X_s^\eps,Y_s^\eta) D_{r_1}^{W_1}Y_s^\eta \cdot D_{r_2}^{W_1}Y_s^\eta \right]ds \right)^{2p} \\
& \quad + K \BE \left( \int_{r_1 \vee r_2}^t e^{\int_{s}^t c_x(X_u^\varepsilon, Y_u^\eta) \thinspace du} \left[ c_{y}(X_s^\eps,Y_s^\eta) D_{r_1, r_2}^{W_1 W_1}Y_s^\eta \right]ds \right)^{2p} \\
\end{aligned}
\end{equation*}
\begin{equation*}
\begin{aligned}
& \quad + K \eps^p \BE \left( \int_{r_1 \vee r_2}^t e^{\int_{s}^t c_x(X_u^\varepsilon, Y_u^\eta) \thinspace du} \left[ \sigma_{yy} (Y_s^\eta) D_{r_1}^{W_1}Y_s^\eta \cdot D_{r_2}^{W_1}Y_s^\eta \right]dW_s^1 \right)^{2p}    \\
& \quad + K \eps^p  \BE \left( \int_{r_1 \vee r_2}^t e^{\int_{s}^t c_x(X_u^\varepsilon, Y_u^\eta) \thinspace du} \left[ \sigma_{y} (Y_s^\eta) D_{r_1, r_2}^{W_1 W_1}Y_s^\eta \right]dW_s^1 \right)^{2p}.
\end{aligned}
\end{equation*}
The above terms can be handled using H\"older's inequality,  Assumption \ref{A:Derivatives-Assumption-c}, the boundedness of the derivatives of $c, \thinspace \sigma$ and Lemmas \ref{L:1st-der-Y-W1}, \ref{L:1st-der-X-W1} and \ref{L:Second-der-Y-W1-W1} to deal with the Malliavin derivatives of the processes $X_t^\eps$ and $Y_t^\eta$. For the moments of the stochastic integrals, one can apply the similar arguments to Lemma \ref{L:1st-der-generic-lemma} to obtain
$\BE \left[ (D_{r_1,r_2}^{W_1 W_1} X_t^\eps)^{2p} \right]  \le K \BE \left( \int_{r_1 \vee r_2}^t e^{-\alpha(t-s)} \left| D_{r_2}^{W_1}X_s^\eps \cdot D_{r_1}^{W_1}X_s^\eps  \right| ds \right)^{2p}.$
We again use H\"older's inequality and Lemma \ref{L:1st-der-X-W1} to get
\begin{equation*}
\begin{aligned}
\BE & \left[  (D_{r_1,r_2}^{W_1 W_1} X_t^\eps)^{2p} \right]  \le K  \int_{r_1 \vee r_2}^t \cdots \int_{r_1 \vee r_2}^t \prod_{i=1}^{2p} e^{-\alpha (t-s_i)} \left[ \BE (D_{r_2}^{W_2}X_{s_i}^\eps)^{4p}\right]^{\frac{1}{4p}} \left[ \BE (D_{r_1}^{W_1}X_{s_i}^\eps)^{4p} \right]^{\frac{1}{4p}}ds_1 \cdots ds_{2p}\\
& \qquad \quad \le K \eps^{2p} \int_{r_1 \vee r_2}^t \cdots \int_{r_1 \vee r_2}^t \prod_{i=1}^{2p} e^{-\alpha (t-s_i)} e^{-\alpha (s_i-r_2)} e^{-\alpha (s_i-r_1)} ds_1 \cdots ds_{2p} \le K \eps^{2p} e^{-{2p \alpha}(t- r_1 \wedge r_2 )},
\end{aligned}
\end{equation*}
which concludes the proof of the lemma.
\end{proof}

In the next section, we focus on getting a bound for the term $\BE \left[|D_{r_1,r_2}^{W_1 W_2} X_t^\eps|^{2p}\right]$, where the process $D_{r_1,r_2}^{W_1 W_2} X_t^\eps$ is the solution of equation \eqref{E:Second-der-X-W1-W2} below. The bound is specified in Lemma \ref{L:Second-der-X-W1-W2}.
\subsection{Bound associated with the term $D_{r_1,r_2}^{W_1 W_2} X_t^\eps$}
For $t \ge r_1 \vee r_2,$ the second-order derivative processes $D_{r_1,r_2}^{W_1 W_2} X_t^\eps$ and $D_{r_1,r_2}^{W_1 W_2} Y_t^\eta$ satisfy the following stochastic integral equations:
\begin{equation}\label{E:Second-der-X-W1-W2}
\begin{aligned}
D_{r_1,r_2}^{W_1 W_2} X_t^\eps & =   \sqrt{\eps} \sigma_y(Y_{r_1}^\eta) D_{r_2}^{W_2}Y_{r_1}^\eta\\
&  + \int_{r_1 \vee r_2}^t \left[c_{xx}(X_s^\eps,Y_s^\eta) D_{r_2}^{W_2}X_s^\eps \cdot D_{r_1}^{W_1}X_s^\eps +  c_{xy}(X_s^\eps,Y_s^\eta) D_{r_2}^{W_2}Y_s^\eta \cdot D_{r_1}^{W_1}X_s^\eps \right]ds \\
&  + \int_{r_1 \vee r_2}^t \left[c_{xy}(X_s^\eps,Y_s^\eta) D_{r_2}^{W_2}X_s^\eps \cdot D_{r_1}^{W_1}Y_s^\eta +  c_{yy}(X_s^\eps,Y_s^\eta) D_{r_1}^{W_1}Y_s^\eta \cdot D_{r_2}^{W_2}Y_s^\eta \right]ds \\
&  + \int_{r_1 \vee r_2}^t \left[c_{x}(X_s^\eps,Y_s^\eta) D_{r_1, r_2}^{W_1 W_2}X_s^\eps  +  c_{y}(X_s^\eps,Y_s^\eta) D_{r_1, r_2}^{W_1 W_2}Y_s^\eta \right]ds \\
&  + \sqrt{\eps} \int_{r_1 \vee r_2}^t \sigma_{yy} (Y_s^\eta) D_{r_1}^{W_1}Y_s^\eta \cdot D_{r_2}^{W_2}Y_s^\eta \thinspace dW_s^1 + \sqrt{\eps} \int_{r_1 \vee r_2}^t \sigma_{y} (Y_s^\eta) D_{r_1, r_2}^{W_1 W_2}Y_s^\eta \thinspace dW_s^1,
\end{aligned}
\end{equation}
and 
\begin{equation}\label{E:Second-der-Y-W1-W2}
\begin{aligned}
D_{r_1,r_2}^{W_1 W_2} Y_t^\eta  & = \frac{1}{\sqrt{\eta}} \tau_y(Y_{r_2}^\eta) D_{r_1}^{W_1}Y_{r_2}^\eta \\
& +  \frac{1}{\eta} \int_{r_1 \vee r_2}^t f_{yy}(Y_s^\eta)\cdot D_{r_2}^{W_2}Y_s^\eta \cdot D_{r_1}^{W_1}Y_s^\eta ds  +  \frac{1}{\eta} \int_{r_1 \vee r_2}^t f_{y}(Y_s^\eta)\cdot D_{r_1, r_2}^{W_1 W_2}Y_s^\eta ds \\
&+ \frac{1}{\sqrt{\eta}} \int_{r_1 \vee r_2}^t \tau_{yy}(Y_s^\eta)\cdot D_{r_2}^{W_2}Y_s^\eta \cdot D_{r_1}^{W_1}Y_s^\eta \thinspace dW_s^2  +  \frac{1}{\sqrt{\eta}} \int_{r_1 \vee r_2}^t \tau_{y}(Y_s^\eta)\cdot D_{r_1, r_2}^{W_1 W_2}Y_s^\eta \thinspace dW_s^2.
\end{aligned}
\end{equation}

\begin{lemma}\label{L:Second-der-Y-W1-W2}
Let the process $D_{r_1,r_2}^{W_1 W_2} Y_t^\eta$ be the solution of equation \eqref{E:Second-der-Y-W1-W2} above. Then, for any $p \in \BN,$ and $t \ge r_1 \vee r_2,$ we have
$$\BE\left[\left|D_{r_1,r_2}^{W_1 W_2} Y_t^\eta\right|^{2p} \right]=0.$$
\end{lemma}
\begin{proof}
The definition of the process $Z$ from equation \eqref{E:Z-process} and equation \eqref{E:Second-der-Y-W1-W2} above yield
\begin{equation*}
\begin{aligned}
D_{r_1,r_2}^{W_1 W_2} Y_t^\eta & = \frac{1}{\sqrt{\eta}} Z_{t,r_1 \vee r_2}  \tau_y(Y_{r_2}^\eta) D_{r_1}^{W_1}Y_{r_2}^\eta + \frac{1}{\sqrt{\eta}} \int_{r_1 \vee r_2}^t Z_{t,s} \tau_{yy}(Y_s^\eta) D_{r_2}^{W_2}Y_s^\eta \cdot D_{r_1}^{W_1}Y_s^\eta \thinspace dW_s^2 \\
& \qquad \quad +  \int_{r_1 \vee r_2}^t Z_{t,s} \left\{ \frac{1}{\eta} f_{yy}(Y_s^\eta) D_{r_2}^{W_2}Y_s^\eta \cdot D_{r_1}^{W_1}Y_s^\eta - \frac{1}{\eta} \tau_{y}(Y_s^\eta) \tau_{yy}(Y_s^\eta) D_{r_2}^{W_2}Y_s^\eta \cdot D_{r_1}^{W_1}Y_s^\eta  \right\} ds. 
\end{aligned}
\end{equation*}
Next, following the similar steps to the proof of Lemma \ref{L:Second-der-Y-W1-W1} and applying Lemmas \ref{L:1st-der-Y-W1}, and \ref{L:1st-der-Y-W2}, we obtain the required result.
\end{proof}

\begin{lemma}\label{L:Second-der-X-W1-W2}
Let the process $D_{r_1,r_2}^{W_1 W_2} X_t^\eps$ be the solution of equation \eqref{E:Second-der-X-W1-W2}. Then, for any $p \in \BN$ and $t \ge r_1 \vee r_2$, there exists a time-independent positive constant $K$ such that
\begin{multline*}
\BE \left[\left|D_{r_1,r_2}^{W_1 W_2} X_t^\eps\right|^{2p} \right] \le K \left[\frac{\eps^p}{\eta^p} + \eps^p \eta^p \right] e^{- (2p)\alpha (t-r_1 \vee r_2)} e^{- (2p)\frac{\lambda}{\eta} (r_1 - r_2)}\ind_{(r_1 \ge r_2)} \\
 + K  \left[ \eps^p \eta^p + {\eps^{2p}} \right] e^{-(2p) \alpha (t - r_1 \wedge r_2)}.
\end{multline*}
\end{lemma}

\begin{proof}
We start solving equation \eqref{E:Second-der-X-W1-W2} to get
\begin{equation*}
\begin{aligned}
D_{r_1,r_2}^{W_1 W_2} X_t^\eps & = \sqrt{\eps} \thinspace \sigma_y(Y_{r_1}^\eta) D_{r_2}^{W_2}Y_{r_1}^\eta e^{\int_{r_1 \vee r_2}^t c_{x}(X_u^\eps,Y_u^\eta) \thinspace du} \\
& + \int_{r_1 \vee r_2}^t e^{\int_s^t c_{x}(X_u^\eps,Y_u^\eta) \thinspace du} \left[c_{xx}(X_s^\eps,Y_s^\eta) D_{r_2}^{W_2}X_s^\eps \cdot D_{r_1}^{W_1}X_s^\eps +  c_{xy}(X_s^\eps,Y_s^\eta) D_{r_2}^{W_2}Y_s^\eta \cdot D_{r_1}^{W_1}X_s^\eps \right]ds \\
& + \int_{r_1 \vee r_2}^t e^{\int_s^t c_{x}(X_u^\eps,Y_u^\eta) \thinspace du}  \left[c_{xy}(X_s^\eps,Y_s^\eta) D_{r_2}^{W_2}X_s^\eps \cdot D_{r_1}^{W_1}Y_s^\eta +  c_{yy}(X_s^\eps,Y_s^\eta) D_{r_1}^{W_1}Y_s^\eta \cdot D_{r_2}^{W_2}Y_s^\eta \right]ds \\
& + \int_{r_1 \vee r_2}^t e^{\int_s^t c_{x}(X_u^\eps,Y_u^\eta) \thinspace du} \left[  c_{y}(X_s^\eps,Y_s^\eta) D_{r_1, r_2}^{W_1 W_2}Y_s^\eta \right]ds \\
& + \sqrt{\eps} \int_{r_1 \vee r_2}^t e^{\int_s^t c_{x}(X_u^\eps,Y_u^\eta) \thinspace du} \sigma_{yy} (Y_s^\eta) D_{r_1}^{W_1}Y_s^\eta \cdot D_{r_2}^{W_2}Y_s^\eta \thinspace dW_s^1 \\
&+  \sqrt{\eps} \int_{r_1 \vee r_2}^t e^{\int_s^t c_{x}(X_u^\eps,Y_u^\eta) \thinspace du} \sigma_{y} (Y_s^\eta) D_{r_1, r_2}^{W_1 W_2}Y_s^\eta \thinspace dW_s^1.
\end{aligned}
\end{equation*}
Now, raising the power to $2p$, taking expectation on both sides, using the boundedness of $\sigma_y, c_{xx}, c_{xy}, c_{yy}, c_{y}$ (Assumptions \ref{A:Derivatives-Assumption-c}) and condition $c_x(x,y) < -\alpha$ (Assumption \ref{A:Derivatives-Assumption-c}), we have
\begin{equation}\label{E:Second-der-X-W1-W2-Eq-1}
\begin{aligned}
\BE \left[(D_{r_1,r_2}^{W_1 W_2} X_t^\eps)^{2p} \right] & \le K \eps^p e^{- (2p)\alpha (t-r_1 \vee r_2)} \BE \left[ \left(D_{r_2}^{W_2}Y_{r_1}^\eta \right)^{2p} \right] \\
& \quad  + K \BE \left[ \left( \int_{r_1 \vee r_2}^t e^{-\alpha (t-s) }  |D_{r_2}^{W_2}X_s^\eps \cdot D_{r_1}^{W_1}X_s^\eps| ds\right)^{2p} \right] \\
& \quad + K \BE \left[ \left( \int_{r_1 \vee r_2}^t e^{-\alpha (t-s) }   |D_{r_2}^{W_2}Y_s^\eta \cdot D_{r_1}^{W_1}X_s^\eps|  ds\right)^{2p} \right] \\
& \quad + K \BE \left[ \left( \int_{r_1 \vee r_2}^t e^{-\alpha (t-s) } \left\{  |D_{r_2}^{W_2}X_s^\eps \cdot D_{r_1}^{W_1}Y_s^\eta |+ |D_{r_1}^{W_1}Y_s^\eta \cdot D_{r_2}^{W_2}Y_s^\eta | \right\} ds\right)^{2p} \right] \\
& \quad + K \BE \left[ \left( \int_{r_1 \vee r_2}^t e^{-\alpha (t-s) }  | D_{r_1, r_2}^{W_1 W_2}Y_s^\eta| ds\right)^{2p} \right] \\
\end{aligned}
\end{equation}
\begin{equation*}
\begin{aligned}
& \quad + K \eps^p \thinspace \BE \left[ \left( \int_{r_1 \vee r_2}^t e^{\int_s^t c_{x}(X_u^\eps,Y_u^\eta) \thinspace du} \sigma_{yy} (Y_s^\eta) D_{r_1}^{W_1}Y_s^\eta \cdot D_{r_2}^{W_2}Y_s^\eta \thinspace dW_s^1  \right)^{2p} \right]\\
& \quad + K \eps^p \thinspace  \BE \left[ \left( \int_{r_1 \vee r_2}^t e^{\int_s^t c_{x}(X_u^\eps,Y_u^\eta) \thinspace du} \sigma_{y} (Y_s^\eta) D_{r_1, r_2}^{W_1 W_2}Y_s^\eta \thinspace dW_s^1  \right)^{2p} \right] \triangleq  \sum_{i=1}^7 \mathsf{I}_i(t, r_1 \vee r_2; \eps,\eta).
\end{aligned}
\end{equation*}
For the terms $\mathsf{I}_i(t, r_1 \vee r_2; \eps,\eta), \thinspace i=4, \cdots, 7$ in equation \eqref{E:Second-der-X-W1-W2-Eq-1} using Lemmas \ref{L:1st-der-Y-W1}, \ref{L:1st-der-Y-W2}, \ref{L:1st-der-X-W2}, \ref{L:Second-der-Y-W1-W2}, boundedness of derivatives of $\sigma$ (Assumption \ref{A:Derivatives-Assumption-c}) and the arguments similar to Lemma \ref{L:1st-der-generic-lemma} to handle the moments of the stochastic integrals,  we have
$\mathsf{I}_i(t, r_1 \vee r_2; \eps,\eta) = 0.$
We now focus on the term $\mathsf{I}_1(t, r_1 \vee r_2; \eps,\eta)$ in equation \eqref{E:Second-der-X-W1-W2-Eq-1}. Employing Lemma \ref{L:1st-der-Y-W2}, we have
\begin{equation}\label{E:Second-der-X-W1-W2-Eq-2}
\mathsf{I}_1(t, r_1 \vee r_2; \eps,\eta) \le K \frac{\eps^p}{\eta^p} e^{- (2p)\alpha (t-r_1 \vee r_2)} e^{- (2p)\frac{\lambda}{\eta} (r_1 - r_2)}\ind_{(r_1 \ge r_2)}.
\end{equation}
Next, for the term $\mathsf{I}_2(t, r_1 \vee r_2; \eps,\eta)$ in equation \eqref{E:Second-der-X-W1-W2-Eq-1}, we use H\"older's inequality repeatedly to get
\begin{equation*}
\mathsf{I}_2(t, r_1 \vee r_2; \eps,\eta) \le K \int_{r_1 \vee r_2}^t \cdots \int_{r_1 \vee r_2}^t \prod_{i=1}^{2p} e^{-\alpha (t-s_i)} \cdot \prod_{i=1}^{2p} \left[  \BE (D_{r_2}^{W_2}X_{s_i}^\eps)^{2p}\right]^{\frac{1}{2p}} \left[ \BE (D_{r_1}^{W_1}X_{s_i}^\eps)^{2p} \right]^{\frac{1}{2p}}ds_1 \cdots ds_{2p}.
\end{equation*}
Applying Lemmas \ref{L:1st-der-X-W1} to deal with the term $D_{r}^{W_1}X_{t}^\eps$, and \ref{L:1st-der-X-W2} to handle the term $D_{r}^{W_2}X_{t}^\eps$, we obtain 
\begin{equation}\label{E:Second-der-X-W1-W2-Eq-3}
\begin{aligned}
\mathsf{I}_2(t, r_1 \vee r_2; \eps,\eta) &  \le K \int_{r_1 \vee r_2}^t \cdots \int_{r_1 \vee r_2}^t \prod_{i=1}^{2p} \left(\sqrt{\eta} + {\sqrt{\eps}}\right)\sqrt{\eps} e^{-\alpha (t-s_i)} e^{-{\alpha}(s_i-r_1)} e^{-{\alpha}(s_i-r_2)} ds_1 \cdots ds_{2p} \\
& \le K \eps^p \left( \eta^p + {\eps^p} \right)  e^{-2p \alpha t} e^{2p {\alpha} r_1 } e^{2p {\alpha} r_2 }   \left(\int_{r_1 \vee r_2}^t e^{ - \alpha s}ds \right)^{2p} \le K \eps^p \left( \eta^p + {\eps^p} \right) e^{-(2p) \alpha (t - r_1 \wedge r_2)}.
\end{aligned}
\end{equation}
We next consider the term $\mathsf{I}_3(t, r_1 \vee r_2; \eps,\eta)$ in equation \eqref{E:Second-der-X-W1-W2-Eq-1}. Using Lemmas \ref{L:1st-der-Y-W2} and \ref{L:1st-der-X-W1}, we have
\begin{equation*}
\begin{aligned}
\mathsf{I}_3(t, r_1 \vee r_2; \eps,\eta) & \le K \int_{r_1 \vee r_2}^t \cdots \int_{r_1 \vee r_2}^t \prod_{i=1}^{2p} e^{-\alpha (t-s_i)} \cdot \prod_{i=1}^{2p} \left[  \BE (D_{r_2}^{W_2}Y_{s_i}^\eta)^{2p}\right]^{\frac{1}{2p}} \left[ \BE (D_{r_1}^{W_1}X_{s_i}^\eps)^{2p} \right]^{\frac{1}{2p}}ds_1 \cdots ds_{2p}\\
 &  \le K \int_{r_1 \vee r_2}^t \cdots \int_{r_1 \vee r_2}^t \prod_{i=1}^{2p} \frac{\sqrt{\eps}}{\sqrt{\eta}} e^{-\alpha (t-s_i)} e^{-{\alpha}(s_i-r_1)} e^{-\frac{\lambda}{ \eta}(s_i-r_2)} ds_1 \cdots ds_{2p} \\
& \le K \left(\frac{\eps}{\eta}\right)^p {\eta^{2p}}  e^{-(2p){\alpha}(t-r_1)} e^{(2p)\frac{\lambda}{ \eta} r_2} e^{-(2p)\frac{\lambda}{ \eta} (r_1 \vee r_2)},
\end{aligned}
\end{equation*}
which gives
\begin{equation}\label{E:Second-der-X-W1-W2-Eq-4}
\mathsf{I}_3(t, r_1 \vee r_2; \eps,\eta)  \le K \eps^p \eta^p e^{- (2p)\alpha (t-r_1 \vee r_2)} e^{- (2p)\frac{\lambda}{\eta} (r_1 - r_2)}\ind_{(r_1 \ge r_2)} + K \eps^p \eta^p e^{-(2p)\alpha (t- r_1 \wedge r_2)}.
\end{equation}
Finally, a combination of the bounds obtained in equations \eqref{E:Second-der-X-W1-W2-Eq-1}, \eqref{E:Second-der-X-W1-W2-Eq-2}, \eqref{E:Second-der-X-W1-W2-Eq-3}, and \eqref{E:Second-der-X-W1-W2-Eq-4} above yields the required result.
\end{proof}

In the next section, we focus on establishing a bound for the term $\BE \left[\left|D_{r_1,r_2}^{W_2 W_2} X_t^\eps\right|^{2p}\right]$, where the process $D_{r_1,r_2}^{W_2 W_2} X_t^\eps$ is the solution of equation \eqref{E:Second-der-X-W2-W2}. The bound is outlined in Lemma \ref{L:Second-der-X-W1-W1}.
\subsection{Bound associated with the term $D_{r_1,r_2}^{W_2 W_2} X_t^\eps$}
For $t \ge r_1 \vee r_2,$ the second-order derivative processes $D_{r_1,r_2}^{W_2 W_2} X_t^\eps$ and $D_{r_1,r_2}^{W_2 W_2} Y_t^\eta$ satisfy the following stochastic integral equations:
\begin{equation}\label{E:Second-der-X-W2-W2}
\begin{aligned}
D_{r_1,r_2}^{W_2 W_2} X_t^\eps & =  \int_{r_1 \vee r_2}^t \left[c_{xx}(X_s^\eps,Y_s^\eta) D_{r_2}^{W_2}X_s^\eps \cdot D_{r_1}^{W_2}X_s^\eps +  c_{xy}(X_s^\eps,Y_s^\eta) D_{r_2}^{W_2}Y_s^\eta \cdot D_{r_1}^{W_2}X_s^\eps \right]ds \\
& \quad \quad + \int_{r_1 \vee r_2}^t \left[c_{xy}(X_s^\eps,Y_s^\eta) D_{r_2}^{W_2}X_s^\eps \cdot D_{r_1}^{W_2}Y_s^\eta +  c_{yy}(X_s^\eps,Y_s^\eta) D_{r_1}^{W_2}Y_s^\eta \cdot D_{r_2}^{W_2}Y_s^\eta \right]ds \\
& \quad \quad + \int_{r_1 \vee r_2}^t \left[c_{x}(X_s^\eps,Y_s^\eta) D_{r_1, r_2}^{W_2 W_2}X_s^\eps  +  c_{y}(X_s^\eps,Y_s^\eta) D_{r_1, r_2}^{W_2 W_2}Y_s^\eta \right]ds \\
& \quad \quad + \sqrt{\eps} \int_{r_1 \vee r_2}^t \sigma_{yy} (Y_s^\eta) D_{r_1}^{W_2}Y_s^\eta \cdot D_{r_2}^{W_2}Y_s^\eta \thinspace dW_s^1 + \sqrt{\eps} \int_{r_1 \vee r_2}^t \sigma_{y} (Y_s^\eta) D_{r_1, r_2}^{W_2 W_2}Y_s^\eta \thinspace dW_s^1,
\end{aligned}
\end{equation}
and 
\begin{equation}\label{E:Second-der-Y-W2-W2}
\begin{aligned}
D_{r_1,r_2}^{W_2 W_2} Y_t^\eta  & = \frac{1}{\sqrt{\eta}} \tau_y(Y_{r_2}^\eta) D_{r_1}^{W_2}Y_{r_2}^\eta + \frac{1}{\sqrt{\eta}} \tau_y(Y_{r_1}^\eta) D_{r_2}^{W_2}Y_{r_1}^\eta \\
& +  \frac{1}{\eta} \int_{r_1 \vee r_2}^t f_{yy}(Y_s^\eta)\cdot D_{r_2}^{W_2}Y_s^\eta \cdot D_{r_1}^{W_2}Y_s^\eta \thinspace ds  +  \frac{1}{\eta} \int_{r_1 \vee r_2}^t f_{y}(Y_s^\eta)\cdot D_{r_1, r_2}^{W_2 W_2}Y_s^\eta \thinspace ds \\
&+ \frac{1}{\sqrt{\eta}} \int_{r_1 \vee r_2}^t \tau_{yy}(Y_s^\eta)\cdot D_{r_2}^{W_2}Y_s^\eta \cdot D_{r_1}^{W_2}Y_s^\eta \thinspace dW_s^2  +  \frac{1}{\sqrt{\eta}} \int_{r_1 \vee r_2}^t \tau_{y}(Y_s^\eta)\cdot D_{r_1, r_2}^{W_2 W_2}Y_s^\eta \thinspace dW_s^2.
\end{aligned}
\end{equation}

\begin{lemma}\label{L:Second-der-Y-W2W2}
Let the process $D_{r_1,r_2}^{W_2 W_2} Y_t^\eta$ be the solution of equation \eqref{E:Second-der-Y-W2-W2} above. Then, for any $p \in \BN,$ and $t \ge r_1 \vee r_2,$ there exists a time-independent positive constant $K$ such that
$$\BE \left[\left|D_{r_1,r_2}^{W_2 W_2} Y_t^\eta\right|^{2p} \right] \le \frac{K}{\eta^p} e^{-\frac{(2p)\lambda}{\eta}(t-r_1 \vee r_2)} e^{- (2p)\frac{\lambda}{\eta} (r_1 \vee r_2 - r_1 \wedge r_2)} + \frac{K}{\eta^{2p}} e^{-\frac{(2p)\lambda}{\eta}(t- r_1 \wedge r_2)}.$$
\end{lemma}

\begin{proof}
Recalling the definition of the process $Z$ from equation \eqref{E:Z-process} and equation \eqref{E:Second-der-Y-W2-W2} above, we obtain
\begin{equation*}
\begin{aligned}
D_{r_1,r_2}^{W_2 W_2} Y_t^\eta & = \frac{1}{\sqrt{\eta}} Z_{t,r_1 \vee r_2}  \tau_y(Y_{r_2}^\eta) D_{r_1}^{W_2}Y_{r_2}^\eta + \frac{1}{\sqrt{\eta}} Z_{t,r_1 \vee r_2}  \tau_y(Y_{r_1}^\eta) D_{r_2}^{W_2}Y_{r_1}^\eta \\
& + \frac{1}{\sqrt{\eta}} \int_{r_1 \vee r_2}^t Z_{t,s} \tau_{yy}(Y_s^\eta) D_{r_2}^{W_2}Y_s^\eta \cdot D_{r_1}^{W_2}Y_s^\eta \thinspace dW_s^2 \\
& +  \int_{r_1 \vee r_2}^t Z_{t,s} \left\{ \frac{1}{\eta} f_{yy}(Y_s^\eta) D_{r_2}^{W_2}Y_s^\eta \cdot D_{r_1}^{W_2}Y_s^\eta - \frac{1}{\eta} \tau_{y}(Y_s^\eta) \tau_{yy}(Y_s^\eta) D_{r_2}^{W_2}Y_s^\eta \cdot D_{r_1}^{W_2}Y_s^\eta  \right\} ds. 
\end{aligned}
\end{equation*}
Increasing the power to $2p$ on both sides of the above equation followed by taking expectation, we have
\begin{equation}\label{E:Sec-der-Y-W1W1-Eq1}
\begin{aligned}
\BE  \left[(D_{r_1,r_2}^{W_2 W_2} Y_t^\eta)^{2p} \right]  & \le  \frac{1}{\eta^p}\|\tau_y\|_\infty \BE \left[ \left( Z_{t,r_1 \vee r_2} D_{r_1}^{W_2}Y_{r_2}^\eta \right)^{2p} \right] + \frac{1}{\eta^p}\|\tau_y\|_\infty \BE \left[ \left( Z_{t,r_1 \vee r_2} D_{r_2}^{W_2}Y_{r_1}^\eta \right)^{2p} \right] \\
& + \frac{K}{\eta^{2p}} \BE \left( \int_{r_1 \vee r_2}^t Z_{t,s} \left\{ f_{yy}(Y_s^\eta)  -  \tau_{y}(Y_s^\eta) \tau_{yy}(Y_s^\eta)\right\} D_{r_2}^{W_2}Y_s^\eta \cdot D_{r_1}^{W_2}Y_s^\eta \thinspace ds \right)^{2p} \\
& \qquad \qquad \qquad \qquad \quad  + \frac{K}{\eta^p} \BE \left( \int_{r_1 \vee r_2}^t Z_{t,s} \tau_{yy}(Y_s^\eta) D_{r_2}^{W_2}Y_s^\eta \cdot D_{r_1}^{W_2}Y_s^\eta \thinspace dW_s^2 \right)^{2p} \\
& \qquad \qquad \qquad \qquad \quad  \triangleq \mathsf{J}_1(t, r_1 \vee r_2; \eta) + \mathsf{J}_2(t, r_1 \vee r_2; \eta) + \mathsf{J}_3(t, r_1 \vee r_2; \eta).
\end{aligned}
\end{equation}
Next, for $\mathsf{J}_1(t, r_1 \vee r_2; \eta)$ , using H\"older's inequality, Assumption \ref{A:Derivatives-Assumption-f} and Lemmas \ref{L:Second-der-Integrating-factor}, \ref{L:1st-der-Y-W2}, we obtain
$\mathsf{J}_1(t, r_1 \vee r_2; \eta)  \le \frac{K}{\eta^p} e^{-\frac{(2p)\lambda}{\eta}(t-r_1 \vee r_2)} \left[ e^{- (2p)\frac{\lambda}{\eta} (r_1 - r_2)}\ind_{(r_1 \ge r_2)}   + e^{- (2p)\frac{\lambda}{\eta} (r_2 - r_1)}\ind_{(r_2 \ge r_1)} \right] \\
 = \frac{K}{\eta^p} e^{-\frac{(2p)\lambda}{\eta}(t-r_1 \vee r_2)} e^{- (2p)\frac{\lambda}{\eta} (r_1 \vee r_2 - r_1 \wedge r_2)}.$
For $\mathsf{J}_2(t, r_1 \vee r_2; \eta)$, using Assumption \ref{A:Derivatives-Assumption-f} and Lemmas \ref{L:1st-der-Y-W2}, \ref{L:Second-der-Integrating-factor}, we have
\begin{equation*}
\begin{aligned}
\mathsf{J}_2(t,& r_1 \vee r_2; \eta)  \le \frac{K}{\eta^{2p}} \int_{r_1 \vee r_2}^t \cdots \int_{ r_1 \vee r_2}^t  \prod_{i=1}^{2p} \left[  (\BE Z_{t,s_i}^{4p}) \right]^{\frac{1}{4p}}  \left[  \BE(D_{r_2}^{W_2}Y_{s_i}^\eta)^{8p} \right]^{\frac{1}{8p}} \left[  \BE(D_{r_1}^{W_2}Y_{s_i}^\eta)^{8p} \right]^{\frac{1}{8p}}ds_1 \cdots ds_{2p} \\  
& \qquad  \le \frac{K}{\eta^{2p}} \int_{r_1 \vee r_2}^t \cdots \int_{ r_1 \vee r_2}^t  \prod_{i=1}^{2p} e^{-\frac{\lambda}{\eta}(t-s_i)}\frac{1}{\eta} e^{-\frac{\lambda}{\eta}(s_i-r_2)} e^{-\frac{\lambda}{\eta}(s_i-r_1)} ds_1 \cdots ds_{2p} =  \frac{K}{\eta^{2p}} e^{-\frac{(2p)\lambda}{\eta}(t- r_1 \wedge r_2)}. 
\end{aligned}
\end{equation*}
Finally, for $\mathsf{J}_3(t, r_1 \vee r_2; \eta)$ in equation \eqref{E:Sec-der-Y-W1W1-Eq1}, we  use Lemma \ref{L:L2} to get $\mathsf{J}_3(t, r_1 \vee r_2; \eta)  \le \frac{K}{\eta^{2p}} e^{-\frac{(2p)\lambda}{\eta}(t- r_1 \wedge r_2)}.$
Putting these bounds together in equation \eqref{E:Sec-der-Y-W1W1-Eq1}, we obtain the required result.
\end{proof}

\begin{lemma}\label{L:Second-der-X-W2W2}
Let the process $D_{r_1,r_2}^{W_2 W_2} X_t^{\eps}$ be the solution of equation \eqref{E:Second-der-Y-W2-W2}. Then, for any $p \in \BN$ and $t \ge r_1 \vee r_2,$ there exists a time-independent positive constant $K$ such that
\begin{equation*}
\BE \left[\left|D_{r_1,r_2}^{W_2 W_2} X_t^{\eps}\right|^{2p} \right]  \le K \left[ \eps^{2p} +\eta^{2p} + \eps^{p}\eta^{p}\right] e^{-(2p)\alpha (t- r_1 \wedge r_2)}
+ K e^{-(2p)\alpha (t- r_1 \vee r_2)} e^{-(2p)\frac{\lambda}{\eta} (r_1 \vee r_2 - r_1 \wedge r_2)}.
\end{equation*}
\end{lemma}

\begin{proof}
We start solving equation \eqref{E:Second-der-Y-W2-W2} to get
\begin{equation*}
\begin{aligned}
D_{r_1,r_2}^{W_2 W_2} X_t^\eps & = \int_{r_1 \vee r_2}^t e^{\int_s^t c_{x}(X_u^\eps,Y_u^\eta) \thinspace du}\left\{ c_{xx}(X_s^\eps,Y_s^\eta) D_{r_2}^{W_2}X_s^\eps \cdot D_{r_1}^{W_2}X_s^\eps +  c_{xy}(X_s^\eps,Y_s^\eta) D_{r_2}^{W_2}Y_s^\eta \cdot D_{r_1}^{W_2}X_s^\eps \right. \\
& + \left. c_y(X_s^\eps,Y_s^\eta) D_{r_1, r_2}^{W_2 W_2}Y_s^\eta   + c_{xy}(X_s^\eps,Y_s^\eta) D_{r_2}^{W_2}X_s^\eps \cdot D_{r_1}^{W_2}Y_s^\eta + c_{yy}(X_s^\eps,Y_s^\eta) D_{r_2}^{W_2}Y_s^\eta \cdot D_{r_1}^{W_2}Y_s^\eta
\right\}ds   \\
& + \sqrt{\eps} \int_{r_1 \vee r_2}^t e^{\int_s^t c_{x}(X_u^\eps,Y_u^\eta) \thinspace du} \sigma_{yy} (Y_s^\eta) D_{r_1}^{W_2}Y_s^\eta \cdot D_{r_2}^{W_2}Y_s^\eta \thinspace dW_s^1 \\
& + \sqrt{\eps} \int_{r_1 \vee r_2}^t e^{\int_s^t c_{x}(X_u^\eps,Y_u^\eta) \thinspace du} \sigma_{y} (Y_s^\eta) D_{r_1, r_2}^{W_2 W_2}Y_s^\eta \thinspace dW_s^1.
\end{aligned}
\end{equation*}
Raising the power to $2p$, then taking the expectation on both sides followed by using the condition $c_x(x,y) < -\alpha$ (Assumption \ref{A:Derivatives-Assumption-c}) and the boundedness of the derivatives of the function $c$ (Assumption \ref{A:Derivatives-Assumption-c}), we have
\begin{equation}\label{E:Second-der-X-W2W2}
\begin{aligned}
\BE & \left[(D_{r_1,r_2}^{W_2 W_2} X_t^{\eps})^{2p} \right]  \le K \BE\left( \int_{r_1 \vee r_2}^t e^{-\alpha (t-s)} |D_{r_2}^{W_2}X_s^\eps \cdot D_{r_1}^{W_2}X_s^\eps | \thinspace ds \right)^{2p} \\
& \quad  + K \BE \left( \int_{r_1 \vee r_2}^t e^{-\alpha (t-s)} |D_{r_2}^{W_2}Y_s^\eta \cdot D_{r_1}^{W_2}X_s^\eps | \thinspace ds \right)^{2p}  + K \BE \left( \int_{r_1 \vee r_2}^t e^{-\alpha (t-s)} |D_{r_1, r_2}^{W_2 W_2}Y_s^\eta  | \thinspace ds \right)^{2p} \\
& \quad  + K \BE \left( \int_{r_1 \vee r_2}^t e^{-\alpha (t-s)} |D_{r_2}^{W_2}X_s^\eps \cdot D_{r_1}^{W_2}Y_s^\eta | \thinspace ds \right)^{2p} \\
& \quad  + K \BE\left( \int_{r_1 \vee r_2}^t e^{-\alpha (t-s)} |D_{r_2}^{W_2}Y_s^\eta \cdot D_{r_1}^{W_2}Y_s^\eta | \thinspace ds \right)^{2p} \\
& \quad + K \eps^p \thinspace \BE \left( \int_{r_1 \vee r_2}^t e^{\int_s^t c_{x}(X_u^\eps,Y_u^\eta)du} \sigma_{yy} (Y_s^\eta) D_{r_1}^{W_2}Y_s^\eta \cdot D_{r_2}^{W_2}Y_s^\eta \thinspace dW_s^1 \right)^{2p} \\
&  \quad + K \eps^p \thinspace  \BE\left( \int_{r_1 \vee r_2}^t e^{\int_s^t c_{x}(X_u^\eps,Y_u^\eta)du} \sigma_{y} (Y_s^\eta) D_{r_1, r_2}^{W_2 W_2}Y_s^\eta \thinspace dW_s^1 \right)^{2p}  \triangleq \sum_{i=1}^7 \mathsf{K}_i(t, r_1 \vee r_2; \eps, \eta).
\end{aligned}
\end{equation}
We finish the proof by putting the bounds together from Lemmas \ref{L:K1-2}, \ref{L:K3-4}, \ref{L:K5} and \ref{L:K6-7} below.
\end{proof}

Further, in Lemmas from \ref{L:K1-2} to \ref{L:K6-7}, we estimate the terms $\mathsf{K}_i(t, r_1 \vee r_2; \eps, \eta)$, $i=1,\cdots,7$ defined in equation \eqref{E:Second-der-X-W2W2} above.
\begin{lemma}\label{L:K1-2}
Let the terms $\mathsf{K}_1(t, r_1 \vee r_2; \eps, \eta)$ and $\mathsf{K}_2(t, r_1 \vee r_2; \eps, \eta)$ be defined in equation \eqref{E:Second-der-X-W2W2}. Then, for any $p \in \BN$ and $t \ge r_1 \vee r_2,$ there exists a time-independent positive constant $K$ such that
\begin{equation*}
\begin{aligned}
\mathsf{K}_1(t, r_1 \vee r_2; \eps, \eta) & \le K \left[\eta^{2p} + {\eps^{2p}}\right]  e^{-(2p) \alpha (t- r_1 \wedge r_2)}, \\
\mathsf{K}_2(t, r_1 \vee r_2; \eps, \eta) & \le K \left[ \eta^{2p} + {\eta}^p \eps^p \right] e^{- (2p)\alpha (t-r_1)} e^{- (2p)\frac{\lambda}{\eta} (r_1 - r_2)}\ind_{(r_1 \ge r_2)} + K \left[ \eta^{2p} + {\eta}^p \eps^p \right] e^{-(2p) \alpha (t- r_1)}.
\end{aligned}
\end{equation*}
\end{lemma}
\begin{proof}
Using H\"older's inequality, Lemma \ref{L:1st-der-X-W2} to control the term $\BE \left[|D_r^{W_2}X_t^{\varepsilon}|^{4p} \right]$, we have
\begin{equation*}
\begin{aligned}
\mathsf{K}_1(t,& r_1 \vee r_2; \eps, \eta) \le K \underbrace{\int_{r_1 \vee r_2}^t \cdots \int_{ r_1 \vee r_2}^t}_{2p\operatorname{-times}}  \prod_{i=1}^{2p} e^{-\alpha (t-s_i)}  \left[  \BE(D_{r_2}^{W_2}X_{s_i}^\eps)^{4p} \right]^{\frac{1}{4p}} \left[  \BE(D_{r_1}^{W_2}X_{s_i}^\eps)^{4p} \right]^{\frac{1}{4p}}ds_1 \cdots ds_{2p} \\
& \le K \int_{r_1 \vee r_2}^t \cdots \int_{ r_1 \vee r_2}^t  \prod_{i=1}^{2p} e^{-\alpha (t-s_i)} (\eps^{2p}+ \eta^{2p})^{\frac{1}{2p}} e^{-\alpha(s_i-r_2)} e^{-\alpha(s_i-r_1)} ds_1 \cdots ds_{2p} \\
& \le K (\eps^{2p}+ \eta^{2p}) \left( \int_{r_1 \vee r_2}^t e^{-\alpha(t-s)} e^{-\alpha(s-r_2)} e^{-\alpha(s-r_1)} \thinspace ds \right)^{2p} \le K \left[\eta^{2p} + {\eps^{2p}}\right]  e^{-(2p) \alpha (t- r_1 \wedge r_2)}.
\end{aligned}
\end{equation*}
For $\mathsf{K}_2(t, r_1 \vee r_2; \eps, \eta)$, we again use H\"older's inequality and Lemmas \ref{L:1st-der-X-W2}, \ref{L:1st-der-Y-W2} to obtain
\begin{equation*}
\begin{aligned}
\mathsf{K}_2&(t, r_1 \vee r_2; \eps, \eta)  \le K \int_{r_1 \vee r_2}^t \cdots \int_{ r_1 \vee r_2}^t  \prod_{i=1}^{2p} e^{-\alpha (t-s_i)} \frac{1}{\sqrt{\eta}} (\eps^{2p}+ \eta^{2p})^{\frac{1}{4p}} e^{-\frac{\lambda}{\eta}(s_i-r_2)} e^{-\alpha(s_i-r_1)} ds_1 \cdots ds_{2p} \\
& =  K \frac{(\eps^{p}+ \eta^{p})}{\eta^p} e^{-2p \alpha t} e^{2p \alpha r_1}  e^{2p \frac{\lambda}{ \eta} r_2}  \left(\int_{r_1 \vee r_2}^t e^{-\frac{\lambda}{\eta}s}ds \right)^{2p} \le  K {\eta^{p}(\eps^{p}+ \eta^{p})}  e^{-2p \alpha t} e^{2p \alpha r_1}  e^{2p \frac{\lambda}{ \eta} r_2}   e^{- 2p\frac{\lambda}{\eta}(r_1 \vee r_2)} \\
&  \le K \left[ \eta^{2p} + {\eta}^p \eps^p \right] e^{- (2p)\alpha (t-r_1)} e^{- (2p)\frac{\lambda}{\eta} (r_1 - r_2)}\ind_{(r_1 \ge r_2)} + K \left[ \eta^{2p} + {\eta}^p \eps^p \right] e^{-(2p) \alpha (t- r_1)}. 
\end{aligned}
\end{equation*}
Putting these bounds together, we complete the proof of the lemma.
\end{proof}

\begin{lemma}\label{L:K3-4}
Let the terms $\mathsf{K}_3(t, r_1 \vee r_2; \eps, \eta)$ and $\mathsf{K}_4(t, r_1 \vee r_2; \eps, \eta)$ be defined in equation \eqref{E:Second-der-X-W2W2}. Then, for any $p \in \BN$ and $t \ge r_1 \vee r_2,$ there exists a time-independent positive constant $K$ such that
\begin{equation*}
\begin{aligned}
\mathsf{K}_3(t, r_1 \vee r_2; \eps, \eta) & \le K e^{-(2p)\alpha (t-r_1 \vee r_2)} e^{-(2p)\frac{\lambda}{\eta} (r_1 \vee r_2 - r_1 \wedge r_2)}, \\
\mathsf{K}_4(t, r_1 \vee r_2; \eps, \eta) & \le K \left[ \eta^{2p} + {\eta}^p \eps^p \right] e^{- (2p)\alpha (t-r_2)} e^{- (2p)\frac{\lambda}{\eta} (r_2 - r_1)}\ind_{(r_2 \ge r_1)} + K \left[ \eta^{2p} + {\eta}^p \eps^p \right] e^{-(2p) \alpha (t- r_2)}.
\end{aligned}
\end{equation*}
\end{lemma}
\begin{proof}
For $\mathsf{K}_3(t, r_1 \vee r_2; \eps, \eta),$ using H\"older's inequality and Lemma \ref{L:Second-der-Y-W2W2}, we have
\begin{equation*}
\begin{aligned}
\mathsf{K}_3(t, r_1 \vee r_2; \eps, \eta) & \le K \int_{r_1 \vee r_2}^t \cdots \int_{ r_1 \vee r_2}^t  \prod_{i=1}^{2p} e^{-\alpha (t-s_i)}  \left[  \BE(D_{r_1, r_2}^{W_2 W_2}Y_{s_i}^\eta)^{2p} \right]^{\frac{1}{2p}} ds_1 \cdots ds_{2p} \\
& \le K \int_{r_1 \vee r_2}^t \cdots \int_{ r_1 \vee r_2}^t  \prod_{i=1}^{2p} e^{-\alpha (t-s_i)}  \left[ \frac{1}{\sqrt{\eta}} e^{-\frac{\lambda}{\eta}(s_i-r_1 \vee r_2)} e^{-\frac{\lambda}{\eta} (r_1 \vee r_2 - r_1 \wedge r_2)} \right. \\
& \left. \qquad \qquad \qquad \qquad \qquad \qquad \qquad \qquad \qquad \qquad  + \frac{1}{\eta} e^{-\frac{\lambda}{\eta}(s_i- r_1 \wedge r_2)} \right] ds_1 \cdots ds_{2p} \\
& \le \frac{K}{\eta^p} e^{-2p \alpha t} e^{2p \frac{\lambda}{\eta}(r_1 \vee r_2)}  e^{-(2p)\frac{\lambda}{\eta} (r_1 \vee r_2 - r_1 \wedge r_2)} \eta^{2p} e^{-2p\left( \frac{\lambda}{\eta} - \alpha \right)(r_1 \vee r_2)} \\
& \qquad \qquad \qquad \qquad \qquad \qquad \quad + \frac{K}{\eta^{2p}}e^{-2p \alpha t} e^{2p \frac{\lambda}{\eta}(r_1 \wedge r_2)}\eta^{2p} e^{-2p\left( \frac{\lambda}{\eta} - \alpha \right)(r_1 \vee r_2)},
\end{aligned}
\end{equation*}
where the last inequality is obtained using $\eta \ll 1$ and the property that the term $ \frac{\lambda}{\eta} - \alpha$ is a positive number. Further, using a simplification, we obtain
$\mathsf{K}_3(t, r_1 \vee r_2; \eps, \eta) \le K e^{-(2p)\alpha (t-r_1 \vee r_2)} e^{-(2p)\frac{\lambda}{\eta} (r_1 \vee r_2 - r_1 \wedge r_2)}.$
The bound for $\mathsf{K}_4(t, r_1 \vee r_2; \eps, \eta)$ can be obtained in a similar fashion. This finishes the proof of the lemma.
 \end{proof}

\begin{lemma}\label{L:K5}
Let the term $\mathsf{K}_5(t, r_1 \vee r_2; \eps, \eta)$ be defined in equation \eqref{E:Second-der-X-W2W2}. Then, for any $p \in \BN$ and $t \ge r_1 \vee r_2,$ there exists a time-independent positive constant $K$ such that 
\begin{equation*}
\begin{aligned}
\mathsf{K}_5(t, r_1 \vee r_2; \eps, \eta) & \le K e^{-(2p)\alpha (t-r_1 \vee r_2)} e^{-(2p)\frac{\lambda}{\eta} (r_1 \vee r_2 - r_1 \wedge r_2)}. 
\end{aligned}
\end{equation*}
\end{lemma}
\begin{proof}
H\"older's inequality and Lemma \ref{L:1st-der-Y-W2} to deal with the term $\BE \left[|D_r^{W_2}Y_t^{\eta}|^{4p} \right]$ yield
\begin{equation*}
\begin{aligned}
\mathsf{K}_5& (t, r_1 \vee r_2; \eps, \eta)  \le \frac{K}{\eta^{2p}} \int_{r_1 \vee r_2}^t \cdots \int_{ r_1 \vee r_2}^t  \prod_{i=1}^{2p} e^{-\alpha (t-s_i)}  e^{-\frac{\lambda}{\eta}(s_i-r_2)} e^{-\frac{\lambda}{\eta}(s_i-r_1)} ds_1 \cdots ds_{2p} \\
& \le \frac{K}{\eta^{2p}} e^{-2p \alpha t} e^{\frac{2p \lambda}{\eta}r_1} e^{\frac{2p \lambda}{\eta}r_2} e^{-\frac{2p \lambda}{\eta}(r_1 \vee r_2)} \left( \int_{r_1 \vee r_2}^t e^{-\left( \frac{\lambda}{\eta}-\alpha\right)s} ds \right)^{2p} \le K e^{-(2p)\alpha (t-r_1 \vee r_2)} e^{-(2p)\frac{\lambda}{\eta} (r_1 \vee r_2 - r_1 \wedge r_2)},
\end{aligned}
\end{equation*}
where, the last inequality is obtained using $\eta \ll 1$ and the property that the term $ \frac{\lambda}{\eta} - \alpha$ is a positive number. Hence, the lemma is proved.
\end{proof}
\begin{lemma}\label{L:K6-7}
Let $\mathsf{K}_6(t, r_1 \vee r_2; \eps, \eta)$ and $\mathsf{K}_7(t, r_1 \vee r_2; \eps, \eta)$ be defined in equation \eqref{E:Second-der-X-W2W2}. Then, for any $p \in \BN$ and $t \ge r_1 \vee r_2,$ there exists a time-independent positive constant $K$ such that
\begin{equation*}
\begin{aligned}
\mathsf{K}_6(t, r_1 \vee r_2; \eps, \eta) + \mathsf{K}_7(t, r_1 \vee r_2; \eps, \eta)  & \le K \left[\eps^p + \frac{\eps^p}{\eta^{p}} \right] e^{-(2p) \alpha (t- r_1 \vee r_2 )}  e^{-(2p)\frac{\lambda}{ \eta}(r_1 \vee r_2- r_1 \wedge r_2)}. \\
\end{aligned}
\end{equation*}
\end{lemma}

\begin{proof}
Let us first consider the term $\mathsf{K}_6(t, r_1 \vee r_2; \eps, \eta),$ then the bound for the term $\mathsf{K}_7(t, r_1 \vee r_2; \eps, \eta)$ can be obtained in a similar way. Using H\"older's inequality and martingale moment inequality followed by Assumption \ref{A:Derivatives-Assumption-c}, we have
\begin{equation*}
\begin{aligned}
& \mathsf{K}_6(t, r_1 \vee r_2; \eps, \eta) \\
&  = K \eps^p \BE \left[ \left( e^{\int_{r_1 \vee r_2}^t c_x(X_u^\varepsilon, Y_u^\eta) \thinspace du} \right)^{2p} \left( \int_{r_1 \vee r_2}^t e^{-\int_{r_1 \vee r_2}^s c_x(X_u^\varepsilon, Y_u^\eta) \thinspace du} \sigma_{yy} (Y_s^\eta) D_{r_1}^{W_2}Y_s^\eta \cdot D_{r_2}^{W_2}Y_s^\eta \thinspace dW_s^1 \right)^{2p} \right] \\
& \le K \eps^p \left[ \BE \left( e^{\int_{r_1 \vee r_2}^t c_x(X_u^\varepsilon, Y_u^\eta) \thinspace du} \right)^{4p} \right]^{\frac{1}{2}}  \left[ \BE \left( \int_{r_1 \vee r_2}^t e^{-\int_{r_1 \vee r_2}^s c_x(X_u^\varepsilon, Y_u^\eta) \thinspace du}\sigma_{yy} (Y_s^\eta) D_{r_1}^{W_2}Y_s^\eta \cdot D_{r_2}^{W_2}Y_s^\eta \thinspace dW_s^1 \right)^{4p}\right]^{\frac{1}{2}} \\
& \le K \eps^p e^{-(2p)\alpha (t-r_1 \vee r_2)} \left[ \BE \left( \int_{r_1 \vee r_2}^t e^{-2\int_{r_1 \vee r_2}^s c_x(X_u^\varepsilon, Y_u^\eta) \thinspace du} \left[\sigma_{yy} (Y_s^\eta) D_{r_1}^{W_2}Y_s^\eta \cdot D_{r_2}^{W_2}Y_s^\eta\right]^2 \thinspace ds \right)^{2p}\right]^{\frac{1}{2}}.
\end{aligned}
\end{equation*}
Next, using the boundedness of $\sigma_{yy}$ from Assumption \ref{A:Derivatives-Assumption-c}, the inequality $-c_x \le |c_x| \le K_{c_x}$, H\"older's inequality followed by Lemma \ref{L:1st-der-Y-W2} to control the term $\BE \left[|D_r^{W_2}Y_t^{\eta}|^{2p} \right]$, we obtain
\begin{equation*}
\begin{aligned}
& \mathsf{K}_6(t, r_1 \vee r_2; \eps, \eta)  \le K \eps^p e^{-(2p)\alpha (t-r_1 \vee r_2)} \left[ \left( \int_{r_1 \vee r_2}^t e^{2 K_{c_x} (s-r_1 \vee r_2)} \frac{1}{\eta^2} e^{- \frac{2\lambda}{\eta}(s-r_1)} e^{- \frac{2\lambda}{\eta}(s-r_2)} \thinspace ds \right)^{2p}\right]^{\frac{1}{2}} \\ 
& \qquad \qquad  \le  K \frac{\eps^p}{\eta^{2p}} e^{-(2p)\alpha (t-r_1 \vee r_2)} \left[  e^{-4p K_{c_x} (r_1 \vee r_2) } e^{\frac{4p \lambda r_1}{\eta}} e^{\frac{4p \lambda r_2}{\eta}}  \left(\int_{r_1 \vee r_2}^t e^{(2K_{c_x}-\frac{4 \lambda}{\eta})s} ds \right)^{2p}  \right]^{\frac{1}{2}}\\
& \qquad \qquad \le K \frac{\eps^p}{\eta^{2p}} e^{-(2p)\alpha (t-r_1 \vee r_2)} e^{-2p K_{c_x} (r_1 \vee r_2) } e^{\frac{-2p \lambda (r_1 \vee r_2)}{\eta}} e^{\frac{2p \lambda r_1}{\eta}} e^{\frac{2p \lambda r_2}{\eta}}  \left(\int_{r_1 \vee r_2}^t   e^{(2K_{c_x}-\frac{2 \lambda}{\eta})s}  ds \right)^p \\
& \qquad \qquad \le K \frac{\eps^p}{\eta^{2p}} e^{-(2p)\alpha (t-r_1 \vee r_2)} e^{-2p K_{c_x} (r_1 \vee r_2) } e^{\frac{-2p \lambda (r_1 \vee r_2)}{\eta}} e^{\frac{2p \lambda r_1}{\eta}} e^{\frac{2p \lambda r_2}{\eta}}  \left( \int_{r_1 \vee r_2}^t  e^{-2 \left(\frac{\lambda}{\eta}- K_{c_x} \right)u}du \right)^{p}.
\end{aligned}
\end{equation*}
For sufficiently small $\eta$, the term $ \frac{\lambda}{\eta} - K_{c_x}$ is a positive number, hence for fixed $\lambda$, $K_{c_x}$ and for some $\eta_0>0$ such that $0 < \eta \le \eta_0$, we have
\begin{equation*}
\begin{aligned}
\mathsf{K}_6(t, r_1 \vee r_2; \eps, \eta) & \le K \frac{\eps^p}{\eta^{2p}} e^{-(2p)\alpha (t-r_1 \vee r_2)} e^{-2p K_{c_x} (r_1 \vee r_2) } e^{\frac{-2p \lambda (r_1 \vee r_2)}{\eta}} e^{\frac{2p \lambda r_1}{\eta}} e^{\frac{2p \lambda r_2}{\eta}} \eta^p e^{-2p \left(\frac{\lambda}{\eta}- K_{c_x} \right)(r_1 \vee r_2)} \\
& = K \frac{\eps^p}{\eta^{p}} e^{-(2p) \alpha (t- r_1 \vee r_2 )}  e^{-(2p)\frac{\lambda}{ \eta}(r_1 \vee r_2- r_1 \wedge r_2)}.
\end{aligned}
\end{equation*}
Hence, the lemma is proved.
\end{proof}

Proposition \ref{P:Second-der-Prop} below is the main result of this subsection. This result is associated with the bound for the second-order derivative term $\|D^2X_t^\eps \otimes D^2X_t^\eps  \|_{\mathscr{H}^{\otimes 2}}^2$. We apply this result in the proof of Theorem \ref{T:Main-Result}.

\begin{proposition}\label{P:Second-der-Prop}
Let the process $X_t^\eps$ be the solution of the first equation in system \eqref{E:Multiscale-Diffusion-Main-Eq}. Then, for any $t \ge 0,$ there exists a time-independent positive constant $K$ such that 
$$\sup_{t \ge 0}\BE \left[\|D^2X_t^\eps \otimes D^2X_t^\eps  \|_{\mathscr{H}^{\otimes 2}}^2 \right] \le K \left( \eta^3 + \eps^3 + \eps^{\frac{3}{2}} \eta^{\frac{3}{2}} \right),$$
where $\otimes$ is the convolution operator defined in equation \eqref{E:Convolution}.
\end{proposition}

Before proving Proposition \ref{P:Second-der-Prop}, we estimate the bounds for the terms $\mathcal{J}_1(t;\eps,\eta)$, $\mathcal{J}_4(t;\eps,\eta)$, $\mathcal{J}_8(t;\eps,\eta)$, $\mathcal{J}_9(t;\eps,\eta)$ and $\mathcal{J}_{16}(t;\eps,\eta)$ defined in Lemmas \ref{L:Int-1}, \ref{L:Int-4}, \ref{L:Int-8}, \ref{L:Int-9} and \ref{L:Int-16} below, respectively. These results are proved in Section \ref{S:Integral-Lemmas}.

\begin{lemma}\label{L:Int-1}
For $0 \le u \le s \le w \le r \le t$, let the term $\mathcal{J}_1(t;\eps,\eta)$ be defined as follows:
$$\mathcal{J}_1(t;\eps,\eta)  \triangleq K \int_{u=0}^t \int_{s=u}^t \int_{w=s}^t \int_{r=w}^t (\eps+ \eta + \sqrt{\eps \eta})^4 e^{-2 \alpha (t-u)} e^{-\alpha(t-w)} e^{-\alpha (t-s)} dr \thinspace dw \thinspace ds \thinspace du. $$
Then, there exists a time-independent positive constant $K$ such that
$$ \mathcal{J}_1(t;\eps,\eta) \le K (\eps+ \eta + \sqrt{\eps \eta})^4.$$
\end{lemma}

\begin{lemma}\label{L:Int-4}
For $0 \le u \le s \le w \le r \le t$, let the term $\mathcal{J}_4(t;\eps,\eta)$ be defined as follows:
\begin{multline*}
\mathcal{J}_4(t;\eps,\eta)  \triangleq  K \int_{u=0}^t \int_{s=u}^t \int_{w=s}^t \int_{r=w}^t  (\eps+ \eta + \sqrt{\eps \eta})^2 e^{-2 \alpha (t-u)} e^{-\alpha(t-w)} e^{-\alpha(t-r)} e^{-\frac{\lambda}{\eta} (r-w)} \\
\times  e^{-\frac{\lambda}{\eta} (w-s)}     dr \thinspace dw \thinspace ds \thinspace du.  
\end{multline*}
Then, there exists a time-independent positive constant $K$ such that
$$ \mathcal{J}_4(t;\eps,\eta) \le K \thinspace \eta^2 \thinspace (\eps+ \eta + \sqrt{\eps \eta})^2 .$$
\end{lemma}

\begin{lemma}\label{L:Int-8}
For $0 \le u \le s \le w \le r \le t$, let the term $\mathcal{J}_8(t;\eps,\eta)$ be defined as follows:
\begin{multline*}
\mathcal{J}_8(t;\eps,\eta)  \triangleq K  \int_{u=0}^t \int_{s=u}^t \int_{w=s}^t \int_{r=w}^t  (\eps+ \eta + \sqrt{\eps \eta}) e^{-\alpha (t-u)} e^{-\alpha(t-s)} e^{- \alpha(t-w)}  e^{-\frac{\lambda}{\eta} (r-w)} e^{-\frac{\lambda}{\eta} (s-u)} \\
\times  e^{-\frac{\lambda}{\eta} (w-s)} dr \thinspace dw \thinspace ds \thinspace du. 
\end{multline*}
Then, there exists a time-independent positive constant $K$ such that
$$ \mathcal{J}_8(t;\eps,\eta) \le K  \thinspace \eta^2 \thinspace (\eps+ \eta + \sqrt{\eps \eta}).$$
\end{lemma}

\begin{lemma}\label{L:Int-9}
For $0 \le u \le s \le w \le r \le t$, let the term $\mathcal{J}_9(t;\eps,\eta)$ be defined as follows:
$$\mathcal{J}_9(t;\eps,\eta)  \triangleq K  \int_{u=0}^t \int_{s=u}^t \int_{w=s}^t \int_{r=w}^t(\eps+ \eta + \sqrt{\eps \eta})^3 e^{-\alpha (t-u)} e^{-\alpha(t-s)} e^{- \alpha(t-w)} e^{- \alpha(t-r)}  e^{-\frac{\lambda}{\eta} (r-u)} dr \thinspace dw \thinspace ds \thinspace du.$$
Then, there exists a time-independent positive constant $K$ such that
$$ \mathcal{J}_9(t;\eps,\eta) \le K \thinspace \eta^2 \thinspace  (\eps+ \eta + \sqrt{\eps \eta})^3.$$
\end{lemma}

\begin{lemma}\label{L:Int-16}
For $0 \le u \le s \le w \le r \le t$, let the term $\mathcal{J}_{16}(t;\eps,\eta)$ be defined as follows:
\begin{multline*}
\mathcal{J}_{16}(t;\eps,\eta)  \triangleq K \int_{u=0}^t \int_{s=u}^t \int_{w=s}^t \int_{r=w}^t  e^{-2 \alpha (t-r)} e^{- \alpha(t-s)} e^{- \alpha(t-w)}   e^{-\frac{\lambda}{\eta} (r-u)} e^{-\frac{\lambda}{\eta} (s-u)} e^{-\frac{\lambda}{\eta} (r-w)} \\
\times e^{-\frac{\lambda}{\eta} (w-s)}     dr \thinspace dw \thinspace ds \thinspace du.
\end{multline*}
Then, there exists a time-independent positive constant $K$ such that
$$ \mathcal{J}_{16}(t;\eps,\eta) \le K \eta^3.$$
\end{lemma}

We are now ready to prove Proposition \ref{P:Second-der-Prop}.
\begin{remark}
It is important to note that we provide the proof of Proposition \ref{P:Second-der-Prop} only for the case ${\sf{A}}_1 \triangleq \{(u,s,w,r) \in [0,t]^4: u \le s \le w \le r\}$ among the total $4!(=24)$ cases. The other cases provide the same bounds specified in Lemmas \ref{L:Int-1}, \ref{L:Int-4}, \ref{L:Int-8}, \ref{L:Int-9} and \ref{L:Int-16}.
\end{remark}

\begin{proof}[Proof of Proposition \ref{P:Second-der-Prop}]
Using equation \eqref{E:Matrix-Second-der} and the Cauchy-Schwarz inequality, we have
\begin{equation}\label{E:Second-der-prod-eq}
\begin{aligned}
\BE &  \left[\|D^2X_t^\eps \otimes D^2X_t^\eps  \|_{\mathscr{H}^{\otimes 2}}^2 \right] \\
&    = \sum_{i,j,k,p=1}^2 \int_{[0,t]^4} \BE \left( D_{u,v}^{W^i, W^k} X_t^{\eps} D_{u,w}^{W^k, W^j} X_t^{\eps} D_{s,v}^{W^i, W^p} X_t^{\eps} D_{s,w}^{W^p, W^j} X_t^{\eps} \right) du \thinspace ds \thinspace dv \thinspace dw \\
&   \le  \sum_{i,j,k,p=1}^2 \int_{[0,t]^4} \left( \BE |D_{u,v}^{W^i W^k}|^4 \right)^{\frac{1}{4}} \left( \BE |D_{u,w}^{W^k W^j}|^4 \right)^{\frac{1}{4}} \left( \BE |D_{s,v}^{W^i W^p}|^4 \right)^{\frac{1}{4}} \left( \BE |D_{s,w}^{W^p W^j}|^4 \right)^{\frac{1}{4}}  du \thinspace ds \thinspace dv \thinspace dw. 
\end{aligned}
\end{equation}
From Lemmas \ref{L:Second-der-X-W1-W1}, \ref{L:Second-der-X-W1-W2} and \ref{L:Second-der-X-W2W2}, we observe that the term which produces the largest bound is $\BE \left[ \left|D_{r_1,r_2}^{W_2 W_2} X_t^{\eps}\right|^{4}\right]$. Hence, for the above equation it is sufficient to obtain a bound for the term
$$\int_{[0,t]^4} \left( \BE |D_{u,r}^{W^2 W^2}X_t^{\eps}|^4 \right)^{\frac{1}{4}} \left( \BE |D_{u,s}^{W^2 W^2}X_t^{\eps}|^4 \right)^{\frac{1}{4}} \left( \BE |D_{w,r}^{W^2 W^2}X_t^{\eps}|^4 \right)^{\frac{1}{4}} \left( \BE |D_{w,s}^{W^2 W^2}X_t^{\eps}|^4 \right)^{\frac{1}{4}}  du \thinspace ds \thinspace dw \thinspace dr.$$
Moving in the direction of getting a bound for the above term, from Lemma \ref{L:Second-der-X-W2W2} for $p=2,$ we have
\begin{equation*}
\begin{aligned}
\left[ \BE \left| D_{r_1,r_2}^{W_2 W_2} X_t^{\eps}\right|^{4} \right]^{\frac{1}{4}} \le K(\eps+ \eta + \sqrt{\eps \eta})e^{-\alpha (t - r_1 \wedge r_2)} + K e^{-\alpha (t- r_1 \vee r_2)} e^{-\frac{\lambda}{\eta} (r_1 \vee r_2 - r_1 \wedge r_2)}. 
\end{aligned}
\end{equation*}
Using this bound for the set ${\sf{A}}_1$, one can obtain
\begin{equation*}
\begin{aligned}
& \int_{{\sf{A}}_1} \left( \BE |D_{u,r}^{W^2 W^2}X_t^{\eps}|^4 \right)^{\frac{1}{4}} \left( \BE |D_{u,s}^{W^2 W^2}X_t^{\eps}|^4 \right)^{\frac{1}{4}} \left( \BE |D_{w,r}^{W^2 W^2}X_t^{\eps}|^4 \right)^{\frac{1}{4}} \left( \BE |D_{w,s}^{W^2 W^2}X_t^{\eps}|^4 \right)^{\frac{1}{4}}  dr \thinspace dw \thinspace ds \thinspace du \\
&  \qquad  = K \int_{u=0}^t \int_{s=u}^t \int_{w=s}^t \int_{r=w}^t  \left\{ (\eps+ \eta + \sqrt{\eps \eta})e^{-\alpha (t - u \wedge r)} +  e^{-\alpha (t- u \vee r)} e^{-\frac{\lambda}{\eta} (u \vee r - u \wedge r)}  \right\} \times  \\
& \qquad \qquad \qquad \qquad \quad \qquad \qquad \left\{ (\eps+ \eta + \sqrt{\eps \eta})e^{-\alpha (t - u \wedge s)} +  e^{-\alpha (t- u \vee s)} e^{-\frac{\lambda}{\eta} (u \vee s - u \wedge s)}   \right\} \times \\
& \qquad \qquad \qquad \qquad \quad \qquad \qquad \left\{ (\eps+ \eta + \sqrt{\eps \eta})e^{-\alpha (t - w \wedge r)} +  e^{-\alpha (t- w \vee r)} e^{-\frac{\lambda}{\eta} (w \vee r - w \wedge r)}   \right\} \times \\
& \qquad \qquad \qquad \qquad  \quad \qquad \qquad \left\{ (\eps+ \eta + \sqrt{\eps \eta})e^{-\alpha (t - s \wedge w)} +  e^{-\alpha (t- s \vee w)} e^{-\frac{\lambda}{\eta} (s \vee w - s \wedge w)}   \right\} dr \thinspace dw \thinspace ds \thinspace du.
\end{aligned}
\end{equation*}
A simplification of the above equation gives
\begin{equation*}
\begin{aligned}
& \int_{{\sf{A}}_1} \left( \BE |D_{u,r}^{W^2 W^2}X_t^{\eps}|^4 \right)^{\frac{1}{4}} \left( \BE |D_{u,s}^{W^2 W^2}X_t^{\eps}|^4 \right)^{\frac{1}{4}} \left( \BE |D_{w,r}^{W^2 W^2}X_t^{\eps}|^4 \right)^{\frac{1}{4}} \left( \BE |D_{w,s}^{W^2 W^2}X_t^{\eps}|^4 \right)^{\frac{1}{4}}  dr \thinspace dw \thinspace ds \thinspace du \\
& \qquad = K \int_{u=0}^t \int_{s=u}^t \int_{w=s}^t \int_{r=w}^t \left[(\eps+ \eta + \sqrt{\eps \eta})^4 e^{-2 \alpha (t-u)} e^{-\alpha(t-w)} e^{-\alpha (t-s)} \right. \\
& \qquad \qquad \qquad \qquad \quad + (\eps+ \eta + \sqrt{\eps \eta})^3 e^{-2 \alpha (t-u)} e^{-2\alpha(t-w)} e^{-\frac{\lambda}{\eta} (w-s)} \\
& \qquad \qquad \qquad \qquad \quad + (\eps+ \eta + \sqrt{\eps \eta})^3 e^{-2 \alpha (t-u)} e^{-\alpha(t-s)} e^{-\alpha(t-r)} e^{-\frac{\lambda}{\eta} (r-w)} \\
& \qquad \qquad \qquad \qquad \quad + (\eps+ \eta + \sqrt{\eps \eta})^2 e^{-2 \alpha (t-u)} e^{-\alpha(t-w)} e^{-\alpha(t-r)} e^{-\frac{\lambda}{\eta} (r-w)} e^{-\frac{\lambda}{\eta} (w-s)} \\
& \qquad \qquad \qquad \qquad \quad  + (\eps+ \eta + \sqrt{\eps \eta})^3 e^{-\alpha (t-u)} e^{-\alpha(t-w)} e^{-2 \alpha(t-s)} e^{-\frac{\lambda}{\eta} (s-u)} \\
& \qquad \qquad \qquad \qquad \quad + (\eps+ \eta + \sqrt{\eps \eta})^2 e^{-\alpha (t-u)} e^{-2\alpha(t-w)} e^{- \alpha(t-s)}  e^{-\frac{\lambda}{\eta} (w-s)} e^{-\frac{\lambda}{\eta} (s-u)} \\
& \qquad \qquad \qquad \qquad \quad + (\eps+ \eta + \sqrt{\eps \eta})^2 e^{-\alpha (t-u)} e^{-2\alpha(t-s)} e^{- \alpha(t-r)}  e^{-\frac{\lambda}{\eta} (r-w)} e^{-\frac{\lambda}{\eta} (s-u)} \\
& \qquad \qquad \qquad \qquad \quad + (\eps+ \eta + \sqrt{\eps \eta}) e^{-\alpha (t-u)} e^{-\alpha(t-s)} e^{- \alpha(t-w)}  e^{-\frac{\lambda}{\eta} (r-w)} e^{-\frac{\lambda}{\eta} (s-u)} e^{-\frac{\lambda}{\eta} (w-s)} \\
& \qquad \qquad \qquad \qquad \quad + (\eps+ \eta + \sqrt{\eps \eta})^3 e^{-\alpha (t-u)} e^{-\alpha(t-s)} e^{- \alpha(t-w)} e^{- \alpha(t-r)}  e^{-\frac{\lambda}{\eta} (r-u)} \\
& \qquad \qquad \qquad \qquad \quad + (\eps+ \eta + \sqrt{\eps \eta})^2 e^{-\alpha (t-u)} e^{-\alpha(t-r)} e^{- 2\alpha(t-w)}  e^{-\frac{\lambda}{\eta} (r-u)} e^{-\frac{\lambda}{\eta} (w-s)} \\
& \qquad \qquad \qquad \qquad \quad + (\eps+ \eta + \sqrt{\eps \eta})^2 e^{-\alpha (t-u)} e^{-2 \alpha(t-r)} e^{- \alpha(t-s)}  e^{-\frac{\lambda}{\eta} (r-u)} e^{-\frac{\lambda}{\eta} (r-w)} \\
& \qquad \qquad \qquad \qquad \quad + (\eps+ \eta + \sqrt{\eps \eta}) e^{-\alpha (t-u)} e^{-2 \alpha(t-r)} e^{- \alpha(t-w)}  e^{-\frac{\lambda}{\eta} (r-u)} e^{-\frac{\lambda}{\eta} (r-w)} e^{-\frac{\lambda}{\eta} (w-s)} \\
& \qquad \qquad \qquad \qquad \quad + (\eps+ \eta + \sqrt{\eps \eta})^2 e^{-2 \alpha (t-s)} e^{- \alpha(t-r)} e^{- \alpha(t-w)}  e^{-\frac{\lambda}{\eta} (r-u)} e^{-\frac{\lambda}{\eta} (s-u)} \\
& \qquad \qquad \qquad \qquad \quad + (\eps+ \eta + \sqrt{\eps \eta}) e^{- \alpha (t-s)} e^{- \alpha(t-r)} e^{- 2\alpha(t-w)}  e^{-\frac{\lambda}{\eta} (r-u)} e^{-\frac{\lambda}{\eta} (s-u)} e^{-\frac{\lambda}{\eta} (w-s)} \\
& \qquad \qquad \qquad \qquad  \quad + (\eps+ \eta + \sqrt{\eps \eta}) e^{-2 \alpha (t-s)} e^{- 2\alpha(t-r)}   e^{-\frac{\lambda}{\eta} (r-u)} e^{-\frac{\lambda}{\eta} (s-u)} e^{-\frac{\lambda}{\eta} (r-w)} \\
& \qquad \qquad \qquad \qquad \quad  \left.+  e^{-2 \alpha (t-r)} e^{- \alpha(t-s)} e^{- \alpha(t-w)}   e^{-\frac{\lambda}{\eta} (r-u)} e^{-\frac{\lambda}{\eta} (s-u)} e^{-\frac{\lambda}{\eta} (r-w)} e^{-\frac{\lambda}{\eta} (w-s)} \right]  dr \thinspace dw \thinspace ds \thinspace du \\
& \qquad \qquad \qquad \qquad  \quad \triangleq \sum_{i=1}^{16} \mathcal{J}_i(t;\eps,\eta).
\end{aligned}
\end{equation*}
Finally, putting the bounds together from Lemmas \ref{L:Int-1}, \ref{L:Int-4}, \ref{L:Int-8}, \ref{L:Int-9} and \ref{L:Int-16}, we obtain the required bound.
\end{proof}

\section{Bounds associated with pre-limit expectation and variance}\label{S:Prelimit-Exp-Var}
In this section, we establish the bounds for the terms $\frac{\sigma_t}{\sqrt{\text{Var}(\theta_t^\eps)}}|\mu_t - \BE (\theta_t^\eps)|$ and $\BE(|\theta_t^\eps|)\left| 1- \frac{\sigma_t}{\sqrt{\text{Var}(\theta_t^\eps)}}\right|$, where $\sigma_t$ is the limiting variance defined in equation \eqref{E:Limiting-Variance}. These bounds are specified in Propositions \ref{P:Prelimit-1} and \ref{P:Prelimit-2}. To prove these results, we construct several Poisson equations \eqref{E:Poisson-equation}, \eqref{E:Pre-limit-Poisson-2}, \eqref{E:2-power-Poisson-equation} and \eqref{E:I2-Poisson-Eq} to handle the fluctuation terms.

\subsection{Bound for the term $\frac{\sigma_t}{\sqrt{\operatorname{Var}(\theta_t^\eps)}}|\mu_t - \BE (\theta_t^\eps)|$} In this section, we obtain a {\sc uit} bound for the term $\frac{\sigma_t}{\sqrt{\text{Var}(\theta_t^\eps)}}|\mu_t - \BE (\theta_t^\eps)|$. This bound is employed in the proof of our main result (Theorem \ref{T:Main-Result}). We start recalling the integral equations for the processes $X_t^\eps$, $\bar{X}_t$ and setting $\theta_t^\eps \triangleq \frac{X_t^\eps- \bar{X}_t}{\sqrt{\eps}}$ to have
\begin{equation*}
\begin{aligned}
\theta_t^\eps 
 = \frac{1}{\sqrt{\eps}}\int_0^t \left[  c(X_s^\eps, Y_s^\eta) -  \bar{c}({X}_s^\eps) \right] ds  + \frac{1}{\sqrt{\eps}} \int_0^t \left[\bar{c}({X}_s^\eps) - \bar{c}(\bar{X}_s) \right] ds  + \int_0^t  \sigma(X_s^\eps) \thinspace dW_s^1.
\end{aligned}
\end{equation*}
To control the fluctuation term $\int_0^t \left[ c(X_s^\eps, Y_s^\eta) - \bar{c}({X}_s^\eps) \right] ds$, we construct the Poisson equation: 
\begin{equation}\label{E:Poisson-equation}
\begin{aligned}
& \mathscr{L}_y \Upsilon(x,y)  = c(x,y)- \bar{c}(x), \quad \int_{\BR} \Upsilon(x,y)\mu_x(dy)=0,
\end{aligned}
\end{equation}
where the operator $\mathscr{L}_y$ is the infinitesimal generator of the process $Y^\eta$ and defined as follows: for any $x \in \BR$, 
\begin{equation}\label{E:Infinitesimal-Gen}
\mathscr{L}_y \Upsilon(x, \cdot) = f(\cdot)\frac{d}{dy}\Upsilon(x, \cdot)+ \frac{1}{2} \tau^2 (\cdot)\frac{d^2}{dy^2}\Upsilon(x,\cdot). 
\end{equation}
Applying It\^o's formula to the function $\Upsilon(x,y),$ we obtain
\begin{equation*}
\begin{aligned}
& \theta_t^\eps   = \frac{\eta}{\sqrt{\eps}}\left[\Upsilon(X_t^\eps, Y_t^\eta)- \Upsilon(X_0^\eps, Y_0^\eta \right] - \frac{\eta}{\sqrt{\eps}}\int_0^t c(X_s^\eps, Y_s^\eta)\Upsilon_x (X_s^\eps, Y_s^\eta) \thinspace ds  + \frac{1}{\sqrt{\eps}} \int_0^t \left[ \bar{c}({X}_s^\eps) - \bar{c}(\bar{X}_s) \right] ds\\
&  - \frac{1}{2}\eta \sqrt{\eps} \int_0^t \Upsilon_{xx} (X_s^\eps, Y_s^\eta) \sigma^2(X_s^\eps) \thinspace ds  + \int_0^t [1- \eta \Upsilon_x(X_s^\eps, Y_s^\eta)] \sigma(X_s^\eps) \thinspace dW_s^1 - \frac{\sqrt{\eta}}{\sqrt{\eps}} \int_0^t  \Upsilon_y(X_s^\eps, Y_s^\eta)\tau(Y_s^\eta) \thinspace dW_s^2  .
\end{aligned}
\end{equation*}
Using Taylor's theorem to the function $\bar{c}$: $\bar{c}({X}_s^\eps) = \bar{c}(\bar{X}_s) + \bar{c}'(\bar{X}_s)(X_s^\eps - \bar{X}_s) + \frac{1}{2} \bar{c}''(z_s^\eps)(X_s^\eps - \bar{X}_s)^2$, where, $z_s^\eps$ is a random point that lies on the line connecting the points $X_s^\eps$ and $\bar{X}_s,$ we have
\begin{equation*}
\begin{aligned}
\theta_t^\eps  & = \int_0^t \bar{c}'(\bar{X}_s) \theta_s^\eps \thinspace ds + \int_0^t \frac{1}{2 \sqrt{\eps}} \bar{c}''(z_s)(X_s^\eps - \bar{X}_s)^2 \thinspace ds  + \frac{\eta}{\sqrt{\eps}}\left[\Upsilon(X_t^\eps, Y_t^\eta)- \Upsilon(X_0^\eps, Y_0^\eta \right] \\
& \qquad \qquad \qquad \quad - \frac{\eta}{\sqrt{\eps}}\int_0^t c(X_s^\eps, Y_s^\eta)\Upsilon_x (X_s^\eps, Y_s^\eta) \thinspace ds - \frac{1}{2}\eta \sqrt{\eps} \int_0^t \Upsilon_{xx} (X_s^\eps, Y_s^\eta) \sigma^2(X_s^\eps) \thinspace ds \\
&  \qquad \qquad \quad \qquad \qquad \quad + \int_0^t [1- \eta \Upsilon_x(X_s^\eps, Y_s^\eta)] \sigma(X_s^\eps)\thinspace dW_s^1   - \frac{\sqrt{\eta}}{\sqrt{\eps}} \int_0^t  \Upsilon_y(X_s^\eps, Y_s^\eta)\tau(Y_s^\eta) \thinspace dW_s^2.
\end{aligned}
\end{equation*}
A simple application of Duhamel's principle (here, $\Psi(t)$ is the solution of equation \eqref{E:Semigroup-Equatn}) yields
\begin{multline}\label{E:theta-R-Q}
\theta_t^\eps   = \Psi(t) \int_0^t [\Psi(s)]^{-1} \sigma(X_s^\eps) \thinspace dW_s^1 - \sqrt{\frac{\eta}{\eps}} \Psi(t) \int_0^t [\Psi(s)]^{-1} \Upsilon_y(X_s^\eps, Y_s^\eta) \tau(Y_s^\eta) \thinspace dW_s^2 \\
 + \mathcal{Q}^\eps(t;\Psi)  + \mathcal{R}^\eps(t;\Psi), 
\end{multline}
and the remainder processes $\mathcal{R}^\eps(t;\Psi)$ and $\mathcal{Q}^\eps(t;\Psi)$ are defined as follows:
\begin{equation}\label{E:Remainder}
\begin{aligned}
\mathcal{R}^\eps(t;\Psi) & \triangleq \frac{\eta}{\sqrt{\eps}}\Psi(t)\left[\Upsilon(X_t^\eps, Y_t^\eta)- \Upsilon(X_0^\eps, Y_0^\eta) \right] - \eta \Psi(t) \int_0^t \Psi(s)^{-1}  \Upsilon_x(X_s^\eps, Y_s^\eta)\sigma(X_s^\eps) \thinspace dW_s^1 \\
& \qquad \quad - \frac{\eta}{\sqrt{\eps}} \Psi(t)\int_0^t \Psi(s)^{-1} c(X_s^\eps, Y_s^\eta)\Upsilon_x (X_s^\eps, Y_s^\eta) \thinspace ds \\
& \qquad \qquad \qquad \quad \qquad \qquad  - \frac{1}{2}\eta \sqrt{\eps} \Psi(t) \int_0^t \Psi(s)^{-1} \Upsilon_{xx} (X_s^\eps, Y_s^\eta)\sigma^2(X_s^\eps) \thinspace ds,  \quad \text{and} \\
\mathcal{Q}^\eps(t;\Psi) & \triangleq \Psi(t) \int_0^t [\Psi(s)]^{-1} \frac{1}{2 \sqrt{\eps}} \bar{c}''(z_s^\eps)(X_s^\eps - \bar{X}_s)^2 \thinspace ds.
\end{aligned}
\end{equation}
Next, we obtain {\sc uit} bounds for the terms $\BE \left[ \left| \mathcal{R}^\eps(t;\Psi) \right|^2 \right] $ and $\BE \left[ \left| \mathcal{Q}^\eps(t;\Psi) \right|^2 \right]$ in Lemmas \ref{L:Remainder-R} and \ref{L:Q-term}. 

\begin{lemma}\label{L:Remainder-R}
Let the process $\mathcal{R}^\eps(t;\Psi)$ be defined in equation \eqref{E:Remainder} and $\Upsilon(x,y)$ is the solution of the Poisson equation \eqref{E:Poisson-equation}. Then, for any $t \ge 0,$ there exists a time-independent positive constant $K$ such that
$$\sup_{t \geq 0} \BE \left[ \left| \mathcal{R}^\eps(t;\Psi) \right|^2 \right]  \le K \frac{\eta^2}{\eps}.$$
\end{lemma}
\begin{proof}
Taking square on both sides, we have
\begin{equation*}
\begin{aligned}
\left| \mathcal{R}^\eps(t;\Psi) \right|^2 & \le K \frac{\eta^2}{{\eps}}|\Psi(t)|^2\left|\Upsilon(X_t^\eps, Y_t^\eta)- \Upsilon(X_0^\eps, Y_0^\eta) \right|^2  + K \frac{\eta^2}{{\eps}}  \left|\int_0^t \Psi(t) \Psi(s)^{-1} c(X_s^\eps, Y_s^\eta)\Upsilon_x (X_s^\eps, Y_s^\eta)ds \right|^2 \\
& \qquad  + \frac{K}{2}\eta^2 {\eps}   \left|\int_0^t \Psi(t) \Psi(s)^{-1} \Upsilon_{xx} (X_s^\eps, Y_s^\eta)ds \right|^2  + K \eta^2  \left|\int_0^t \Psi(t) \Psi(s)^{-1}  \Upsilon_x(X_s^\eps, Y_s^\eta)dW_s^1\right|^2 \\
&  \triangleq J_1^{\eps,\delta}(t) + J_2^{\eps,\delta}(t) + J_3^{\eps,\delta}(t) + J_4^{\eps,\delta}(t).
\end{aligned}
\end{equation*}
For $J_1^{\eps,\delta}(t) $, using the property that the solution $\Upsilon(x,y)$ of the Poisson equation \eqref{E:Poisson-equation} grows at most polynomially in $y$, i.e., $|\Upsilon(X_t^\eps, Y_t^\eta)|^2 \le K (1+ |Y_t^\eta|^{2q})$, Assumption \ref{A:Semigroup-boundedness}, i.e., $\sup_{t \ge 0}|\Psi(t)|^2 < \infty$ and Lemma \ref{L:Moment-bound-Y}, we have
\begin{equation*}
\BE \left[J_1^{\eps,\delta}(t) \right] \le  K \frac{\eta^2}{{\eps}}\sup_{t \ge 0}|\Psi(t)|^2  \left(1+ \sup_{t \ge 0}\BE|Y_t^\eta|^{2q} \right) \le K \frac{\eta^2}{{\eps}}.
\end{equation*}
For $J_2^{\eps,\delta}(t)$, using polynomial growth of $\Upsilon_x(x,y)$ in $y$, Lemma \ref{L:Moment-bound-Y}, and the growth of $c(x,y)$ in variable $x$ from Assumption \ref{A:Derivatives-Assumption-c}, Lemma \ref{L:Moment-bound-X}, one can obtain
$\BE \left[J_2^{\eps,\delta}(t) \right] \le K \frac{\eta^2}{\eps} \sup_{t \ge 0} \left| \int_0^t \Psi(t) \Psi(s)^{-1} ds \right|^2 \le K \frac{\eta^2}{\eps},$
where the uniform boundedness of the quantity $\left| \int_0^t \Psi(t) \Psi(s)^{-1} ds \right|^2$ follows from Assumption \ref{A:Semigroup-boundedness}. Next, for $J_3^{\eps,\delta}(t)$ and $J_4^{\eps,\delta}(t)$, we note for $0 < \eps < 1,$ $1 < \frac{1}{\eps}$. Hence, using polynomial growth of $\Upsilon_x(x,y)$ in $y$, Lemma \ref{L:Moment-bound-Y}, martingale moment inequality and Assumption \ref{A:Semigroup-boundedness}, we get
\begin{equation*}
\BE \left[J_3^{\eps,\delta}(t) \right] + \BE \left[J_4^{\eps,\delta}(t) \right] \le K \frac{\eta^2}{\eps} \sup_{t \ge 0} \left| \int_0^t \Psi(t) \Psi(s)^{-1} ds \right|^2 + K {\eta^2} \sup_{t \ge 0} \left| \int_0^t \Psi(t)^2 \Psi(s)^{-2} ds \right| \le K \frac{\eta^2}{\eps}.
\end{equation*}
Putting all these bounds together, we finish the proof of the lemma.
\end{proof}

\begin{lemma}\label{L:Q-term}
Let the process $\mathcal{Q}^\eps(t;\Psi)$ be defined in equation \eqref{E:Remainder}. Then, for any $t \ge 0,$ there exists a time-independent positive constant $K$ such that
$$\sup_{t \geq 0} \BE \left[ \left| \mathcal{Q}^\eps(t;\Psi) \right|^2 \right]  \le K \left[\eps + \frac{\eta}{\eps} (\sqrt{\eps}+ \sqrt{\eta})\right].$$
\end{lemma}

\begin{proof}
We start noting 
\begin{multline*}
\left| \mathcal{Q}^\eps(t;\Psi) \right|^2 \le \frac{1}{4 \eps} \int_{[0,t]^2} \left( \prod_{i=1}^2 \Psi(t) \Psi(s_i)^{-1} \right) 
 \left[ \bar{c}''(z_1^{\eps}(s_1))(X_{s_1}^\eps - \bar{X}_{s_1})^4 + \bar{c}''(z_2^{\eps}(s_2))(X_{s_2}^\eps - \bar{X}_{s_2})^4 \right] ds_1 \thinspace ds_2.
\end{multline*}
Finally, using the boundedness of the function $\bar{c}''$, Lemma \ref{L:4th-power-LLN} to control the fourth-moment term and Assumption \ref{A:Semigroup-boundedness}, we get the required bound.
\end{proof}

We are now ready to give an estimate to the term $\frac{\sigma_t}{\sqrt{\operatorname{Var}(\theta_t^\eps)}}|\mu_t - \BE (\theta_t^\eps)|$ in Proposition \ref{P:Prelimit-1} below.
\begin{proposition}\label{P:Prelimit-1}
Let $\theta_t^\eps$ be the solution of equation \eqref{E:theta-R-Q}. Then, for any $t \ge 0,$ there exists a time-independent positive constant $K$ such that
$$\sup_{t \ge 0}|\mu_t - \BE (\theta_t^\eps)| \le K \frac{\eta}{\sqrt{\eps}} + K \sqrt{\eps + \frac{\eta}{\eps} (\sqrt{\eps}+ \sqrt{\eta})}.$$
\end{proposition}

\begin{proof}
The proof of the result follows by combining equation \eqref{E:theta-R-Q}, Lemmas \ref{L:Remainder-R} and \ref{L:Q-term} above.
\end{proof}

\subsection{Bound for the term $\BE(|\theta_t^\eps|)\left| 1- \frac{\sigma_t}{\sqrt{\operatorname{Var}(\theta_t^\eps)}}\right|$}
In this section, we obtain a {\sc uit} bound for the term $\BE(|\theta_t^\eps|)\left| 1- \frac{\sigma_t}{\sqrt{\operatorname{Var}(\theta_t^\eps)}}\right|$. This bound is employed in the proof of Theorem \ref{T:Main-Result}. Recalling the definition of the infinitesimal generator $\mathscr{L}_y$ from equation \eqref{E:Infinitesimal-Gen}, we start constructing Poisson equation:
\begin{equation}\label{E:Pre-limit-Poisson-2}
\mathscr{L}_y \Theta (x,y) = q(x,y)- \bar{q}(x), \qquad \int_\BR \Theta (x,y) \mu_x(dy) = 0,
\end{equation}
where, $q(x,y) \triangleq  \sigma^2(y) + \frac{\tau^2(y)}{\gamma^2} (\Upsilon_y(x,y))^2$ is defined in equation \eqref{E:q-function} and $\Upsilon(x,y)$ is the solution of Poisson equation \eqref{E:Poisson-equation}.

\begin{proposition}\label{P:Prelimit-2}
Let the process $\theta_t^\eps$ be the solution of \eqref{E:theta-R-Q} and the limiting variance $\sigma_t^2$ is defined in equation \eqref{E:Limiting-Variance}. Then, for any $t \ge 0,$ there exists a time-independent positive constant $K$ such that
\begin{equation*}
\begin{aligned}
\sup_{t \ge 0}\left|\sqrt{\operatorname{Var}(\theta_t^\eps)} - \sqrt{\sigma_t^2} \right|  & \le K \sqrt{\sqrt{\eta} +  \sqrt{\eps} + \left( \frac{\eta}{\eps} - \frac{1}{\gamma^2} \right)  +  \frac{\eta^2}{{\eps}}+  {\eps + \frac{\eta}{\eps} (\sqrt{\eps}+ \sqrt{\eta})}}.
\end{aligned}
\end{equation*}
\end{proposition}

\begin{proof}
We note that it is sufficient to obtain a bound for the quantity $|\text{Var}(\theta_t^\eps) - \sigma_t^2|.$ Using the definition of variance, we have
$|\text{Var}(\theta_t^\eps) - \sigma_t^2| \le |\BE (|\theta_t^\eps|^2) - \sigma_t^2|$. Hence, we first obtain a bound for the quantity $\BE (|\theta_t^\eps|^2)$. For this, from equation \eqref{E:theta-R-Q} and a triangle type inequality, we have
\begin{equation*}
\begin{aligned}
\BE (|\theta_t^\eps|^2)   \le K \BE \left[ \left| \Psi(t) \int_0^t [\Psi(s)]^{-1} \sigma(X_s^\eps) \thinspace dW_s^1 - \sqrt{\frac{\eta}{\eps}} \Psi(t) \int_0^t [\Psi(s)]^{-1} \Upsilon_y(X_s^\eps, Y_s^\eta) \tau(Y_s^\eta) \thinspace dW_s^2 \right|^2 \right] \\
 + K  \BE \left[ \left| \mathcal{Q}^\eps(t;\Psi) \right|^2 \right]  + K \BE \left[ \left| \mathcal{R}^\eps(t;\Psi) \right|^2 \right].
 \end{aligned}
\end{equation*}
In the above equation, writing $\sqrt{\frac{\eta}{\eps}}$ as $\sqrt{\frac{\eta}{\eps}-\frac{1}{\gamma^2}+ \frac{1}{\gamma^2}}$, then applying martingale moment inequality, Assumption \ref{A:Semigroup-boundedness} and Lemmas \ref{L:Remainder-R}, \ref{L:Q-term}, we have
\begin{equation*}
\begin{aligned}
\BE (|\theta_t^\eps|^2)  & \le \BE \int_0^t e^{\int_s^t 2 \bar{c}'(\bar{X}_u)du} {q}(X_s^\eps, Y_s^\eta) \thinspace ds  + K \left[\eps + \frac{\eta}{\eps} (\sqrt{\eps}+ \sqrt{\eta})\right] + K \frac{\eta^2}{{\eps}} \\
 & \qquad \qquad \qquad \qquad \qquad \qquad \qquad \quad + \left( \frac{\eta}{\eps}-\frac{1}{\gamma^2} \right) \BE \int_0^t e^{\int_s^t 2 \bar{c}'(\bar{X}_u)du} \tau^2(Y_s^\eta) ( \Upsilon_y)(X_s^\eps, Y_s^\eta)^2 \thinspace ds.   
\end{aligned}
\end{equation*}
The definition of $\sigma_t^2$ yields
\begin{equation*}
\begin{aligned}
\text{Var}(\theta_t^\eps) - \sigma_t^2 & = \BE (|\theta_t^\eps|^2)  - \sigma_t^2 - \left[ \BE (\theta_t^\eps) \right]^2  \le \BE \int_0^t e^{\int_s^t 2 \bar{c}'(\bar{X}_u)du} \left[{q}(X_s^\eps, Y_s^\eta) - \bar{q}(\bar{X}_s) \right]  ds  + K \frac{\eta^2}{{\eps}} - \left[ \BE (\theta_t^\eps) \right]^2  \\
& \qquad \quad \quad + K \left[\eps + \frac{\eta}{\eps} (\sqrt{\eps}+ \sqrt{\eta})\right]  + \left( \frac{\eta}{\eps}-\frac{1}{\gamma^2} \right) \BE \int_0^t e^{\int_s^t 2 \bar{c}'(\bar{X}_u)du} \tau^2(Y_s^\eta) ( \Upsilon_y)(X_s^\eps, Y_s^\eta)^2 \thinspace ds.    
\end{aligned}
\end{equation*}
Using Assumption \ref{A:Derivatives-Assumption-f} for the boundedness of $\tau$, the growth of $\Upsilon_y(x,y)$, Lemma \ref{L:Moment-bound-Y} and Lemma \ref{L:Q-difference}, we obtain
$\text{Var}(\theta_t^\eps) - \sigma_t^2   \le K \left[ \sqrt{\eta} +  \sqrt{\eps} + \left( \frac{\eta}{\eps} - \frac{1}{\gamma^2} \right)  +  \frac{\eta^2}{{\eps}}+  {\eps + \frac{\eta}{\eps} (\sqrt{\eps}+ \sqrt{\eta})}\right],$ 
which completes the proof of the lemma.
\end{proof}

\begin{lemma}\label{L:Q-difference}
Let $q(x,y)$ be the function defined in equation \eqref{E:q-function}. Then, for any $t \ge 0,$ there exists a time-independent positive constant $K$ such that
$$\sup_{t \ge 0} \BE \left| \int_0^t e^{\int_s^t 2 \bar{c}'(\bar{X}_u) \thinspace du} \left[ q(X_s^\eps, Y_s^\eta)- \bar{q}(\bar{X}_s) \right] ds \right|\le K (\sqrt{\eps}+ \sqrt{\eta}). $$
\end{lemma}
\begin{proof}
We start using the definition of $\bar{q}(\bar{X}_s)$ and a simple algebra to obtain
\begin{equation*}
\begin{aligned}
 \BE \left| \int_0^t e^{\int_s^t 2 \bar{c}'(\bar{X}_u) \thinspace du} \left[ q(X_s^\eps, Y_s^\eta)- \bar{q}(\bar{X}_s) \right] ds \right| & \le \BE \left| \int_0^t e^{\int_s^t 2 \bar{c}'(\bar{X}_u) \thinspace du} \left[ q(X_s^\eps, Y_s^\eta)- \int_{\BR} {q}(X_s^\eps, y) \mu_{X_s^\eps}(dy) \right] ds \right| \\
 & \quad + \BE \left| \int_0^t e^{\int_s^t 2 \bar{c}'(\bar{X}_u) \thinspace du}  \int_{\BR} {q}(X_s^\eps, y) \left[ \mu_{X_s^\eps} - \mu_{\bar{X}}\right](dy) \thinspace ds \right|  \\
 & \quad  +  \BE \left| \int_0^t e^{\int_s^t 2 \bar{c}'(\bar{X}_u) \thinspace du}  \int_{\BR} \left[ {q}(X_s^\eps, y) - {q}(\bar{X}_s, y) \right] \mu_{\bar{X}}(dy) \thinspace ds \right|\\
 & \le \BE \left| \int_0^t e^{-2 \alpha (t-s)} \left[ q(X_s^\eps, Y_s^\eta)- \int_{\BR} {q}(X_s^\eps, y) \mu_{X_s^\eps}(dy) \right] ds \right| \\
 &  \quad + \BE \left| \int_0^t e^{-2 \alpha (t-s)}  \int_{\BR} {q}(X_s^\eps, y) \left[ \mu_{X_s^\eps} - \mu_{\bar{X}}\right](dy) \thinspace ds \right|  \\
 &  \quad +  \BE \left| \int_0^t e^{-2 \alpha (t-s)}  \int_{\BR} \left[ {q}(X_s^\eps, y) - {q}(\bar{X}_s, y) \right] \mu_{\bar{X}}(dy) \thinspace ds \right|\\
  & \triangleq L_1^{\eps}(t) + L_2^{\eps}(t) + L_3^{\eps}(t),
\end{aligned}
\end{equation*}
where, the last inequality is obtained using $e^{\int_s^t 2 \bar{c}'(\bar{X}_u)du} \le e^{-2 \alpha (t-s)}$ (Assumption \ref{A:Derivatives-Assumption-c}). For $L_1^\eps(t)$ term, and the solution $\Theta(x,y)$ of Poisson equation \eqref{E:Pre-limit-Poisson-2}, applying It\^o's formula to the function $e^{2 \alpha s} \Theta(x,y)$ and using equations \eqref{E:E1} and \eqref{E:E2}, we obtain
\begin{equation*}
\begin{aligned}
 \int_0^t  e^{2 \alpha s} | \mathscr{L}_y \Theta (X_s^\eps, Y_s^\eta)|  ds  & \le  \eta  e^{2 \alpha t} |\Theta(X_t^\eps,Y_t^\eta)| + \eta | \Theta(X_0^\eps,Y_0^\eta)| + \eta \int_0^t 2 \alpha e^{2 \alpha s} |\Theta(X_s^\eps, Y_s^\eta)| \thinspace ds \\
 & + \eta  \int_0^t  e^{2 \alpha s} |\Theta_x(X_s^\eps, Y_s^\eta) c(X_s^\eps, Y_s^\eta) | \thinspace ds  + \eta  \sqrt{\eps} \left| \int_0^t  e^{2 \alpha s} \sigma(X_s^\eps) \Theta_x(X_s^\eps, Y_s^\eta) \thinspace dW_s^1 \right| \\
 & + \frac{\eta}{\sqrt{\eta}} \left| \int_0^t e^{2 \alpha s} \tau(Y_s^\eta) \Theta_y(X_s^\eps, Y_s^\eta) \thinspace dW_s^2 \right| + \frac{\eps \eta}{2} \int_0^t  e^{2 \alpha s} |\sigma^2(X_s^\eps) \Theta_{xx}(X_s^\eps, Y_s^\eta)| \thinspace ds.
\end{aligned}
\end{equation*}
Using martingale moment inequality, boundedness of the function $\sigma, \tau$ (Assumption \ref{A:Derivatives-Assumption-c} and \ref{A:Derivatives-Assumption-f}), the growth of $c(x,y)$ in variable $x$ from Assumption \ref{A:Derivatives-Assumption-c}, Lemma \ref{L:Moment-bound-X}, the growth of the function $\Theta$ in variable $y$ and Lemma \ref{L:Moment-bound-Y}, we get
\begin{equation*}
L_1^\eps(t) \triangleq  \BE \left| \int_0^t e^{-2 \alpha (t-s)} \left[ q(X_s^\eps, Y_s^\eta)- \int_{\BR} {q}(X_s^\eps, y) \mu_{X_s^\eps}(dy) \right] ds \right| \le K(\sqrt{\eps}+ \sqrt{\eta}).
\end{equation*}
Next, for the term $L_2^\eps(t)$, using H\"older's inequality, Lemma \ref{L:2nd-power} and \cite[Lemma 9]{gailus2017statistical}, we have
\begin{equation*}
\begin{aligned}
L_2^\eps(t) & \le \int_0^t e^{-2 \alpha (t-s)} \BE \left( \int_{\mathcal{Y}} (1+ |y|^q) |X_s^\eps - \bar{X}_s| \mu(dy) \right) ds  \le \int_0^t e^{-2 \alpha (t-s)} \BE |X_s^\eps - \bar{X}_s|\thinspace ds\\
& \qquad \qquad \qquad \qquad \qquad \qquad \qquad \qquad \quad \qquad \quad \le \int_0^t e^{-2 \alpha (t-s)} \sqrt{\BE |X_s^\eps - \bar{X}_s|^2}\thinspace ds \le K (\sqrt{\eps}+ \sqrt{\eta}).
\end{aligned}
\end{equation*}
Finally, for $L_3^\eps(t),$ using mean value theorem, Assumption \ref{A:Boundedness of sigma-tau}, and similar arguments to \cite[Lemma 10]{gailus2017statistical}, we have
$L_3^\eps(t) \le K (\sqrt{\eps}+ \sqrt{\eta}).$ Putting all these bounds together, we obtain the required bound.
\end{proof}

\subsection{Uniform-in-time bound for the process $|X_t^\eps - \bar{X}_{t}|^4$}
In this section, we establish a {\sc uit} bound for the term $\BE \left[|X_t^\eps - \bar{X}_{t}|^4 \right]$ in Lemma \ref{L:4th-power-LLN}. In the proof of this lemma, we require an estimate for the term $\BE \left[|X_t^\eps - \bar{X}_{t}|^2 \right]$ which is obtained in Lemma \ref{L:2nd-power}. To establish these bounds, we construct the following Poisson equations, where, the operator $\mathscr{L}_y$ is the infinitesimal generator of the process $Y^\eta$ defined in equation \eqref{E:Infinitesimal-Gen}:
\begin{equation}\label{E:2-power-Poisson-equation}
\mathscr{L}_y \Gamma (x,y) = G(x,y) \triangleq (x- \bar{x}, c(x,y)- \bar{c}(x)), \qquad \int_\BR \Gamma (x,y) \mu_x(dy) = 0,
\end{equation}
\begin{equation}\label{E:I2-Poisson-Eq}
\begin{aligned}
\mathscr{L}_y \Phi(x,y) & = F(x,y) \triangleq ((x- \bar{x})^2, c(x,y)-\bar{c}(x)), \quad \int_\BR \Phi (x,y) \mu_x(dy) = 0,\\
\mathscr{L}_y \Xi(x,y) & = H(x,y) \triangleq \Phi(x,y) \{c(x,y) - \bar{c}(x)\}, \qquad \int_\BR \Xi (x,y) \mu_x(dy) = 0.
\end{aligned}
\end{equation}
We note in the above equations the functions $G,\thinspace F$ and $H$ satisfy the centring condition, i.e., \\ $\int_{\BR} G(x,y) \mu_x(dy) = 0 =\int_{\BR} F(x,y) \mu_x(dy) = \int_{\BR} H(x,y) \mu_x(dy)$. This can be easily verified using integration by parts.
\begin{lemma}\label{L:2nd-power}
Let $X_t^\eps$ and $\bar{X}_t$ be the solutions of equations \eqref{E:Multiscale-Diffusion-Main-Eq} and \eqref{E:Semigroup-Equatn}. For any $t \ge 0,$ there exists a time-independent positive constant $K$ such that
$$\sup_{t \ge 0} \BE \left[|X_t^\eps - \bar{X}_{t}|^2 \right] \le K(\eta + \eps + \eps \eta) .$$
\end{lemma}

\begin{proof}
From the integral representations of the process $X_t^\eps$ and $\bar{X}_t$, we have
\begin{equation}\label{E:Difference-Eqn}
X_t^\eps - \bar{X}_{t} = \int_0^t [c(X_s^\eps , Y_s^\eta)-\bar{c}(X_s^\eps)] \thinspace ds + \int_0^t [\bar{c}(X_s^\eps)- \bar{c}(\bar{X}_{s})] \thinspace ds + \sqrt{\eps} \int_0^t \sigma(X_s^\eps) \thinspace dW_s^1.
\end{equation}
An application of It\^o's formula to the function $\varphi(t,x)= x^2 e^{\alpha t}$ yields
\begin{equation}\label{E:LLn-2nd-power}
\begin{aligned}
\left(X_t^\eps - \bar{X}_{t}\right)^2 e^{ \alpha t} & = \int_0^t  \alpha \left(X_s^\eps - \bar{X}_{s}\right)^2 e^{ \alpha s} \thinspace ds  + 2 \int_0^t \left(X_s^\eps - \bar{X}_{s}\right) e^{ \alpha s} \left\{ c(X_s^\eps, Y_s^\eta)- \overline{c}(X_s^\eps) \right\}ds \\
&  + 2 \int_0^t \left(X_s^\eps - \bar{X}_{s}\right) e^{ \alpha s} \left\{\bar{c}(X_s^\eps) - \bar{c}(\bar{X}_{s}) \right\}ds + 2 \int_0^t \left(X_s^\eps - \bar{X}_{s}\right) e^{ \alpha s} \sqrt{\eps} \sigma(X_s^\eps) dW_s^1 \\
& \qquad \qquad \qquad \qquad \qquad \qquad \qquad \quad \qquad \quad + \frac{1}{2} \eps \int_0^t 2  e^{ \alpha s} \sigma(X_s^\eps)^2 \thinspace ds \triangleq \sum_{i=1}^5 \mathcal{K}_i .  
\end{aligned}
\end{equation}
We now deal with each term separately. First, for $\mathcal{K}_3$, recalling the definition of $\bar{c}(x) \triangleq \int_{\BR} c(x,y) \thinspace \mu_x(dy)$, the assumption $\sup_{y \in \BR}[(x-z) \{c(x,y)- c(z,y)\}] \le -\alpha |x-z|^2$ (Assumption \ref{A:Derivatives-Assumption-c}) and the integral $\int_{\BR} \mu_x(dy)=1$, we have
\begin{equation*}
\begin{aligned}
\mathcal{K}_3  = 2\int_0^t e^{\alpha s} \left(X_s^\eps - \bar{X}_{s}\right) \int_\BR  \left\{ c(X_s^\eps, y) - c(\bar{X}_s, y) \right\} \mu_x (dy) \thinspace ds & \le  - 2 \alpha \int_0^t e^{\alpha s} |X_s^\eps - \bar{X}_{s}|^2 \int_\BR   \mu_x(dy) \thinspace ds\\
&= - 2 \alpha \int_0^t e^{\alpha s}   |X_s^\eps - \bar{X}_{s}|^2 \thinspace  ds.
\end{aligned}
\end{equation*}
Next, to estimate the challenging term $\mathcal{K}_2$ in equation \eqref{E:LLn-2nd-power}, we apply It\^o's formula to the function $e^{ \alpha t} \Gamma(X_t^\eps, Y_t^\eta)$ to obtain (here, $\Gamma(x,y)$ is the solution of Poisson equation \eqref{E:2-power-Poisson-equation})
\begin{equation*}
\begin{aligned}
&  e^{ \alpha t} \Gamma(X_t^\eps, Y_t^\eta) = \Gamma(X_0, Y_0) + \int_0^t \alpha e^{\alpha s} \Gamma(X_s^\eps, Y_s^\eta) \thinspace ds + \int_0^t e^{\alpha s}\Gamma_x(X_s^\eps, Y_s^\eta) c(X_s^\eps, Y_s^\eta) \thinspace ds\\
& + \sqrt{\eps} \int_0^t e^{\alpha s}\Gamma_x(X_s^\eps, Y_s^\eta) \sigma(X_s^\eps)\thinspace dW_s^1 + \frac{1}{\eta}\int_0^t e^{\alpha s}\Gamma_y (X_s^\eps, Y_s^\eta) f(Y_s^\eta)\thinspace ds\\
&  + \frac{1}{\sqrt{\eta}} \int_0^t e^{\alpha s}\Gamma_y (X_s^\eps, Y_s^\eta) \tau(Y_s^\eta)\thinspace dW_s^2 + \frac{\eps }{2} \int_0^t e^{\alpha s} \Gamma_{xx}(X_s^\eps, Y_s^\eta) \sigma(X_s^\eps)^2 \thinspace ds + \frac{1}{2 \eta } \int_0^t e^{\alpha s} \Gamma_{yy}(X_s^\eps, Y_s^\eta) \tau(Y_s^\eta)^2 \thinspace ds. 
\end{aligned}
\end{equation*}
The definition of the operator $\mathscr{L}_y$ from equation \eqref{E:Infinitesimal-Gen} and a rearrangement of the terms in above equation gives
\begin{equation*}
\begin{aligned}
& \mathcal{K}_2   \triangleq 2 \int_0^t \left(X_s^\eps - \bar{X}_{s}\right) e^{ \alpha s} \left\{ c(X_s^\eps, Y_s^\eta)- \overline{c}(X_s^\eps) \right\}ds =  2 \int_0^t e^{\alpha s} \mathscr{L}_y \Gamma(X_s^\eps, Y_s^\eta) \thinspace ds \\
&  = 2 \eta e^{t \alpha} \Gamma(X_t^\eps, Y_t^\eta) - 2 \eta \Gamma(X_0, Y_0) - 2 \eta  \int_0^t \alpha e^{\alpha s} \Gamma(X_s^\eps, Y_s^\eta) \thinspace ds   - 2 \eta \sqrt{\eps} \int_0^t e^{\alpha s}\Gamma_x(X_s^\eps, Y_s^\eta) \sigma(X_s^\eps) \thinspace dW_s^1 \\
& - 2 \eta \int_0^t e^{\alpha s} \Gamma_x(X_s^\eps, Y_s^\eta) c(X_s^\eps, Y_s^\eta) \thinspace ds  - \frac{2 \eta}{\sqrt{\eta}} \int_0^t e^{\alpha s}\Gamma_y (X_s^\eps, Y_s^\eta) \tau(Y_s^\eta) \thinspace dW_s^2 - \eta \eps \int_0^t e^{\alpha s} \Gamma_{xx}(X_s^\eps, Y_s^\eta) \sigma(X_s^\eps)^2 \thinspace ds.
\end{aligned}
\end{equation*}
Putting all these terms together in equation \eqref{E:LLn-2nd-power}, we have
\begin{equation*}
\begin{aligned}
& \left(X_t^\eps - \bar{X}_{t}\right)^2 e^{ \alpha t}  \le -  \alpha \int_0^t e^{\alpha s}  \left(X_s^\eps - \bar{X}_{s}\right)^2  ds + 2 \eta e^{\alpha t} \Gamma(X_t^\eps, Y_t^\eta) - 2 \eta \Gamma(X_0, Y_0) - 2 \eta  \int_0^t \alpha e^{\alpha s} \Gamma(X_s^\eps, Y_s^\eta) \thinspace ds \\
& - 2 \eta \int_0^t e^{\alpha s}\Gamma_x(X_s^\eps, Y_s^\eta) c(X_s^\eps, Y_s^\eta) \thinspace ds - 2 \eta \sqrt{\eps} \int_0^t e^{\alpha s}\Gamma_x(X_s^\eps, Y_s^\eta) \sigma(X_s^\eps) \thinspace dW_s^1 + \frac{\eps}{2}  \int_0^t 2  e^{ \alpha s} \sigma(X_s^\eps)^2 \thinspace ds \\
& - \frac{2 \eta}{\sqrt{\eta}} \int_0^t e^{\alpha s}\Gamma_y (X_s^\eps, Y_s^\eta) \tau(Y_s^\eta) \thinspace dW_s^2 - \eta \eps \int_0^t e^{\alpha s} \Gamma_{xx}(X_s^\eps, Y_s^\eta) \sigma(X_s^\eps)^2 \thinspace ds + 2 \int_0^t \left(X_s^\eps - \bar{X}_{s}\right) e^{ \alpha s} \sqrt{\eps} \sigma(X_s^\eps) \thinspace dW_s^1.
\end{aligned}
\end{equation*}
Finally, taking expectation on both sides, using the growth of $c(x,y)$ in variable $x$ from Assumption \ref{A:Derivatives-Assumption-c}, Lemma \ref{L:Moment-bound-X}, Assumption \ref{A:Derivatives-Assumption-c}, growth of the function $\Gamma(x,y)$ and Lemma \ref{L:Moment-bound-Y}, we complete the proof of lemma. 
\end{proof} 

In the next lemma, we give a uniform-in-time bound for the term $\BE \left[|X_t^\eps - \bar{X}_{t}|^4 \right].$
\begin{lemma}\label{L:4th-power-LLN}
Let $X_t^\eps$ and $\bar{X}_t$ be the solutions of equations \eqref{E:Multiscale-Diffusion-Main-Eq} and \eqref{E:Semigroup-Equatn}. Then, for any $t \ge 0,$ there exists a time-independent positive constant $K$ such that
$$\sup_{t \ge 0} \BE \left[|X_t^\eps - \bar{X}_{t}|^4 \right] \le K [\eps^2 + \eta^{\frac{3}{2}}+ \eta \sqrt{\eps}].$$
\end{lemma}
\begin{proof}
We start applying It\^o's formula to the function $\varphi(t,x)= x^4 e^{2 \alpha t}$ to get
\begin{equation}\label{E:LLn-4th-power}
\begin{aligned}
& \left(X_t^\eps - \bar{X}_{t}\right)^4 e^{2 \alpha t}  = 2 \alpha \int_0^t  \left(X_s^\eps - \bar{X}_{s}\right)^4 e^{2 \alpha s} \thinspace ds  + 4 \int_0^t \left(X_s^\eps - \bar{X}_{s}\right)^3 e^{2 \alpha s} \left\{ c(X_s^\eps, Y_s^\eta)- \bar{c}(X_s^\eps) \right\}ds \\
& \qquad \qquad \quad  + 4 \int_0^t \left(X_s^\eps - \bar{X}_{s}\right)^3 e^{2 \alpha s} \left\{\bar{c}(X_s^\eps) - \bar{c}(\bar{X}_{s}) \right\}ds + 4 \sqrt{\eps} \int_0^t \left(X_s^\eps - \bar{X}_{s}\right)^3 e^{2 \alpha s}  \sigma(X_s^\eps) dW_s^1 \\
& \qquad \qquad  \quad+ \frac{1}{2} \eps \int_0^t 12 (X_s^\eps - \bar{X}_s)^2 e^{2 \alpha s} \sigma(X_s^\eps)^2 ds \triangleq  \sum_{i=1}^5 \mathcal{I}_i .  
\end{aligned}
\end{equation}
For $\mathcal{I}_3$, recalling the definition of $\bar{c}(x) \triangleq \int_{\BR} c(x,y) \thinspace \mu_x(dy)$, the assumption $\sup_{y \in \BR}[(x-z) \{c(x,y)- c(z,y)\}] \le -\alpha |x-z|^2$ (Assumption \ref{A:Derivatives-Assumption-c}) and the integral $\int_{\BR} \mu_x(dy)=1$, we have
\begin{equation*}
\begin{aligned}
\mathcal{I}_3 & = 4 \int_0^t (X_s^\eps - \bar{X}_s)^2 e^{2 \alpha s} \int_{\BR} \left[ (X_s^\eps - \bar{X}_s) \{c(X_s^\eps, y)-c(\bar{X}_s, y)\} \right] \mu_x(dy) \thinspace ds  \le -4 \alpha \int_0^t e^{2 \alpha s} \left(X_s^\eps - \bar{X}_s\right)^4 ds.
\end{aligned}
\end{equation*}
Next, using the boundedness of $\sigma$ and Young's inequality, we obtain a sufficiently small $\zeta >0$ such that 
$\mathcal{I}_5 \le C \zeta \int_0^t  (X_s^\eps - \bar{X}_s)^4 e^{2 \alpha s}ds + C \eps^2 \int_0^t  e^{2 \alpha s} ds.$
Given that $\Phi(x,y)$ and $\Xi(x,y)$ are the solutions of Poisson equations in \eqref{E:I2-Poisson-Eq}, we apply Lemma \ref{L:4th-power-I2} to handle the term $\mathcal{I}_2$:
\begin{equation*}
\begin{aligned}
& \mathcal{I}_2 =  4 \int_0^t e^{2\alpha s} (X_s^\eps- \bar{X}_s) \mathscr{L}_y \Phi(X_s^\eps, Y_s^\eta) \thinspace ds \\
&   = 4 \eta e^{2 \alpha t} (X_t^\eps- \bar{X}_t)\Phi(X_t^\eps, Y_t^\eta)  - 4 \eta \int_0^t 2 \alpha e^{2 \alpha s} (X_s^\eps- \bar{X}_s)\Phi(X_s^\eps, Y_s^\eta) \thinspace ds \\
& - 4 \eta^2  e^{2 \alpha t} \Xi(X_t^\eps,Y_t^\eta) + \eta^2  \Xi(X_0^\eps,Y_0^\eta) + 4 \eta^2 \int_0^t 2 \alpha e^{2 \alpha s} \Xi(X_s^\eps, Y_s^\eta) \thinspace ds - 4{\eta \eps}\int_0^t e^{2\alpha s} \sigma^2(X_s^\eps)  \Phi_{x}(X_s^\eps, Y_s^\eta) \thinspace ds \\
 & + 4 \eta^2  \int_0^t  e^{2 \alpha s} \Xi_x(X_s^\eps, Y_s^\eta) c(X_s^\eps, Y_s^\eta) \thinspace ds  + 4 \eta^2  \sqrt{\eps}\int_0^t  e^{2 \alpha s} \sigma(X_s^\eps) \Xi_x(X_s^\eps, Y_s^\eta) \thinspace dW_s^1\\
 & + 4 \frac{\eta^2}{\sqrt{\eta}}\int_0^t e^{2 \alpha s} \tau(X_s^\eps) \Xi_y(X_s^\eps, Y_s^\eta) \thinspace dW_s^2 + 2{\eps \eta^2} \int_0^t  e^{2 \alpha s} \sigma^2(X_s^\eps) \Xi_{xx}(X_s^\eps, Y_s^\eta) \thinspace ds\\
& - 4 \eta  \int_0^t e^{2\alpha s} \Phi(X_s^\eps, Y_s^\eta) \{\bar{c}(X_s^\eps) - \bar{c}(\bar{X}_s) \} \thinspace ds  - 4  \sqrt{\eps} \eta \int_0^t e^{2\alpha s} \sigma(X_s^\eps) \Phi(X_s^\eps, Y_s^\eta) \thinspace dW_s^1 \\
& - 4  \eta  \int_0^t e^{2\alpha s} (X_s^\eps- \bar{X}_s) \Phi_x(X_s^\eps, Y_s^\eta)c(X_s^\eps, Y_s^\eta) \thinspace ds -  4 \eta \sqrt{\eps}\int_0^t e^{2\alpha s} (X_s^\eps- \bar{X}_s) \sigma(X_s^\eps) \Phi_x(X_s^\eps, Y_s^\eta) \thinspace dW_s^1 \\
& - 4 \frac{\eta}{\sqrt{\eta}}\int_0^t e^{2\alpha s} (X_s^\eps- \bar{X}_s) \tau(X_s^\eps) \Phi_y(X_s^\eps, Y_s^\eta) \thinspace dW_s^2  - 2{\eps \eta}\int_0^t e^{2\alpha s} (X_s^\eps- \bar{X}_s) \sigma^2(X_s^\eps) \Phi_{xx}(X_s^\eps, Y_s^\eta) \thinspace ds.
\end{aligned}
\end{equation*}
In equation \eqref{E:LLn-4th-power}, putting all the terms together followed by taking expectation, we have
\begin{equation*}
\begin{aligned}
& e^{2 \alpha t} \BE\left[ \left(X_t^\eps - \bar{X}_{t}\right)^4\right]    = - 2 \alpha  \int_0^t e^{2 \alpha s}  \BE\left[ \left(X_s^\eps - \bar{X}_{s}\right)^4\right]  ds  +  K \zeta \int_0^t   e^{2 \alpha s} \BE\left[ \left(X_s^\eps - \bar{X}_{s}\right)^4\right]  ds + K \eps^2 \int_0^t  e^{2 \alpha s} ds \\
& \quad + 4 \eta e^{2 \alpha t} \BE \left[ (X_t^\eps- \bar{X}_t)\Phi(X_t^\eps, Y_t^\eta)\right]  - 4 \eta \BE \int_0^t 2 \alpha e^{2 \alpha s} (X_s^\eps- \bar{X}_s)\Phi(X_s^\eps, Y_s^\eta) \thinspace ds \\
& \quad - 4 \eta^2  e^{2 \alpha t} \BE \left[ \Xi(X_t^\eps,Y_t^\eta) \right] + \eta^2  \BE \left[ \Xi(X_0^\eps,Y_0^\eta) \right] + 8 \alpha \eta^2 \BE \int_0^t e^{2 \alpha s} \Xi(X_s^\eps, Y_s^\eta) \thinspace ds \\
\end{aligned}
\end{equation*}
\begin{equation*}
\begin{aligned}
 & \quad + 4 \eta^2 \BE \int_0^t  e^{2 \alpha s} \Xi_x(X_s^\eps, Y_s^\eta) c(X_s^\eps, Y_s^\eta) \thinspace ds  + 2{\eps \eta^2} \BE \int_0^t  e^{2 \alpha s} \sigma^2(X_s^\eps) \Xi_{xx}(X_s^\eps, Y_s^\eta) \thinspace ds\\
& \quad - 4 \eta \BE  \int_0^t e^{2\alpha s} \Phi(X_s^\eps, Y_s^\eta) \{\bar{c}(X_s^\eps) - \bar{c}(\bar{X}_s) \} \thinspace ds  - 4  \eta \BE  \int_0^t e^{2\alpha s} (X_s^\eps- \bar{X}_s) \Phi_x(X_s^\eps, Y_s^\eta)c(X_s^\eps, Y_s^\eta) \thinspace ds  \\
&  \quad - 2 {\eps \eta} \BE \int_0^t e^{2\alpha s} (X_s^\eps- \bar{X}_s) \sigma^2(X_s^\eps) \Phi_{xx}(X_s^\eps, Y_s^\eta) \thinspace ds  - 4{\eta \eps} \BE \int_0^t e^{2\alpha s} \sigma^2(X_s^\eps)  \Phi_{x}(X_s^\eps, Y_s^\eta) \thinspace ds.
\end{aligned}
\end{equation*}
Applying H\"older's inequality and Assumption \ref{A:Derivatives-Assumption-c}, we get
\begin{equation*}
\begin{aligned}
e^{2 \alpha t} & \BE\left[ \left(X_t^\eps - \bar{X}_{t}\right)^4\right]  \le ( - 2 \alpha + K \zeta)  \int_0^t e^{2 \alpha s} \BE \left[ \left(X_s^\eps - \bar{X}_{s}\right)^4 \right] ds + K \eps^2 \int_0^t  e^{2 \alpha s} ds \\
& + 4 \eta e^{2 \alpha t} \left[ \BE |X_t^\eps- \bar{X}_t|^2 \right]^{\frac{1}{2}} \left[ \BE |\Phi(X_t^\eps, Y_t^\eta)|^2 \right]^{\frac{1}{2}}  + 8 \eta \int_0^t  \alpha e^{2 \alpha s}\left[ \BE |X_s^\eps- \bar{X}_s|^2 \right]^{\frac{1}{2}} \left[ \BE |\Phi(X_s^\eps, Y_s^\eta)|^2 \right]^{\frac{1}{2}} \thinspace ds \\
& + 4 \eta^2  e^{2 \alpha t} \BE |\Xi(X_t^\eps,Y_t^\eta)| + \eta^2  \BE |\Xi(X_0^\eps,Y_0^\eta)| + 8 \eta^2 \BE \int_0^t \alpha e^{2 \alpha s} |\Xi(X_s^\eps, Y_s^\eta)| \thinspace ds \\
 & + 4 \eta^2 \BE \int_0^t  e^{2 \alpha s} \Xi_x(X_s^\eps, Y_s^\eta) c(X_s^\eps, Y_s^\eta)ds  + 2 {\eps \eta^2} \BE \int_0^t  e^{2 \alpha s} \sigma^2(X_s^\eps) \Xi_{xx}(X_s^\eps, Y_s^\eta) \thinspace ds\\
& - 4 \eta \BE  \int_0^t e^{2\alpha s} \Phi(X_s^\eps, Y_s^\eta) \{\bar{c}(X_s^\eps) - \bar{c}(\bar{X}_s) \} \thinspace ds   + K{\eta \eps} \int_0^t e^{2\alpha s} \sigma^2(X_s^\eps) \BE| \Phi_{x}(X_s^\eps, Y_s^\eta)| \thinspace ds\\
&  + 4  \eta \int_0^t e^{2\alpha s} \left[ \BE |X_s^\eps- \bar{X}_s|^2 \right]^{\frac{1}{2}} \left[ \BE |\Phi_x(X_s^\eps, Y_s^\eta)c(X_s^\eps, Y_s^\eta)|^2 \right]^{\frac{1}{2}} \thinspace ds  \\
&  + K {\eps \eta} \int_0^t e^{2\alpha s} \left[ \BE |X_s^\eps- \bar{X}_s|^2 \right]^{\frac{1}{2}} \left[ \BE |\Phi_{xx}(X_s^\eps, Y_s^\eta)|^2 \right]^{\frac{1}{2}} \thinspace ds.
\end{aligned}
\end{equation*}
Finally, using the growth of the functions $\Phi(x,y)$, $\Xi(x,y)$, $c(x,y)$, and Lemmas \ref{L:2nd-power}, \ref{L:Moment-bound-Y}, \ref{L:Moment-bound-X}, we get the required bound.
\end{proof}

\section{Appendix}\label{S:Appendix}
\subsection{Proof of Auxiliary Results}
\begin{lemma}\label{L:Moment-bound-Y}
Let $Y_t^\eta$ be the solution of the second equation in system \eqref{E:Multiscale-Diffusion-Main-Eq}. For any $p \in \BN,$ there exists a time-independent positive constant $K$ such that
$$\sup_{t \ge 0} \BE [|Y_t^\eta|^p]  \le K.$$
\end{lemma}
\begin{proof}
We start noting that it is sufficient to prove the lemma only for even integer powers; for odd powers, one can apply H\"older's inequality. Applying It\^o's formula to the function $\varphi(t,x)= e^{\frac{tp \lambda^*}{2 \eta}}x^p$, we have
\begin{equation*}
\begin{aligned}
e^{\frac{tp \lambda^*}{2 \eta}}\left(Y_t^\eta \right)^p &= (y_0)^p + \int_0^t \frac{p \lambda^*}{2 \eta} e^{\frac{s p \lambda^*}{2 \eta}}\left(Y_s^\eta \right)^p \thinspace ds +  \int_0^t p e^{\frac{s p \lambda^*}{2 \eta}} \left(Y_s^\eta\right)^{p-1} \frac{1}{\eta} f(Y_s^\eta) \thinspace ds \\
& \qquad \qquad \quad + \int_0^t p e^{\frac{s p \lambda^*}{2 \eta}} \left(Y_s^\eta \right)^{p-1} \frac{1}{\sqrt{\eta}} \tau(Y_s^\eta)  \thinspace dW_s^2 + \frac{p(p-1)}{2} \int_0^t e^{\frac{s p \lambda^*}{2 \eta}} \left(Y_s^\eta \right)^{p-2} \frac{1}{{\eta}} \tau(Y_s^\eta)^2  \thinspace ds\\
& = \left(y_0 \right)^p + \int_0^t p e^{\frac{s p \lambda^*}{2 \eta}} \left[ \frac{\lambda^*}{2 \eta} \left(Y_s^\eta \right)^2 + Y_s^\eta \cdot \frac{1}{\eta} f(Y_s^\eta) \right] \left(Y_s^\eta \right)^{p-2} ds \\
 & \qquad \qquad \quad + \int_0^t p e^{\frac{s p \lambda^*}{2 \eta}} \left(Y_s^\eta\right)^{p-1} \frac{1}{\sqrt{\eta}} \tau(Y_s^\eta)  \thinspace dW_s^2 + \frac{p(p-1)}{2} \int_0^t e^{\frac{s p \lambda^*}{2 \eta}} \left(Y_s^\eta\right)^{p-2} \frac{1}{{\eta}} \tau(Y_s^\eta)^2  \thinspace ds.
\end{aligned}
\end{equation*}
Using the condition $y \cdot f(y) \le -\lambda^* |y|^2$ (Assumption \ref{A:Derivatives-Assumption-f}) and the uniform boundedness of the function $\tau$ (Assumption \ref{A:Derivatives-Assumption-f}), we get
\begin{equation*}
e^{\frac{tp \lambda^*}{2 \eta}}\left(Y_t^\eta \right)^p \le \left(y_0 \right)^p - \int_0^t \frac{p \lambda^*}{2 \eta} e^{\frac{s p \lambda^*}{2 \eta}} \left(Y_s^\eta\right)^p ds + \frac{K}{\eta} \int_0^t e^{\frac{s p \lambda^*}{2 \eta}} \left(Y_s^\eta\right)^{p-2} ds + \frac{p}{\sqrt{\eta}}\int_0^t e^{\frac{s p \lambda^*}{2 \eta}} \tau(Y_s^\eta) \left(Y_s^\eta\right)^{p-1} dW_s^2.
\end{equation*}
Next, using Young's inequality $K |Y_s^\eta|^{p-2}\le C_1 \zeta |Y_s^\eta|^p + C_2$ for some sufficiently small $\zeta>0$ followed by taking expectation, we have
\begin{equation*}
\begin{aligned}
\BE \left[ \left(Y_t^\eta \right)^p \right] & \le (y_0)^p e^{\frac{-tp \lambda^*}{2 \eta}}+ \left(- \frac{p \lambda^*}{2 \eta}+ \frac{\zeta}{\eta} K\right) e^{\frac{-tp \lambda^*}{2 \eta}}\int_0^t e^{\frac{sp \lambda^*}{2 \eta}}\left( Y_s^\eta\right)^p ds + \frac{K}{\eta} e^{\frac{-tp \lambda^*}{2 \eta}} \int_0^t e^{\frac{sp \lambda^*}{2 \eta}} ds \\
& \le (y_0)^p e^{\frac{-tp \lambda^*}{2 \eta}} + \frac{K}{\eta}\int_0^t e^{- \frac{p \lambda^*}{2 \eta}(t-s)} ds \le K.
\end{aligned}
\end{equation*}
Hence, the lemma is proved.
\end{proof}

\begin{lemma}\label{L:Moment-bound-X}
Let $X_t^\eps$ be the solution of the first equation in system \eqref{E:Multiscale-Diffusion-Main-Eq}. For any $t \ge 0,$ and $p \in \BN,$ there exists a time-independent positive constant $K$ such that
$$\sup_{t \ge 0} \BE\left[ |X_t^\eps|^p \right] \le K.$$
\end{lemma}
\begin{proof}
The proof of the lemma is a simple application of It\^o's formula to the function $\varphi(t,x) = x^p e^{\frac{p \alpha t}{2}}$, Young's inequality and Assumption \ref{A:Derivatives-Assumption-c}.
\end{proof}

The following result is used in the proof of Lemma \ref{L:4th-power-LLN}. 
\begin{lemma}\label{L:4th-power-I2}
Let the term $\mathcal{I}_2$ be defined in equation \eqref{E:LLn-4th-power}. Then, for the solutions $\Phi(x,y)$ and $\Xi(x,y)$ of Poisson equations \eqref{E:I2-Poisson-Eq}, we have the following representation of $\mathcal{I}_2:$
\begin{equation*}
\begin{aligned}
 \mathcal{I}_2 & =   \int_0^t e^{2\alpha s} (X_s^\eps- \bar{X}_s) \mathscr{L}_y \Phi(X_s^\eps, Y_s^\eta) \thinspace ds \\
&  = \eta e^{2 \alpha t} (X_t^\eps- \bar{X}_t)\Phi(X_t^\eps, Y_t^\eta)  - \eta \int_0^t 2 \alpha e^{2 \alpha s} (X_s^\eps- \bar{X}_s)\Phi(X_s^\eps, Y_s^\eta) \thinspace ds \\
&  - \eta^2  e^{2 \alpha t} \Xi(X_t^\eps,Y_t^\eta) + \eta^2  \Xi(X_0^\eps,Y_0^\eta) + \eta^2 \int_0^t 2 \alpha e^{2 \alpha s} \Xi(X_s^\eps, Y_s^\eta) \thinspace ds - {\eta \eps}\int_0^t e^{2\alpha s} \sigma^2(X_s^\eps)  \Phi_{x}(X_s^\eps, Y_s^\eta) \thinspace ds \\
 &  + \eta^2  \int_0^t  e^{2 \alpha s} \Xi_x(X_s^\eps, Y_s^\eta) c(X_s^\eps, Y_s^\eta) \thinspace ds  + \eta^2  \sqrt{\eps}\int_0^t  e^{2 \alpha s} \sigma(X_s^\eps) \Xi_x(X_s^\eps, Y_s^\eta) \thinspace dW_s^1\\
 &  + \frac{\eta^2}{\sqrt{\eta}}\int_0^t e^{2 \alpha s} \tau(X_s^\eps) \Xi_y(X_s^\eps, Y_s^\eta) \thinspace dW_s^2 + \frac{\eps \eta^2}{2} \int_0^t  e^{2 \alpha s} \sigma^2(X_s^\eps) \Xi_{xx}(X_s^\eps, Y_s^\eta) \thinspace ds\\
\end{aligned}
\end{equation*}
\begin{equation*}
\begin{aligned}
&  - \eta  \int_0^t e^{2\alpha s} \Phi(X_s^\eps, Y_s^\eta) \{\bar{c}(X_s^\eps) - \bar{c}(\bar{X}_s) \} \thinspace ds  -  \sqrt{\eps} \eta \int_0^t e^{2\alpha s} \sigma(X_s^\eps) \Phi(X_s^\eps, Y_s^\eta) \thinspace dW_s^1 \\
&  -  \eta  \int_0^t e^{2\alpha s} (X_s^\eps- \bar{X}_s) \Phi_x(X_s^\eps, Y_s^\eta)c(X_s^\eps, Y_s^\eta) \thinspace ds - \eta \sqrt{\eps}\int_0^t e^{2\alpha s} (X_s^\eps- \bar{X}_s) \sigma(X_s^\eps) \Phi_x(X_s^\eps, Y_s^\eta) \thinspace dW_s^1 \\
&  - \frac{\eta}{\sqrt{\eta}}\int_0^t e^{2\alpha s} (X_s^\eps- \bar{X}_s) \tau(X_s^\eps) \Phi_y(X_s^\eps, Y_s^\eta) \thinspace dW_s^2  - \frac{\eps \eta}{2}\int_0^t e^{2\alpha s} (X_s^\eps- \bar{X}_s) \sigma^2(X_s^\eps) \Phi_{xx}(X_s^\eps, Y_s^\eta) \thinspace ds.
\end{aligned}
\end{equation*}
\end{lemma}

\begin{proof}
Recalling the definition of the process $\mathcal{I}_2$ from equation \eqref{E:LLn-4th-power}, we have
$\mathcal{I}_2= 4 \int_0^t (X_s^\eps - \bar{X}_{s}) e^{2 \alpha s} (X_s^\eps - \bar{X}_{s})^2\left\{ c(X_s^\eps, Y_s^\eta)- \bar{c}(X_s^\eps) \right\}ds.$
For the solution $\Phi(x,y)$ of equation \eqref{E:I2-Poisson-Eq}, we apply It\^o's formula to the function $\phi(t,x,y) = e^{2 \alpha t} (x-z) \Phi(x,y)$ for fixed $z$ (for the processes $X^\eps$ and $Y^\eta$,  $\phi(t, X_t^\eps, Y_t^\eta) = e^{2 \alpha t} (X_t^\eps- \bar{X}_t)\Phi(X_t^\eps, Y_t^\eta)$ ) to get
\begin{equation*}
\begin{aligned}
e^{2 \alpha t} & (X_t^\eps- \bar{X}_t)\Phi(X_t^\eps, Y_t^\eta)  = 2 \alpha \int_0^t  e^{2 \alpha s} (X_s^\eps- \bar{X}_s)\Phi(X_s^\eps, Y_s^\eta) \thinspace ds \\
& \qquad \qquad + \int_0^t e^{2\alpha s} \Phi(X_s^\eps, Y_s^\eta) \thinspace d[X_s^\eps- \bar{X}_s] + \int_0^t e^{2\alpha s} (X_s^\eps- \bar{X}_s) \Phi_x(X_s^\eps, Y_s^\eta)  \thinspace dX_s^\eps\\
& \qquad \qquad + \int_0^t e^{2\alpha s} (X_s^\eps- \bar{X}_s) \Phi_y(X_s^\eps, Y_s^\eta) \thinspace dY_s^\eta + \frac{\eps}{2}\int_0^t e^{2\alpha s} (X_s^\eps- \bar{X}_s) \sigma^2(X_s^\eps) \Phi_{xx}(X_s^\eps, Y_s^\eta) \thinspace ds \\
& \qquad \qquad + {\eps}\int_0^t e^{2\alpha s} \sigma^2(X_s^\eps)  \Phi_{x}(X_s^\eps, Y_s^\eta) \thinspace ds + \frac{1}{2 \eta} \int_0^t e^{2\alpha s} (X_s^\eps- \bar{X}_s) \tau^2(Y_s^\eta) \Phi_{yy}(X_s^\eps, Y_s^\eta) \thinspace ds,
\end{aligned}
\end{equation*}
where, we use the following derivative expressions in It\^o's formula:\\ $\phi_x(t,x,y)  = e^{2 \alpha t} \left[ \Phi (x,y) + (x-z)\Phi_x(x,y) \right], \thinspace \phi_y(t,x,y)= e^{2 \alpha t}  (x-z)\Phi_y(x,y), \thinspace \phi_{xx}(t,x,y)  = \\ e^{2 \alpha t} \left[ (x-z) \Phi_{xx} (x,y) + 2\Phi_x(x,y) \right], \phi_{yy}(t,x,y)= e^{2 \alpha t}  (x-z)\Phi_{yy}(x,y).$
Next, the integral representation of the processes $X_t^\eps-\bar{X}_t$, $X_t^\eps$ and $Y_t^\eta,$ and a simple algebra yield
\begin{equation*}
\begin{aligned}
& e^{2 \alpha t} (X_t^\eps- \bar{X}_t)  \Phi(X_t^\eps, Y_t^\eta) \\
& \qquad  = 2 \alpha \int_0^t  e^{2 \alpha s} (X_s^\eps- \bar{X}_s)\Phi(X_s^\eps, Y_s^\eta) \thinspace ds  +  \int_0^t e^{2\alpha s} \Phi(X_s^\eps, Y_s^\eta) \{c(X_s^\eps, Y_s^\eta) - \bar{c}(X_s^\eps) \} \thinspace ds \\
& \qquad +   \int_0^t e^{2\alpha s} \Phi(X_s^\eps, Y_s^\eta) \{\bar{c}(X_s^\eps) - \bar{c}(\bar{X}_s) \} \thinspace ds  +  \sqrt{\eps} \int_0^t e^{2\alpha s} \sigma(X_s^\eps) \Phi(X_s^\eps, Y_s^\eta) \thinspace dW_s^1 \\
& \qquad  +  \int_0^t e^{2\alpha s} (X_s^\eps- \bar{X}_s) \Phi_x(X_s^\eps, Y_s^\eta)c(X_s^\eps, Y_s^\eta) \thinspace ds + \sqrt{\eps}\int_0^t e^{2\alpha s} (X_s^\eps- \bar{X}_s) \sigma(X_s^\eps) \Phi_x(X_s^\eps, Y_s^\eta) \thinspace dW_s^1 \\
& \qquad  + \frac{1}{\eta} \int_0^t e^{2\alpha s} (X_s^\eps- \bar{X}_s) \Phi_y(X_s^\eps, Y_s^\eta)f(Y_s^\eta) \thinspace ds + \frac{1}{\sqrt{\eta}}\int_0^t e^{2\alpha s} (X_s^\eps- \bar{X}_s) \tau(Y_s^\eta) \Phi_y(X_s^\eps, Y_s^\eta) \thinspace dW_s^2 \\
& \qquad  + \frac{\eps}{2}\int_0^t e^{2\alpha s} (X_s^\eps- \bar{X}_s) \sigma^2(X_s^\eps) \Phi_{xx}(X_s^\eps, Y_s^\eta) \thinspace ds  + {\eps}\int_0^t e^{2\alpha s} \sigma^2(X_s^\eps)  \Phi_{x}(X_s^\eps, Y_s^\eta)\thinspace ds \\
& \qquad  + \frac{1}{2 \eta} \int_0^t e^{2\alpha s} (X_s^\eps- \bar{X}_s) \tau^2(Y_s^\eta) \Phi_{yy}(X_s^\eps, Y_s^\eta)\thinspace ds.
\end{aligned}
\end{equation*}
The definition of the operator $\mathscr{L}_y$ defined in \eqref{E:Infinitesimal-Gen} and a rearrangement of the above terms give
\begin{equation*}
\begin{aligned}
 \int_0^t e^{2\alpha s} & (X_s^\eps- \bar{X}_s)  \mathscr{L}_y \Phi(X_s^\eps, Y_s^\eta) \thinspace ds =   \int_0^t e^{2\alpha s} (X_s^\eps- \bar{X}_s)^3 \left\{ c(X_s^\eps, Y_s^\eta)- \overline{c}(X_s^\eps) \right\} \thinspace ds \\
 & = \eta \thinspace e^{2 \alpha t} (X_t^\eps- \bar{X}_t)\Phi(X_t^\eps, Y_t^\eta)  - \eta \int_0^t 2 \alpha e^{2 \alpha s} (X_s^\eps- \bar{X}_s)\Phi(X_s^\eps, Y_s^\eta) \thinspace ds \\
& - \eta \int_0^t e^{2\alpha s} \Phi(X_s^\eps, Y_s^\eta) \{c(X_s^\eps, Y_s^\eta) - \bar{c}(X_s^\eps) \} \thinspace ds  - \eta  \int_0^t e^{2\alpha s} \Phi(X_s^\eps, Y_s^\eta) \{\bar{c}(X_s^\eps) - \bar{c}(\bar{X}_s) \} \thinspace ds  \\
& -  \sqrt{\eps} \eta \int_0^t e^{2\alpha s} \sigma(X_s^\eps) \Phi(X_s^\eps, Y_s^\eta) \thinspace dW_s^1  -  \eta  \int_0^t e^{2\alpha s} (X_s^\eps- \bar{X}_s) \Phi_x(X_s^\eps, Y_s^\eta)c(X_s^\eps, Y_s^\eta) \thinspace ds\\
& - \eta \sqrt{\eps}\int_0^t e^{2\alpha s} (X_s^\eps- \bar{X}_s) \sigma(X_s^\eps) \Phi_x(X_s^\eps, Y_s^\eta) \thinspace dW_s^1  - \frac{\eta}{\sqrt{\eta}}\int_0^t e^{2\alpha s} (X_s^\eps- \bar{X}_s) \tau(Y_s^\eta) \Phi_y(X_s^\eps, Y_s^\eta) \thinspace dW_s^2 \\
& - \frac{\eps \eta}{2}\int_0^t e^{2\alpha s} (X_s^\eps- \bar{X}_s) \sigma^2(X_s^\eps) \Phi_{xx}(X_s^\eps, Y_s^\eta) \thinspace ds  - {\eta \eps}\int_0^t e^{2\alpha s}  \sigma^2(X_s^\eps) \Phi_{x}(X_s^\eps, Y_s^\eta) \thinspace ds.
\end{aligned}
\end{equation*}
Finally, for the term $\int_0^t e^{2\alpha s}  \Phi(X_s^\eps, Y_s^\eta) \{c(X_s^\eps, Y_s^\eta) - \bar{c}(X_s^\eps) \} ds$ in the above equation, we use Lemma \ref{L:4th-Power-I22} below to get the required expression for the term $\mathcal{I}_2$, which completes the proof of the lemma.
\end{proof}

\begin{lemma}\label{L:4th-Power-I22}
Let $\Phi(x,y)$ and $\Xi(x,y)$ be the solutions of the first and second Poisson equations in \eqref{E:I2-Poisson-Eq}. Then,
\begin{equation*}
\begin{aligned}
 \int_0^t e^{2\alpha s} & \Phi(X_s^\eps, Y_s^\eta) \{c(X_s^\eps, Y_s^\eta) - \bar{c}(X_s^\eps) \} \thinspace ds = \int_0^t  e^{2 \alpha s} \mathscr{L}_y \Xi (X_s^\eps, Y_s^\eta) \thinspace ds \\
 & \qquad \qquad = \eta \thinspace  e^{2 \alpha t} \Xi(X_t^\eps,Y_t^\eta) - \eta  \Xi(X_0^\eps,Y_0^\eta) - \eta \int_0^t 2 \alpha e^{2 \alpha s} \Xi(X_s^\eps, Y_s^\eta) \thinspace ds \\
 & \qquad \qquad  - \eta  \int_0^t  e^{2 \alpha s} \Xi_x(X_s^\eps, Y_s^\eta) c(X_s^\eps, Y_s^\eta) \thinspace ds  - \eta  \sqrt{\eps}\int_0^t  e^{2 \alpha s} \sigma(X_s^\eps) \Xi_x(X_s^\eps, Y_s^\eta) \thinspace dW_s^1\\
 & \qquad \qquad  - \frac{\eta}{\sqrt{\eta}}\int_0^t e^{2 \alpha s} \tau(Y_s^\eta) \Xi_y(X_s^\eps, Y_s^\eta) \thinspace dW_s^2 - \frac{\eps \eta}{2} \int_0^t  e^{2 \alpha s} \sigma^2(X_s^\eps) \Xi_{xx}(X_s^\eps, Y_s^\eta) \thinspace ds.
\end{aligned}
\end{equation*}
\end{lemma}
\begin{proof}
We apply It\^o's formula to the function $\phi(x,y) = e^{2 \alpha t} \thinspace \Xi(x,y)$ to obtain
\begin{equation}\label{E:E1}
\begin{aligned}
e^{2 \alpha t} \Xi(X_t^\eps,Y_t^\eta) & = \Xi(X_0^\eps,Y_0^\eta) + 2 \alpha \int_0^t  e^{2 \alpha s} \Xi(X_s^\eps, Y_s^\eta) \thinspace  ds \\
 & \qquad \qquad +  \int_0^t  e^{2 \alpha s} \Xi_x(X_s^\eps, Y_s^\eta) c(X_s^\eps, Y_s^\eta) \thinspace  ds    +  \sqrt{\eps}\int_0^t  e^{2 \alpha s} \sigma(X_s^\eps) \Xi_x(X_s^\eps, Y_s^\eta) \thinspace  dW_s^1\\
 & \qquad \qquad + \frac{1}{\eta} \int_0^t e^{2 \alpha s} \Xi_y(X_s^\eps, Y_s^\eta) f(Y_s^\eta) \thinspace  ds    + \frac{1}{\sqrt{\eta}}\int_0^t e^{2 \alpha s} \tau(Y_s^\eta) \Xi_y(X_s^\eps, Y_s^\eta) \thinspace  dW_s^2 \\
 & \qquad \qquad  + \frac{\eps}{2} \int_0^t  e^{2 \alpha s} \sigma^2(X_s^\eps) \Xi_{xx}(X_s^\eps, Y_s^\eta) \thinspace  ds + \frac{1}{2 \eta} \int_0^t  e^{2 \alpha s} \tau^2(Y_s^\eta) \Xi_{yy}(X_s^\eps, Y_s^\eta) \thinspace  ds.
\end{aligned}
\end{equation}
The definition of the operator $\mathscr{L}_y$ defined in \eqref{E:Infinitesimal-Gen} and a rearrangement of the above terms yield
\begin{equation}\label{E:E2}
\begin{aligned}
 \int_0^t  e^{2 \alpha s} & \mathscr{L}_y \Xi (X_s^\eps, Y_s^\eta) \thinspace ds   = \eta  e^{2 \alpha t} \Xi(X_t^\eps,Y_t^\eta) - \eta  \Xi(X_0^\eps,Y_0^\eta) - 2 \alpha \eta \int_0^t  e^{2 \alpha s} \Xi(X_s^\eps, Y_s^\eta) \thinspace ds \\
 & \qquad \qquad - \eta  \int_0^t  e^{2 \alpha s} \Xi_x(X_s^\eps, Y_s^\eta) c(X_s^\eps, Y_s^\eta) \thinspace ds  - \eta  \sqrt{\eps}\int_0^t  e^{2 \alpha s} \sigma(X_s^\eps) \Xi_x(X_s^\eps, Y_s^\eta) \thinspace dW_s^1\\
 & \qquad \qquad - \frac{\eta}{\sqrt{\eta}}\int_0^t e^{2 \alpha s} \tau(Y_s^\eta) \Xi_y(X_s^\eps, Y_s^\eta) \thinspace dW_s^2 - \frac{\eps \eta}{2} \int_0^t  e^{2 \alpha s} \sigma^2(X_s^\eps) \Xi_{xx}(X_s^\eps, Y_s^\eta) \thinspace ds,
\end{aligned}
\end{equation}
which completes the proof of the lemma.
\end{proof}

\subsection{Proof of Lemmas \ref{L:Int-1} through \ref{L:Int-16}}\label{S:Integral-Lemmas}
\begin{proof}[Proof of Lemma \ref{L:Int-1}]
Using simple integration techniques, we have
\begin{equation*}
\mathcal{J}_1(t;\eps,\eta) = K (\eps+ \eta + \sqrt{\eps \eta})^4  \int_{u=0}^t \int_{s=u}^t \int_{w=s}^t (t-w)  e^{-2 \alpha (t-u)} e^{-\alpha(t-w)} e^{-\alpha (t-s)} \thinspace dw \thinspace ds \thinspace du.
\end{equation*}
In the equation above, we utilize a change of variables, along with a change in the order of integration and the arrangement of the new variables: \( t-u = x \), \( t-s = y \), \( t-w = z \) and $0 \le z \le y \le x \le t$ to obtain
\begin{equation*}
\begin{aligned}
\mathcal{J}_1(t;\eps,\eta) & = K (\eps+ \eta + \sqrt{\eps \eta})^4  \int_{x=0}^t \int_{y=0}^x \int_{z=0}^y z  e^{-2 \alpha x} e^{-\alpha z} e^{-\alpha y} \thinspace dz \thinspace dy \thinspace dx \\
& = K (\eps+ \eta + \sqrt{\eps \eta})^4  \int_{x=0}^t \int_{z=0}^x \int_{y=z}^x e^{-\alpha y} (z e^{-\alpha z}) e^{-2 \alpha x} \thinspace dy \thinspace dz \thinspace dx \\
& \le K (\eps+ \eta + \sqrt{\eps \eta})^4  \int_{x=0}^t  \int_{z=0}^x z e^{-2 \alpha z} e^{-2 \alpha x} \thinspace dz \thinspace dx  \le K (\eps+ \eta + \sqrt{\eps \eta})^4.
\end{aligned}
\end{equation*}
This completes the proof of the lemma.
\end{proof}

\begin{proof}[Proof of Lemma \ref{L:Int-4}]
Using simple integration techniques and $0 < \eta \ll 1$, we have
\begin{equation*}
\begin{aligned}
\mathcal{J}_4(t;\eps,\eta) &  \le K \eta \thinspace (\eps+ \eta + \sqrt{\eps \eta})^2 \int_{u=0}^t \int_{s=u}^t \int_{w=s}^t e^{-2 \alpha (t-u)}  e^{- 2\alpha(t-w)} e^{-\frac{\lambda}{\eta} (w-s)} dw \thinspace ds \thinspace du.
\end{aligned}
\end{equation*}
In the equation above, we utilize a change of variables, along with a change in the order of integration and the arrangement of the new variables: \( t-u = x \), \( t-s = y \), \( t-w = z \) and $0 \le z \le y \le x \le t$ to obtain
\begin{equation*}
\begin{aligned}
\mathcal{J}_4(t;\eps,\eta) \le K \eta \thinspace (\eps+ \eta + \sqrt{\eps \eta})^2  \int_{x=0}^t \int_{y=0}^x \int_{z=0}^y   e^{-2\alpha x} e^{-2 \alpha z} e^{-\frac{\lambda}{\eta}(y-z)} \thinspace dz \thinspace dy \thinspace dx  = K \eta^2 \thinspace (\eps+ \eta + \sqrt{\eps \eta})^2,
\end{aligned}
\end{equation*}
which concludes the proof of the lemma.
\end{proof}
\begin{proof}[Proof of Lemma \ref{L:Int-8}]
Using simple integration techniques and $0 < \eta \ll 1$, we have
\begin{equation*}
\begin{aligned}
\mathcal{J}_8(t;\eps,\eta) &  = K (\eps+ \eta + \sqrt{\eps \eta})  \int_{u=0}^t \int_{s=u}^t \int_{w=s}^t e^{-\alpha (t-u)} e^{-\alpha (t-s)} e^{- \alpha(t-w)}  e^{-\frac{\lambda}{\eta} (w-s)} e^{-\frac{\lambda}{\eta} (s-u)} \times \\
& \qquad \qquad \qquad \qquad \qquad \qquad \qquad \qquad \qquad \qquad \qquad \qquad \qquad  \int_{r=w}^t e^{-\frac{\lambda}{\eta} (r-w)} dr \thinspace dw \thinspace ds \thinspace du \\
&  \le K \eta \thinspace (\eps+ \eta + \sqrt{\eps \eta})  \int_{u=0}^t \int_{s=u}^t \int_{w=s}^t e^{-\alpha (t-u)} e^{-\alpha (t-s)}  e^{-\alpha(t-w)} e^{-\frac{\lambda}{\eta} (w-s)} e^{-\frac{\lambda}{\eta} (s-u)} dw \thinspace ds \thinspace du.
\end{aligned}
\end{equation*}
In this inequality, we utilize a change of variables, along with a change in the order of integration and the arrangement of the new variables: \( t-u = x \), \( t-s = y \), \( t-w = z \) and $0 \le z \le y \le x \le t$ to obtain
\begin{equation*}
\begin{aligned}
\mathcal{J}_8(t;\eps,\eta) & \le K \eta \thinspace (\eps+ \eta + \sqrt{\eps \eta})  \int_{x=0}^t \int_{y=0}^x \int_{z=0}^y  e^{-\alpha x} e^{-\alpha y} e^{-\alpha z} e^{- \frac{\lambda}{\eta} (y-z)} e^{- \frac{\lambda}{\eta} (x-y)} dz \thinspace dy \thinspace dx \\
& = K \eta \thinspace (\eps+ \eta + \sqrt{\eps \eta}) \int_{x=0}^t \int_{y=0}^x \int_{z=0}^y  e^{-\alpha x} e^{-\alpha y} e^{-\alpha z} e^{- \frac{\lambda}{\eta} (x-z)}  dz \thinspace dy \thinspace dx \\
& = K \eta \thinspace  (\eps+ \eta + \sqrt{\eps \eta}) \int_{x=0}^t \int_{z=0}^x \int_{y=z}^x e^{-\alpha x} e^{-\alpha y} e^{-\alpha z} e^{- \frac{\lambda}{\eta} (x-z)}  dy \thinspace dz \thinspace dx  \\
& \le K \eta \thinspace (\eps+ \eta + \sqrt{\eps \eta}) \int_{x=0}^t \int_{z=0}^x e^{-\alpha x}  e^{-2\alpha z} e^{- \frac{\lambda}{\eta} (x-z)} dz \thinspace dx \\
& = K \eta \thinspace (\eps+ \eta + \sqrt{\eps \eta}) \int_{z=0}^t \int_{x=z}^t e^{-\left(\alpha + \frac{\lambda}{\eta} \right) x} e^{-2\alpha z} e^{\frac{\lambda}{\eta}z} \thinspace dx \thinspace dz \le K \eta^2 \thinspace (\eps+ \eta + \sqrt{\eps \eta}).
\end{aligned}
\end{equation*}
This completes the proof of the lemma.
\end{proof}

\begin{proof}[Proof of Lemma \ref{L:Int-9}]
The proof of the result is similar to the proof of Lemmas \ref{L:Int-4} and \ref{L:Int-8}, hence for the sake of brevity, we omit the complete details.
\end{proof}

\begin{proof}[Proof of Lemma \ref{L:Int-16}]
We start using a change of variables, along with a change in the order of integration and the arrangement of the new variables: \( t-u = x \), \( t-s = y \), \( t-w = z \), $t-r = v$ and $0 \le v \le  z \le y \le x \le t$ to obtain
\begin{equation*}
\begin{aligned}
\mathcal{J}_{16}(t;\eps,\eta) & = K \int_{x=0}^t \int_{y=0}^x \int_{z=0}^y \int_{v=0}^z  e^{- 2 \alpha v} e^{- \alpha y}  e^{- \alpha z} e^{-\frac{\lambda}{\eta}(x-v)}  e^{-\frac{\lambda}{\eta}(x-y)} e^{-\frac{\lambda}{\eta}(z-v)} e^{-\frac{\lambda}{\eta}(y-z)} \thinspace dv \thinspace dz \thinspace dy \thinspace dx \\
& = K \int_{x=0}^t \int_{y=0}^x \int_{z=0}^y \int_{v=0}^z  e^{- 2 \alpha v} e^{- \alpha y}  e^{- \alpha z}  e^{-\frac{2\lambda}{\eta}(x-v)}  \thinspace dv \thinspace dz \thinspace dy \thinspace dx \\
& = K  \int_{x=0}^t \int_{y=0}^x \int_{z=0}^y e^{- \alpha y}  e^{- \alpha z} e^{-\frac{2\lambda}{\eta}x} \int_{v=0}^z e^{2 \left(\frac{\lambda}{\eta}-\alpha\right)v} \thinspace dv \thinspace dz \thinspace dy \thinspace dx.
\end{aligned}
\end{equation*}
Next, using $\frac{\lambda}{\eta}-\alpha>0$ for $\eta \ll 1$ and $\lambda > \lambda - \alpha \eta > \lambda - \alpha \eta_0 $ for some $\eta_0$, we have
\begin{equation*}
\begin{aligned}
\mathcal{J}_{16}(t;\eps,\eta) & \le K \frac{\eta}{\lambda - \alpha \eta_0} \int_{x=0}^t \int_{y=0}^x \int_{z=0}^y e^{- \alpha y}  e^{- \alpha z} e^{-\frac{2\lambda}{\eta}x} e^{2 \left(\frac{\lambda}{\eta}-\alpha\right)z} \thinspace dz \thinspace dy \thinspace dx \\
& = K \eta \int_{x=0}^t \int_{y=0}^x e^{- \alpha y} e^{-\frac{2\lambda}{\eta}x}  \int_{z=0}^y e^{- \alpha z} e^{2 \left(\frac{\lambda}{\eta}-\alpha\right)z} \thinspace dz \thinspace dy \thinspace dx.
\end{aligned}
\end{equation*}
In the equation above, for the integral $\int_{z=0}^y e^{- \alpha z} e^{2 \left(\frac{\lambda}{\eta}-\alpha\right)z} \thinspace dz$, using integration by parts, we have
\begin{equation*}
\begin{aligned}
\int_{z=0}^y e^{- \alpha z} e^{2 \left(\frac{\lambda}{\eta}-\alpha\right)z} \thinspace dz & \le  \eta \thinspace e^{- \alpha y} e^{2 \left(\frac{\lambda}{\eta}-\alpha\right)y} + K \eta \int_{z=0}^y e^{- \alpha z } e^{2 \left(\frac{\lambda}{\eta}-\alpha\right)z} \thinspace dz\\
& \le \eta \thinspace e^{- \alpha y} e^{2 \left(\frac{\lambda}{\eta}-\alpha\right)y} + K \eta \int_{z=0}^y e^{2 \left(\frac{\lambda}{\eta}-\alpha\right)z} \thinspace dz \le \eta \thinspace e^{- \alpha y} e^{2 \left(\frac{\lambda}{\eta}-\alpha\right)y} + K \eta^2  e^{2 \left(\frac{\lambda}{\eta}-\alpha\right)y}.
\end{aligned}
\end{equation*}
Putting the above expressions together and changing the order of integration again, we have
\begin{equation*}
\begin{aligned}
\mathcal{J}_{16}(t;\eps,\eta) & \le K \eta^2 \int_{x=0}^t \int_{y=0}^x e^{- 2 \alpha y} e^{-\frac{2\lambda}{\eta}x} e^{2 \left(\frac{\lambda}{\eta}-\alpha\right)y} \thinspace dy \thinspace dx + K \eta^3 \int_{x=0}^t \int_{y=0}^x e^{- \alpha y} e^{-\frac{2\lambda}{\eta}x} e^{2 \left(\frac{\lambda}{\eta}-\alpha\right)y} \thinspace dy \thinspace dx \\
& = K \eta^2 \left[ \int_{y=0}^t \int_{x=y}^t e^{- 2 \alpha y} e^{-\frac{2\lambda}{\eta}x} e^{2 \left(\frac{\lambda}{\eta}-\alpha\right)y} \thinspace dy \thinspace dx + \int_{y=0}^t \int_{x=y}^t e^{- \alpha y} e^{-\frac{2\lambda}{\eta}x} e^{2 \left(\frac{\lambda}{\eta}-\alpha\right)y} \thinspace dy \thinspace dx \right] \le K \eta^3,
\end{aligned}
\end{equation*}
which concludes the proof of the lemma.
\end{proof}

\subsection{Poisson Equation}\label{S:Poisson-Equation}
In this section, we present some preliminaries related to the Poisson equation in the whole space from \cite{pardoux2003poisson} that we frequently use in this manuscript. 

Assume that Conditions \ref{A:Boundedness of sigma-tau} and \ref{A:Derivatives-Assumption-c} are satisfied and let for any $x\in \BR$, $\mu_x$ represent the invariant measure of the process $Y^\eta$. We also assume that for the function $\mathsf{H} \in \mathscr{C}^{2, \alpha}(\BR,\mathscr{Y}),$ where $\mathscr{Y} \subseteq \BR,$
\begin{equation}\label{E:Centring-condition}
\int_{\mathscr{Y}} \mathsf{H} (x, y) \mu_x(dy)=0,
\end{equation}
and that, for some positive constants $K$, $\mathsf{p}_1$, $\mathsf{p}_2$, $\mathsf{p}_3$ and $q$,
\begin{equation}\label{E:Poisson-eq-H-growth}
\begin{aligned}
|\mathsf{H}(x,y)|  \le K (1+ |x|^{\mathsf{p}_1})(1+ |y|^q), \quad
|\mathsf{H}_{x}(x,y)| & \le K (1+ |x|^{\mathsf{p}_2})(1+ |y|^q), \\
|\mathsf{H}_{xx}(x,y)| & \le K (1+ |x|^{\mathsf{p}_3})(1+ |y|^q).
\end{aligned}
\end{equation}
Let $\mathscr{L}_y$ be the infinitesimal generator for the process $Y^\eta$. Then, the Poisson equation
$\mathscr{L}_y \mathsf{v}(x,y) =\mathsf{H} (x, y), \thinspace \int_{\mathscr{Y}} \mathsf{v}(x,y) \mu_x(dy)=0 $
has a unique solution that satisfies $\mathsf{v}(\cdot, y) \in \mathscr{C}^2$ for every $y\in \mathscr{Y},$ $\mathsf{v}_{xx} \in \mathscr{C}( \BR \times \mathscr{Y}),$ and there exist positive constants $\mathsf{K}$ and $m$ such that
\begin{equation}\label{E:Poisson-eq-Sol-growth}
\begin{aligned}
|\mathsf{v}(x,y)| + |\mathsf{v}_y(x,y)|  \le \mathsf{K} (1+ |x|^{\mathsf{p}_1})(1+ |y|^m), \thinspace
|\mathsf{v}_{x}(x,y)| + |\mathsf{v}_{xy}(x,y)|  & \le \mathsf{K} (1+ |x|^{\mathsf{p}_2})(1+ |y|^m), \\
|\mathsf{v}_{xx}(x,y)| + |\mathsf{v}_{xxy}(x,y)| & \le \mathsf{K} (1+ |x|^{\mathsf{p}_3})(1+ |y|^m).
\end{aligned}
\end{equation}
We note that in the growth of the functions $\mathsf{H}$ and $\mathsf{v}$ (as well as in their derivatives), the exponents of the variable $x$ are the same, while the powers of the $y$ variables may differ. In our analysis, the function \(\mathsf{H}(x, y)\) always satisfies the centering condition \eqref{E:Centring-condition} and can take one of the following forms: \(c(x, y)- \bar{c}(x)\), or \(q(x, y) - \bar{q}(x)\), where the function \(q(x, y)\) is defined in equation \eqref{E:q-function}. Since these functions meet the required conditions \eqref{E:Centring-condition}, and \eqref{E:Poisson-eq-H-growth}, their associated Poisson equation will also satisfy the growth conditions specified in equation \eqref{E:Poisson-eq-Sol-growth}.

\subsection{Malliavin derivatives}\label{S:MD}
\noindent
In this section, we introduce some preliminaries from Malliavin calculus. For the complete details of this topic, we refer to
\cite{nualart2006malliavin}.

Broadly speaking, the Malliavin derivative of a random variable is the Fr{\'e}chet derivative taken along directions in the Wiener space. It measures the sensitivity of the random variable to infinitesimal perturbations of the underlying Brownian path at time $t$. To make it precise,
let $\mathscr{H} \triangleq L^2
 \left(\mathbb{R}_+\right)$ and consider the isonormal Gaussian
 process $\left\{ W_h
   \colon h \in \mathscr{H} \right\}$
and the set of smooth cylindrical random variables
of the form $F = f
\left( W_{\varphi_1}, \cdots , W_{\varphi_n} \right)$, where $n \geq 1$, $\{\varphi_i\}^n_{i=1} \subset
\mathscr{H}$, and $f \in C_b^{\infty} \left(
  \mathbb{R}^n \right)$. The Malliavin derivative of a smooth cylindrical random variable
$F$ is defined as the $\mathscr{H}$-valued random variable given by
\begin{equation*}
DF = \sum_{i=1}^n  \frac{\partial f}{\partial x_i} \left( W_{\varphi_1},
  \cdots, W_{\varphi_n} \right)\varphi_i.
\end{equation*}
The derivative operator $D$ is closable from $L^2(\Omega)$
into $L^2(\Omega ; \mathscr{H})$, and we continue to denote by $D$
its closure, the domain of which we denote by $\mathbb{D}^{1,2}$, and which is a Hilbert space in the Sobolev-type norm
$\left\lVert F \right\rVert_{1,2}^2 = E (F^2) + E \left( \left\lVert DF \right\rVert_{\mathscr{H}}^2 \right).$
Similarly, one can obtain a derivative operator $D
\colon\mathbb{D}^{1, 2}(\mathscr{H}) \to L^2(\Omega; \mathscr{H}
\otimes \mathscr{H})$ as the closure of $D \colon L^2(\Omega;
\mathscr{H}) \to L^2(\Omega; \mathscr{H} \otimes \mathscr{H})$. We
then set $D^2F = D(DF)$. Note that more generally with $p > 1$, one can
analogously obtain $\mathbb{D}^{1, p}$ as Banach spaces of Sobolev
type by working with $L^p(\Omega)$.

\subsection*{Acknowledgments}
The author is grateful to Prof. Konstantinos Spiliopoulos for introducing the problem and for several valuable discussions. This work is supported by {\sc nsf} {\sc dms}-2311500.



\bibliographystyle{alpha}
\bibliography{References}

@article{bourguin2021typical,
  title={Typical dynamics and fluctuation analysis of slow--fast systems driven by fractional {B}rownian motion},
  author={Bourguin, Solesne and Gailus, Siragan and Spiliopoulos, Konstantinos},
  journal={Stochastics and Dynamics},
  volume={21},
  number={07},
  pages={2150030},
  year={2021},
  publisher={World Scientific}
}

@article{goddard2023study,
  title={On the study of slow--fast dynamics, when the fast process has multiple invariant measures},
  author={Goddard, Benjamin D. and Ottobre, Michela and Painter, Kevin John and Souttar, Iain},
  journal={Proceedings of the Royal Society A},
  volume={479},
  number={2278},
  pages={20230322},
  year={2023},
  publisher={The Royal Society}
}

@article{freidlin1999comparison,
  title={A comparison of homogenization and large deviations, with applications to wavefront propagation},
  author={Freidlin, Mark I. and Sowers, Richard B.},
  journal={Stochastic processes and their applications},
  volume={82},
  number={1},
  pages={23--52},
  year={1999},
  publisher={Elsevier}
}

@article{freidlin2021averaging,
  title={Averaging in the case of multiple invariant measures for the fast system},
  author={Freidlin, Mark I. and Koralov, Leonid},
  journal={Electronic Journal of Probability},
  volume={26},
  pages={1--17},
  year={2021},
  publisher={The Institute of Mathematical Statistics and the Bernoulli Society}
}

@article{budhiraja2024large,
  title={Large deviations for small noise diffusions over long time},
  author={Budhiraja, Amarjit and Zoubouloglou, Pavlos},
  journal={Transactions of the American Mathematical Society, Series B},
  volume={11},
  number={01},
  pages={1--63},
  year={2024}
}

@article{rockner2021averaging,
  title={Averaging Principle and Normal Deviations for Multiscale Stochastic Systems},
  author={R{\"o}ckner, Michael and Xie, Longjie},
  journal={Communications in Mathematical Physics},
  volume={383},
  number={3},
  pages={1889--1937},
  year={2021},
  publisher={Springer}
}

@article{rockner2021diffusion,
  title={Diffusion approximation for fully coupled stochastic differential equations},
  author={R{\"o}ckner, Michael and Xie, Longjie},
  journal={The Annals of Probability},
  volume={49},
  pages={1205-1236},
  year={2021}
}

@article{cerrai2009averaging,
  title={Averaging principle for a class of stochastic reaction--diffusion equations},
  author={Cerrai, Sandra and Freidlin, Mark I.},
  journal={Probability theory and related fields},
  volume={144},
  number={1},
  pages={137--177},
  year={2009},
  publisher={Springer}
}

@article{Khasminskii68,
  title={On the principle of averaging the {I}t\^o stochastic differential equations},
  author={Khasminskii, R.Z.},
  journal={Kybernetika},
  number={4},
  pages={260--279},
  year={1968}
}

@article{bour_spilio_2025,
  title={Quantitative fluctuation analysis of multiscale diffusion systems via {M}alliavin calculus},
  author={Bourguin, Solesne and Spiliopoulos, Konstantinos},
  journal={Stochastic Processes and their Applications},
  volume={180},
  pages={104524},
  year={2025},
  publisher={Elsevier}
}

@article{spiliopoulos2014fluctuation,
  title={Fluctuation analysis and short time asymptotics for multiple scales diffusion processes},
  author={Spiliopoulos, Konstantinos},
  journal={Stochastics and Dynamics},
  volume={14},
  number = {3},
  pages={1350026},
  year={2014}
}

@article{Dobson-Crisan,
  title={{P}oisson equations with locally {L}ipschitz coefficients and uniform-in-time averaging for stochastic differential equations via strong exponential
stability},
  author={Crisan, Dan and Dobson, Paul and Goddard, Ben and Ottobre, Michela and Souttar, Iain},
  journal={Annales de l'Institut Henri Poincaré (B) Probabilités et Statistiques},
  year={2024}
}

@article{Shivam-Kostas,
  title={Uniform-in-time bounds for a stochastic hybrid system with fast periodic sampling and small white-noise},
  author={Shivam Singh Dhama and Konstantinos Spiliopoulos},
  journal={Stochastics and Dynamics},
  volume={25},
  number = {1},
  pages={2550006},
  year={2025}
}

@article{Dobson-Ceisan-21,
  title={Uniform-in-time estimates for the weak error of the {E}uler
method for {SDE}s and a pathwise approach to derivative estimates for diffusion semigroups},
  author={Crisan,  Dan and  Dobson, Paul and  Ottobre, Michela},
  journal={Transactions of the American Mathematical Society},
  volume={374},
  number = {5},
  pages={3289-3330},
  year={2021}
}

@article{gailus2017statistical,
  title={Statistical inference for perturbed multiscale dynamical systems},
  author={Gailus, Siragan and Spiliopoulos, Konstantinos},
  journal={Stochastic Processes and their Applications},
  volume={127},
  number={2},
  pages={419--448},
  year={2017},
  publisher={Elsevier}
}

@book{nualart2006malliavin,
  title={The Malliavin calculus and related topics},
  author={Nualart, David},
  volume={1995},
  year={2006},
  publisher={Springer}
}

@article{vido_2020,
  title={An improved second-order {P}oincar{\'e} inequality for functionals of {G}aussian fields},
  author={Vidotto, Anna},
  journal={Journal of Theoretical Probability},
  volume={33},
  number={1},
  pages={396--427},
  year={2020},
  publisher={Springer}
}

@article{pardoux2003poisson,
  title={On {P}oisson equation and diffusion approximation 2},
  author={Pardoux, E. and Veretennikov, A. Yu},
  journal={The Annals of Probability},
  volume={31},
  number={3},
  pages={1166--1192},
  year={2003},
  publisher={Institute of Mathematical Statistics}
}

@article{hairer2020averaging,
  title={AVERAGING DYNAMICS DRIVEN BY FRACTIONAL {B}ROWNIAN MOTION},
  author={Hairer, Martin and Li, Xue-Mei},
  journal={The Annals of Probability},
  volume={48},
  number={4},
  pages={1826--1860},
  year={2020}
}

@article{gerolla2025fluctuations,
  title={Fluctuations of stochastic {PDE}s with long-range correlations},
  author={Gerolla, Luca and Hairer, Martin and Li, Xue-Mei},
  journal={The Annals of Applied Probability},
  volume={35},
  number={2},
  pages={1198--1232},
  year={2025},
  publisher={Institute of Mathematical Statistics}
}

@article{huang2020central,
  title={A central limit theorem for the stochastic heat equation},
  author={Huang, Jingyu and Nualart, David and Viitasaari, Lauri},
  journal={Stochastic Processes and Their Applications},
  volume={130},
  number={12},
  pages={7170--7184},
  year={2020},
  publisher={Elsevier}
}

@article{nualart2022central,
  title={Central limit theorems for stochastic wave equations in dimensions one and two},
  author={Nualart, David and Zheng, Guangqu},
  journal={Stochastics and Partial Differential Equations: Analysis and Computations},
  volume={10},
  number={2},
  pages={392--418},
  year={2022},
  publisher={Springer}
}

@article{budhiraja2016moderate,
  title={Moderate deviation principles for stochastic differential equations with jumps},
  author={Budhiraja, Amarjit and Dupuis, Paul and Ganguly, Arnab},
  journal={The Annals of Probability},
  volume={44},
  number={3},
  year={2016}
}

@article{bourguin2026uniform,
  title={Uniform-in-time quantitative fluctuations of large scale interacting particle systems},
  author={Bourguin, Solesne and Spiliopoulos, Konstantinos},
  journal={arXiv:2605.03057},
  year={2026}
}

@article{sharrock2026efficient,
  title={Efficient Online Learning in Interacting Particle Systems},
  author={Sharrock, Louis and Kantas, Nikolas and Pavliotis, Grigorios A.},
  journal={arXiv:2602.20875},
  year={2026}
}

@incollection{kifer2001averaging,
  title={Averaging and climate models},
  author={Kifer, Yuri},
  booktitle={Stochastic climate models},
  pages={171--188},
  year={2001},
  publisher={Springer}
}

@book{pavliotis2008multiscale,
  title={Multiscale methods: averaging and homogenization},
  author={Pavliotis, Grigorios A. and Stuart, Andrew M.},
  year={2008},
  publisher={Springer Science \& Business Media}
}

@book{KS91,
	author = {Ioannis Karatzas and Steven Shreve},
	edition = {second},
	publisher = {Springer-Verlag New York},
	series = {Graduate Texts in Mathematics},
	title = {Brownian Motion and Stochastic Calculus},
	volume = {113},
	year = {1991}
	}
\end{document}